\documentclass[11pt,letterpaper]{amsart}
\usepackage{amsaddr}
\usepackage{amsmath}
\usepackage{xurl}
\usepackage{amsxtra}
\usepackage{amstext}
\usepackage{amssymb}
\usepackage{amsthm}
\usepackage{color}
\usepackage{latexsym}
\usepackage{verbatim}
\usepackage{tabls}
\usepackage{graphicx}
\usepackage{caption}
\usepackage[dvipsnames]{xcolor}
\usepackage[hidelinks]{hyperref}
\usepackage{microtype}

\makeatletter
\renewcommand{\@biblabel}[1]{[#1]}
\makeatother

\usepackage[
    backend=biber,
    maxnames=99,
    minnames=99,
    maxcitenames=99, 
    mincitenames=99
]{biblatex}
\theoremstyle{plain}
\newtheorem{thm}{Theorem}[section]
\newtheorem{cor}[thm]{Corollary}
\newtheorem{conj}[thm]{Conjecture}
\newtheorem{lem}[thm]{Lemma}
\newtheorem{prop}[thm]{Proposition}

\theoremstyle{definition}

\numberwithin{equation}{section}
\def\R1{\widetilde{R}}
\def\T1{\widetilde{T}}
\def\dist{\operatorname{dist}}
\def\supp{\operatorname{supp}}
\def\Lip{\operatorname{Lip}}
\def\eps{\varepsilon}
\def\kap{\varkappa}
\def\R{\mathbb{R}}

\def\diam{\operatorname{diam}}

\def\dy{\mathbb{D}}

\def\wt{\widetilde}
\def\Lipodd{\Lip_0^{\operatorname{odd}}}

\def\VP{\mathcal{VP}}
\def\X{\mathbf{X}}
\def\Y{\mathbf{Y}}
\def\dist{\operatorname{dist}}

\def\0{\mathbf{0}}

\def\Z{\mathbf{Z}}

\def\HSDC{{\bf HSDC}}

\def\Lipeven{\Lip_0^{\operatorname{even}}}

\def\XXint#1#2#3{{\setbox0=\hbox{$#1{#2#3}{\int}$}
     \vcenter{\hbox{$#2#3$}}\kern-.5\wd0}}

\author{John Hoffman}
\address{%
{Department of Mathematics, University of the Basque Country (EHU), Leioa, Spain}\\
{johnloche.hoffman@ehu.eus}\\
https://orcid.org/0009-0006-3788-9677}

\author{Benjamin Jaye}
\address{%
School of Math, Georgia Tech, Atlanta, GA, USA\\
bjaye3@gatech.edu \\
https://orcid.org/0000-0001-6326-7392}

\title[On the boundedness of CZOs]
{On Singular Integrals and Quantitative Rectifiability in Parabolic Space and the Heisenberg Group}
\begin{document}

\begin{abstract}
David and Semmes proved that, for an Ahlfors regular measure in Euclidean space, boundedness of all Calder\'on--Zygmund singular integral operators of the same integer dimension as the measure implies uniform rectifiability.  We prove analogues of this result in parabolic space and in Heisenberg groups.  In general, however, the direct extension fails in both settings: there are Ahlfors regular product measures on which the relevant odd kernels cancel, but whose spatial or horizontal component has too large a dimension for these measures to be uniformly rectifiable.  We identify a natural additional Carleson packing condition which quantitatively excludes these product measures.  Under this condition the analogue of the David--Semmes implication holds in parabolic space and in Heisenberg groups.
\end{abstract}

\maketitle

\begingroup\linespread{0.9}\selectfont\tableofcontents\endgroup

\section{Introduction}

\subsection*{Overview}

A central problem in harmonic analysis and geometric measure theory is to understand the relationship between boundedness of singular integral operators on a measure and geometric regularity of that measure. In Euclidean space, a fundamental theorem of David and Semmes states that an Ahlfors-regular measure $\mu$ is uniformly rectifiable if and only if all smooth, odd Calder\'on--Zygmund singular integral operators of the appropriate dimension are bounded on $L^2(\mu)$. Uniform rectifiability is a condition that imposes Lipschitz regularity across locations and scales, and is ``uniform'' in the sense that the constants controlling this regularity are independent of location and scale. The relationship between uniform rectifiability and singular integrals has played an important role in the study of harmonic measure and elliptic boundary value problems.

This paper develops analogues of the David--Semmes theorem in parabolic space and Heisenberg groups. However, the direct analogue fails in parabolic space in every relevant codimension, and in Heisenberg groups in sufficiently high codimension. In both settings, we construct Ahlfors-regular product measures for which all admissible singular integral operators are bounded, but which are not uniformly rectifiable.

To recover a theorem in these settings, we introduce a higher spatial dimension condition (HSDC). It requires the locations and scales at which a measure closely resembles one of these counterexample product measures to satisfy a Carleson packing estimate (in other words, it is rare to encounter a scale that looks like one of these measures). In parabolic space, we prove that HSDC, together with boundedness of all admissible singular integrals, implies uniform rectifiability in every integer codimension. Since uniform rectifiability itself implies HSDC, this is a necessary and sufficient additional geometric condition among Ahlfors regular measures satisfying the singular integral boundedness hypothesis.

In Heisenberg groups, the need for this additional condition depends on the codimension of the measure. For $(k+2)$-Ahlfors regular measures in $\mathbb H^N$, boundedness alone implies the vertical $k$-plane square Carleson estimate when $N\leq k\leq 2N-1$. In the range $0\leq k<N$, the counterexamples from the parabolic setting occur, and we obtain the same conclusion under the corresponding Heisenberg-HSDC condition. In particular, the result in the first Heisenberg group strengthens Orponen's weak geometric lemma \cite[Theorem~1.4]{Orp} to a square Carleson estimate.

The two geometries are treated within a common group framework, which elucidates how the two geometries are remarkably similar from the perspective of singular integral operators and quantitative geometric measure theory.

\subsection*{Background and main results}

Understanding the geometry of a measure in $\R^n$ for which an associated singular integral operator is well behaved is a fundamental problem in harmonic analysis, with its origins in understanding removable sets for analytic capacity \cite{Gar, Tol2, Tol3, Dav1, DM, NTV1, NTV2, NTV3} and boundary value problems for elliptic PDE \cite{AMT, AHMTV, HM, HMM, HMU}.  David and Semmes introduced the notion of uniform rectifiability, and discovered that it was a natural condition from the standpoint of singular integral operators:

\begin{thm}[David--Semmes]\label{DSthm}
Suppose that $\mu$ is a $k$-Ahlfors regular measure (henceforth $k$-ADR)
on $\mathbb R^n$ for $k\in \{1,\dots, n-1\}$.  Consider kernels
$K\in C^{\infty}(\R^n\backslash \{0\})$ satisfying
\[
K(-X)=-K(X),
\qquad
|\nabla^j K(X)| \leq C(j)|X|^{-j-k}.
\]
The following are equivalent:
\begin{enumerate}
    \item Every associated convolution singular integral operator is
    bounded on $L^2(\mu)$.
    \item The measure $\mu$ is uniformly rectifiable.
\end{enumerate}
\end{thm}

David--Semmes introduced many characterizations of uniform rectifiability, with one being a Carleson condition for the Jones square function: there exists $C>0$ such that  
\begin{equation}\label{jonessquare}
    \int_0^R\int_{B_R(\X)} \beta_\mu(\Y,r)^2 \frac{d\mu(\Y) dr}{r} \leq CR^{k},\;\text{ for all }  \X \in \supp(\mu) \text{ and } R>0.
\end{equation}
 Here $\beta_\mu(\Y,r)$ measures the deviation of the measure $\mu$ from an appropriate $k$-dimensional plane which optimally approximates (in the least squares sense) $\supp(\mu)$ in the ball $B(Y, r)$.  In \cite{DS1} they prove the `only if' direction of Theorem \ref{DSthm} (from the boundedness of CZOs to the uniform rectifiability), and it is this direction that we will refer to as the David--Semmes theorem in this paper.  On the other hand, the `if' direction of Theorem \ref{DSthm} is a culmination of foundational work studying Calder\'on--Zygmund operators on Lipschitz graphs by Calder\'{o}n and Coifman--McIntosh--Meyer.

David and Semmes asked whether $L^2(\mu)$-boundedness of the $k$-Riesz transform\footnote{The CZO with kernel $K(x) = \frac{x}{|x|^{k+1}}$ for $x\in \R^n$.} alone implies uniform rectifiability.  This question has been answered affirmatively in the case $k=1$ by Mattila--Melnikov--Verdera \cite{MMV}, and $k=n-1$ by Nazarov--Tolsa--Volberg \cite{NToV}.  These results, and refinements of them, have played an essential role in recent breakthroughs in the study of harmonic measure and free boundary problems (described below).  The intermediate cases $k=2,\dots,n-2$ form the longstanding higher-codimension David--Semmes problem; a recent preprint \cite{OpenAIRiesz2026} claims a positive resolution.   The Nazarov--Tolsa--Volberg theorem \cite{NToV} has been extended to ``Riesz transforms'' associated with gradients of fundamental solutions to more general classes of elliptic equations.  For example, Prat, Puliatti, and Tolsa extended the Nazarov--Tolsa--Volberg result to elliptic equations associated to matrices with H\"older coefficients in \cite{PPT}.  The current state-of-the-art is due to Molero, Mourgoglou, Puliatti, and Tolsa in \cite{MMPT}, where they prove the Nazarov--Tolsa--Volberg result for elliptic equations associated to matrices with coefficients satisfying a Dini-mean oscillation condition.  Dabrowski and Tolsa have recently extended the Nazarov--Tolsa--Volberg theorem to measures with significantly weaker density hypotheses than ADR in an important work \cite{DT}.

Some very important machinery in the proof of the Nazarov--Tolsa--Volberg result was introduced by Eiderman, Nazarov, and Volberg in \cite{ENV}, where the authors show that the Riesz transform cannot be $L^2$-bounded with respect to a measure with poor lower-density properties.  The result of \cite{ENV} was later extended to ``Riesz transforms'' associated to elliptic equations with H\"older coefficients by Conde-Alonso, Mourgoglou, and Tolsa in \cite{CMT}.  In \cite{MMPT}, Molero, Mourgoglou, Puliatti, and Tolsa extended the result further to operators with DMO-type (Dini mean oscillation) coefficients.

In a recent survey on rectifiability, Mattila highlights the scarcity of results concerning the geometric consequences of the boundedness of singular integrals in \emph{parabolic space} and in the \emph{Heisenberg group} (see {Sections~10.5 and~10.6} of~\cite{Mat}). {For codimension-one $\operatorname{Lip}(1,1/2)$ graphs, Bortz, Hofmann, Martell, and Nystr\"om \cite[Proposition~B.2]{BHMN} proved that uniform $L^2$ bounds for the truncations of every smooth, spatially odd, homogeneous parabolic singular integral imply parabolic uniform rectifiability.} The goal of this paper is to take a step in this direction by developing David--Semmes-type theorems in these settings. We first state the parabolic result in arbitrary codimension. In parabolic space, the direct analogue of Theorem~\ref{DSthm} is \emph{false}: for each relevant spatial dimension $k$, there exist $(k+2)$-Ahlfors regular measures for which a large class of $(k+2)$-dimensional parabolic Calder\'on--Zygmund operators (CZOs) are $L^2$-bounded, yet the measures fail to be parabolically uniformly rectifiable. The ``large class of CZOs'' is defined precisely in Subsection~\ref{CZOdefs}; it is designed to retain the essential features of the \emph{parabolic Riesz transform}. These counter-example measures are highly concentrated in the spatial component. We recover a David--Semmes-type theorem in the parabolic setting by imposing, in addition to boundedness of singular integrals, a Carleson set condition called the {\bf HSDC} (\emph{higher spatial dimension condition}). In dimension $k$, this condition says that the scales at which the measure is close to these counter-examples form a Carleson family. We prove the following theorem for every $k\in\{0,\dots,n-1\}$; see Subsection~\ref{parastate} for the precise statement.

\begin{thm}\label{pthm}
    Fix $k\in\{0,\dots,n-1\}$. Suppose that $\mu$ is a $(k+2)$-Ahlfors regular measure on parabolic space $\mathbb R^{n+1}$ and there exists $\eps>0$ such that the $k$-dimensional $\eps$-{\bf HSDC} condition is satisfied by $\mu$.  If, additionally, all $(k+2)$-dimensional CZOs map $L^2(\mu)$ to itself boundedly, then $\mu$ is a uniformly rectifiable measure.
\end{thm}

{Because uniform rectifiability implies \HSDC{}, \HSDC{} is a sharp extra condition we can impose in addition to boundedness of SIOs to obtain uniform rectifiability. Indeed, suppose that $\mathbf A$ is another condition which, together with the boundedness hypothesis of Theorem \ref{pthm}, implies uniform rectifiability. Then any Ahlfors regular measure satisfying both boundedness and $\mathbf A$ satisfies \HSDC{}. Thus, under the SIO boundedness hypothesis, any such condition $\mathbf A$ must be equivalent to or stronger than \HSDC{}.}

The examples we construct show that the na\"ive translation of the Nazarov--Tolsa--Volberg theorem to the parabolic setting is false, while Theorem \ref{pthm} gives us hope that a sharp theorem can still be recovered through the {\bf HSDC} condition.  We conjecture the following.

\begin{conj}\label{parads}
Suppose that $\mu$ is a codimension-one Ahlfors regular measure on parabolic space $\mathbb R^{n+1}$.  If the parabolic Riesz transform is $L^2(\mu)$-bounded, and $\mu$ satisfies the $\eps$-{\bf HSDC} for some $\eps>0$, then $\mu$ is a uniformly rectifiable measure.
\end{conj}

In the case of a parabolic, codimension-one Ahlfors regular measure $\mu$, the universally agreed upon definition of uniform rectifiability is estimate {\eqref{flatsquare} with $k=n-1$}, with $\beta_\mu$ measuring the ``least squares'' distance to an optimal \emph{vertical hyperplane}, i.e. a codimension-one hyperplane containing a line parallel to the $t$-axis (see Subsection~\ref{URdefn}).  This definition of parabolic uniform rectifiability was introduced by Hofmann, Lewis, and Nystr\"om in \cite{HLN}.  To our knowledge, there has been no prior work about parabolic uniform rectifiability in codimension greater than one.  In the sequel we will define parabolic uniformly rectifiable measures on $\mathbb R^{n+1}$ of higher codimension as Ahlfors regular measures (with appropriate homogeneity) satisfying {\eqref{flatsquare}}, where the $\beta_\mu$ coefficients optimize the distance to a higher-codimensional vertical plane.

We next turn to the Heisenberg group.  Singular integral theory in this setting already has a substantial history.  Chousionis and Mattila studied Riesz-type kernels on Ahlfors--David regular subsets of Heisenberg groups, showing that boundedness imposes strong restrictions on the dimension and on approximation by homogeneous subgroups; they also used related singular integral methods for self-similar sets and removability for Lipschitz harmonic functions \cite{CMAD,CM}.  In codimension one, Chousionis, F\"assler, and Orponen proved $L^2$-boundedness of three-dimensional Calder\'on--Zygmund operators on compactly supported $C^{1,\alpha}$ intrinsic graphs, assuming uniform boundedness on vertical planes, and introduced vertical $\beta$-numbers for intrinsic graphs \cite{CFO1,CFO2}.

F\"assler and Orponen gave sufficient conditions for boundedness of the Heisenberg Riesz transform on intrinsic Lipschitz graphs in terms of vertical oscillation, and also proved boundedness results for odd and horizontally odd kernels on regular curves and Lipschitz flags \cite{FO,FOcurves}.  Chousionis and Li studied nonnegative kernels connected to $1$-rectifiability \cite{CL}, while Chousionis, Li, and Young proved strong geometric lemmas for intrinsic Lipschitz graphs in higher Heisenberg groups and showed that the Heisenberg Riesz transform can be unbounded on intrinsic Lipschitz graphs in $\mathbb H^1$ \cite{CLY1,CLY2,CLY3}.  We also mention the necessary condition for boundedness of the Heisenberg Riesz transform due to Dabrowski and Villa \cite{DV}.

The closest precedent to this result in the first Heisenberg group is Orponen's note on the local symmetry condition \cite{Orp}.  Under the same $3$-dimensional horizontally antisymmetric singular-integral boundedness hypothesis in $\mathbb H^1$ as in the codimension-one case below, Orponen proved that a closed $3$-regular set satisfies a local symmetry condition, and that this condition implies the weak geometric lemma for vertical $\beta$-numbers.  Thus, in $\mathbb H^1$, the following theorem strengthens Orponen's conclusion from the weak geometric lemma to the square Carleson estimate for vertical $\beta$-numbers.  In higher Heisenberg groups, the same theorem gives all codimensions in which higher-horizontal-dimensional product measures cannot occur.  A vertical $k$-plane in $\mathbb H^N$ has the form $(a+W)\times\mathbb R$ with $\dim W=k$.

\begin{thm}\label{hthm}
Let $N\geq1$ and $N\leq k\leq 2N-1$.  Suppose that $\mu$ is a $(k+2)$-ADR measure on $\mathbb H^N$ and that every admissible $(k+2)$-dimensional Calder\'on--Zygmund singular integral operator is bounded on $L^2(\mu)$.  Then $\mu$ satisfies the vertical $k$-plane $\beta$-Carleson estimate; in particular, $\mu$ is uniformly rectifiable in the sense of Subsection~\ref{URdefn}.
\end{thm}

There is no need to introduce an analogue of the HSDC condition in Theorem \ref{hthm}.  However, in the range $0\leq k\leq N-1$, pathological measures similar to those found in the parabolic case exist, and in that range we impose the Heisenberg-HSDC condition introduced precisely in Subsection~\ref{heisenbergtheorems} to rule out the measure being close to one of these pathological measures in a Carleson packing sense.

\begin{thm}\label{hthmHSDC}
Let $N\geq1$ and $0\leq k\leq N-1$.  Suppose that $\mu$ is a $(k+2)$-ADR measure on $\mathbb H^N$ satisfying the $\eps$-Heisenberg-HSDC condition for some $\eps>0$.  If every admissible $(k+2)$-dimensional Calder\'on--Zygmund singular integral operator is bounded on $L^2(\mu)$, then $\mu$ satisfies the vertical $k$-plane $\beta$-Carleson estimate.  Equivalently, $\mu$ is uniformly rectifiable in the sense of Subsection~\ref{URdefn}.
\end{thm}

The Nazarov--Tolsa--Volberg theorem (\cite{NToV})  has had profound consequences for the study of free boundary problems involving elliptic measures, and, relatedly, connections between absolute continuity of elliptic measures and rectifiability (both qualitative and quantitative).  We hope that our results will be an important first step toward establishing analogs of the Nazarov--Tolsa--Volberg theorem in parabolic space and Heisenberg groups. We also think that the {\bf HSDC} condition should play an important role in the theory of caloric and parabolic measures and parabolic rectifiability.

In 1992 Bishop (\cite[Conjecture 10]{Bi}) conjectured that absolute continuity of harmonic measure with respect to surface measure should imply rectifiability of the boundary. Bishop's conjecture was solved in 2015 in \cite{AHMTV} where the authors effectively used estimates on the Green function to control the Riesz transform and verify the hypotheses of the codimension-one David--Semmes conjecture.  Later, in \cite{MT}, Mourgoglou and Tolsa quantified the main result of \cite{AHMTV} using extremely technical arguments involving a corona decomposition with respect to harmonic measure.  In \cite{GMT}, Garnett, Mourgoglou, and Tolsa show that, under the codimension-one Ahlfors-regularity condition, if one also assumes either ``Carleson measure estimates for harmonic functions'' or ``$\varepsilon$-approximability''---two analytic conditions closely related to the absolute continuity of harmonic measure--- then one can conclude uniform rectifiability by establishing that the associated Riesz transforms are $L^2$-bounded.

``Two-phase'' free boundary problems for harmonic (or, more generally, elliptic) measures are problems where one assumes that two harmonic measures in distinct connected components of a domain enjoy some mutual absolute continuity, in order to conclude that the boundary of the domain enjoys some regularity.  In \cite{Bi}, Bishop also conjectured that mutual absolute continuity of harmonic measures in separated connected components of a domain should imply rectifiability of the common boundary of those components.  Two-phase problems are more delicate than one-phase problems, but similar (and more difficult) arguments can be used to tackle them.  In particular, Girela-Sarri\'on and Tolsa proved a more delicate version of the codimension-one David--Semmes problem in \cite{GT} which is suitable for applications to two-phase problems.  The main result of \cite{GT} was later used in solutions to qualitative and quantitative two-phase problems for harmonic measure; we refer the reader to \cite{AMT}, \cite{AMT2}, and \cite{AMTV} (the first two authored by Azzam, Mourgoglou, and Tolsa, and the last authored by Azzam, Mourgoglou, Tolsa, and Volberg).  In particular, \cite{AMTV} resolves the ``two-phase'' conjecture of Bishop (\cite{Bi}).  Very recently, Merlo, Mourgoglou, and Puliatti (\cite{MMP}) have extended many of the results above (one-phase problems, two-phase problems, and the result of Girela-Sarri\'on and Tolsa) to elliptic measures with coefficients satisfying a ``Dini-mean oscillation'' condition.  We also record that, prior to \cite{MMP}, Puliatti (\cite{P}) had extended the result of Girela-Sarri\'on and Tolsa to elliptic equations associated to matrices with H\"older coefficients.  

In the parabolic case, a converse to Theorem \ref{pthm} is known {in codimension one}.  In \cite{BHHLN1,BHHLN2,BHHLN3} Bortz, Hofmann, Luna-Garc\'ia and Nystr\"om, together with the first author of this paper, developed several foundational tools for the theory of parabolic quantitative rectifiability.  {In forthcoming work, Bortz and the first author extend the main theorem of \cite[Theorem~5]{BHHLN2}, via minor modifications to the proof in \cite{BHHLN2}, to include regular $\operatorname{Lip}(1,1/2)$ graphs of dimension $k+2$ for every $n\ge1$ and $0\le k\le n-1$, including the endpoint $k=0$. In codimension one, the corona decomposition theorem of \cite{BHHLN3} applies for every $n\ge1$. Combining the extended version of \cite[Theorem~5]{BHHLN2}, which includes the cases $k=0$ and $n=1$, with the results of \cite{BHHLN3,BHHLN1} gives the following theorem.}  

\begin{thm}\label{parabolicCZOurconverse}
    {Let $n\geq1$. Suppose that $\mu$ is an $(n+1)$-Ahlfors regular, parabolically uniformly rectifiable measure on $\mathbb R^{n+1}$. Then all admissible $(n+1)$-dimensional CZOs are bounded on $L^2(\mu)$.}
\end{thm}

Combining Theorems \ref{pthm} and \ref{parabolicCZOurconverse} with the necessity of HSDC for uniformly rectifiable measures {(Lemma~\ref{HSDCnecessary})} gives the following formulation.

\begin{cor}\label{paraboliciffcor}
{Let $n\geq1$.} Let $\mu$ be an $(n+1)$-ADR measure on parabolic space $\mathbb R^{n+1}$.  Then $\mu$ is uniformly rectifiable if and only if
\begin{enumerate}
    \item all {admissible} $(n+1)$-dimensional CZOs are bounded on $L^2(\mu)$, and
    \item the $\eps$-HSDC condition holds for some $\eps>0$.
\end{enumerate}
\end{cor}

The theory of one- and two-phase free boundary problems for \emph{caloric} measure is very undeveloped compared to its elliptic counterpart.  Caloric measure is defined in a completely analogous way to harmonic measure, but the underlying PDE is the heat equation instead of Laplace's equation.  {Mattila \cite{Mat} has pointed out that understanding the geometric consequences of boundedness of parabolic singular integrals is an important step toward understanding the relationship between absolute continuity of caloric measure and parabolic rectifiability.}  Recently, Bortz, Hofmann, Martell, and Nystr\"om have produced two important results regarding one-phase problems for caloric measure.  

\begin{thm}[\cite{BHMN}]
Suppose that $\Gamma = \{(x,A(x,t),t)\} \subset \mathbb R^{n+1}$ is a codimension-one parabolic Lipschitz graph, and let $\Omega := \{(x,x_n,t): x_n>A(x,t)\}$.  If caloric measure satisfies the weak-$A_\infty$ property with respect to parabolic surface measure on $\Gamma$, then $\Gamma$ is a parabolic uniformly rectifiable set.
\end{thm}

We remark that graphs of parabolic Lipschitz functions need not be parabolic uniformly rectifiable (see \cite{BHHLN1} for a counter-example).   This is in contrast to the Euclidean setting, where Lipschitz graphs are a prototypical example of a uniformly rectifiable set.  Very recently, the authors have extended their result to the case when $\Gamma$ is a \emph{time-symmetric} ADR set.\footnote{In codimension one, either time-backward or time-forward ADR implies the $\HSDC$; see Lemma~\ref{onesidedHSDC}.}  The time-symmetric ADR condition is a condition that forces every surface ball to have ample intersection with the set in the forward and backward time directions.  
\begin{thm}[\cite{BHMN2}]
Suppose that $\Omega \subseteq \mathbb R^{n+1}$ is an open subset of parabolic space, that $\Sigma = \partial \Omega$ is time-symmetric ADR, and that $\Omega$ satisfies a uniform interior corkscrew condition.  If caloric measure in $\Omega$ satisfies the weak-$A_\infty$ condition with respect to parabolic surface measure on $\partial \Omega$, then $\Sigma$ is a parabolic uniformly rectifiable set. 
\end{thm}
The proofs of the two theorems above do not rely on analysis of the Riesz transform.

In the past decade there have been substantial results concerning quantitative rectifiability and Heisenberg groups.  In \cite{NY1}, Naor and Young resolved several open problems by proving a discrete ``vertical versus horizontal'' isoperimetric inequality in a discrete Heisenberg group $\mathbb H^n_\mathbb Z$, $n\geq 2$.  In order to prove this inequality, they prove a corresponding inequality in the continuous Heisenberg group $\mathbb H^n$, and use the continuous case to deduce the discrete case.  The methods of their proof utilize many tools from the area of quantitative rectifiability, in particular the ``corona decomposition'' type arguments introduced by David and Semmes in \cite{DS1}.  This discrete isoperimetric inequality is used to find an optimal (up to constants) lower bound on the integrality gap of the Goemans--Linial semidefinite program for the Sparsest Cut problem, and also to show that any embedding {into $L_1$} of the ball of radius $r$ centered at the origin in $(\mathbb H_\mathbb Z^2,d_W)$ (the second discrete Heisenberg group equipped with the word metric) has distortion at least {$C\sqrt{\log r}$}. 
In {\cite{NY2}}, Naor and Young prove a ``vertical-versus-horizontal'' Poincar\'e inequality in the first Heisenberg group.  As a corollary, they obtain, in the first Heisenberg group, bounds on the distortion of embeddings of balls centered at the origin.  In this work they develop a new multiscale decomposition for intrinsic Lipschitz graphs, which was later modified in the aforementioned works \cite{CLY1,CLY2}.

There have also been significant advances in extending Preiss' theorem to Heisenberg groups and parabolic space.  These have been developed in the codimension-one setting by Merlo (\cite{M}), and Merlo, Mourgoglou, and Puliatti ({\cite{MMPparabolic}}), respectively.

The first half of this paper develops a cylindrical blow-up argument originating in the work of Nazarov, Tolsa and the second author \cite{JNT}, building on ideas from \cite{JN1,JN2,JNRT}. This technique requires substantial modifications to work in the parabolic and Heisenberg settings. A version of this technique also appears in OpenAI's recent claimed solution of the David--Semmes conjecture in higher codimension \cite{OpenAIRiesz2026}.

\subsection*{Acknowledgements}
{J.H. was supported in part by grant EUR2025-165042, funded by MICIU/\allowbreak AEI/\allowbreak 10.13039/\allowbreak 501100011033, and by grant PID2024-157724NB-I00, funded by MICIU/\allowbreak AEI/\allowbreak 10.13039/\allowbreak 501100011033 and by ERDF/EU.}

Research supported in part by NSF grants DMS-2453251 and DMS-2049477 to B.J.  This research was undertaken in part while B.J. was a Simons Fellow.

\section{The common framework}

\subsection{The group, Haar measure, and metric}

Fix a skew-symmetric bilinear form satisfying
\[
\eta:\R^n\times\R^n\longrightarrow\R,
\qquad
\eta(X,Y)=-\eta(Y,X),
\qquad
|\eta(X,Y)|\leq \frac12|X||Y|,
\tag{$\eta$}
\label{etahypothesis}
\]
for all $X,Y\in\R^n$.  In particular, $\eta(X,X)=0$ for every
$X\in\R^n$.
Let
\[
\mathbb G_\eta=(\R^n\times\R,\cdot),
\qquad
(X,t)\cdot(Y,s)=(X+Y,t+s+\eta(X,Y)).
\]

The identity is $\0=(0,0)$ and, since $\eta(X,X)=0$,
$(X,t)^{-1}=(-X,-t)$.  {For fixed $p=(a,\tau)$, consider} the left and right translation maps
\[
(Y,s)\longmapsto (a+Y,\tau+s+\eta(a,Y)),
\qquad
(Y,s)\longmapsto (Y+a,s+\tau+\eta(Y,a)).
\]
The derivative matrices of these maps are block triangular, with diagonal blocks
$I_n$ and $1$, and hence have determinant one.  Thus, the $n+1$-dimensional Lebesgue measure
$m_{n+1}$ is invariant under both left and right translation,
so it is both a left and a right Haar measure on $\mathbb G_\eta$.
We define the parabolic dilations
\[
\delta_r(X,t)=(rX,r^2t),\qquad r>0.
\]

The parabolic group is the case $\eta=0$.  The Heisenberg group
$\mathbb H^N$, with $n=2N$, is the case $\eta=\frac12\omega$, where
\[
\omega((x,y),(x',y'))=x\cdot y'-y\cdot x'.
\]

We use the Kor\'{a}nyi norm and the associated left-invariant distance
\[
\|(X,t)\|=(|X|^4+16t^2)^{1/4},
\qquad
d(\X,\Y)=\|\Y^{-1}\cdot\X\|.
\]
Explicitly, if $\X=(X,t)$ and $\Y=(Y,s)$, then
\begin{equation}\label{explicitetadistance}
d(\X,\Y)
=\left(|X-Y|^4+16\bigl(t-s+\eta(-Y,X)\bigr)^2\right)^{1/4}.
\end{equation}

Under \eqref{etahypothesis}, $d$ is a metric.  To verify the triangle inequality, write
\[
\eta(X,Y)=\langle AX,Y\rangle
\]
for the skew-symmetric linear map $A$ representing $\eta$.  Skew-symmetry of $A$ yields that
$AX\perp X$, while \eqref{etahypothesis} gives $\|A\|\leq\frac12$.
Consequently,
\begin{equation}\label{etacompatibility}
\langle X,Y\rangle^2+4\eta(X,Y)^2
\leq |X|^2|Y|^2.
\end{equation}
Indeed, $X$ and $2AX$ are orthogonal and $|2AX|\leq|X|$, so, if
$AX\neq0$, Bessel's inequality yields
\[
\langle X,Y\rangle^2+4\langle AX,Y\rangle^2
\leq |X|^2\left(
\left\langle\frac{X}{|X|},Y\right\rangle^2
+\left\langle\frac{AX}{|AX|},Y\right\rangle^2
\right)
\leq |X|^2|Y|^2.
\]
The case $AX=0$ is immediate.  Now, with
$|(a,b)|_2=\sqrt{a^2+b^2}$ for $a,b\in\R$, write
\begin{align*}
\|(X,t)\cdot(Y,s)\|^2
={}&\left|\bigl(|X+Y|^2,4(t+s+\eta(X,Y))\bigr)\right|_2\\
={}&\left|
(|X|^2,4t)+(|Y|^2,4s)
+2\bigl(\langle X,Y\rangle,2\eta(X,Y)\bigr)
\right|_2\\
\leq{}&\|(X,t)\|^2+\|(Y,s)\|^2+2|X||Y|\\
\leq{}&\bigl(\|(X,t)\|+\|(Y,s)\|\bigr)^2,
\end{align*}
where the first inequality uses \eqref{etacompatibility}.  

Notice also that
\[
\|\delta_r\X\|=r\|\X\|
\]
and that the restriction to the time line is
\[
d_t(t,s):=d((0,t),(0,s))=2\sqrt{|t-s|}.
\]

The spatial reflection
\[
\overline{(X,t)}=(-X,t)
\]
is an isometric group automorphism, since
$\eta(-X,-Y)=\eta(X,Y)$ and $\|\overline\X\|=\|\X\|$.

\subsection{Vertical subgroups and isotropic subspaces}

For $p\in\mathbb G_\eta$, write $L_p(q)=p\cdot q$.  For
$0\leq m\leq n$, an affine vertical $m$-plane is a set
$(a+W)\times\R$, where $W\subset\R^n$ is an $m$-plane.  We denote the collection of these planes by
$\VP_m$.  When $m=0$, we interpret $\mathcal H^0\vert_W$ as a unit point mass.
Every linear vertical plane $P=W\times\R$ is a subgroup.  We write
$\mathcal H^m$ for Euclidean
$m$-dimensional Hausdorff measure in the spatial component; its restriction
to an $m$-plane is a fixed constant multiple of $m$-dimensional Lebesgue
measure on that plane.  If $m=\dim W$, the same calculation as
above, now on $W\times\R$, shows that
$\mathcal H^m\vert_W\times\mathcal L^1$ is both left and right Haar measure
on $P$.  Each left coset of $P$ is, as a set, $(a+W)\times\R$ for some
$a\in\R^n$.  If $P=V\times\R$, where $V\subset\mathbb R^n$ is affine, and
$X_P$ is the Euclidean projection of $X$ onto $V$, then
\begin{equation}\label{verticalplanedistance}
\dist((X,t),P)=|X-X_P|.
\end{equation}
Indeed, after fixing the horizontal comparison point $X_P$, its time
coordinate can be chosen to cancel the time term in
\eqref{explicitetadistance}.  Thus distance to a vertical plane is spatial distance.

We call a linear subspace $W\subset\R^n$ \emph{$\eta$-isotropic} (or just isotropic) if
\[
\eta(W,W)=\{0\},
\qquad
\eta(W,W):=\{\eta(w,z):w,z\in W\}.
\]
If $W$ is isotropic, then on $W\times\R$ group addition is given by parabolic addition (i.e. $(X,t)\cdot (Y,s) = (X+Y, t+s)$) and the restricted
distance is
\[
\bigl(|w-z|^4+16(t-s)^2\bigr)^{1/4}.
\]
By the left invariance of the metric, $L_p:W\times\R\to p\cdot(W\times\R)$ is an isometry
for every $p\in\mathbb G_\eta$.  

\begin{lem}\label{maxisotropic}
Let $I_\eta$ denote the maximal dimension of an $\eta$-isotropic
subspace of $\R^n$.  Then
\begin{equation}\label{maximalisotropicdimension}
I_\eta=n-\tfrac12\operatorname{rank}(\eta).
\end{equation}
Here $\operatorname{rank}(\eta):=\operatorname{rank}A$, where $A$ is
the skew-symmetric matrix representing $\eta$.
\end{lem}

\begin{proof} 
Set
$K=\ker A$ and $U=K^\perp$.  Skew-symmetry of $A$ gives $A(U)\subset U$,
and $A|_U$ has zero kernel.  Thus the restriction of $A$ to $U$ is invertible.

We claim that $\dim U$ is even.  To see this, choose
$e_1\ne0$ in $U$ and choose $f_1\in U$ so that
$\eta(e_1,f_1)=1$ (recall that $A|_U$ is invertible on $U$).  Now consider
\[
U_1=\{u\in U:\eta(u,e_1)=\eta(u,f_1)=0\}.
\]
The space $U_1$ has codimension two in $U$, and, with $P_{U_1}$ the
orthogonal projection onto $U_1$, the map
$P_{U_1}A|_{U_1}:U_1\to U_1$ is invertible.  Indeed,
\[
U=\operatorname{span}\{e_1,f_1\}\oplus U_1.
\]
If $u\in U_1$ satisfies $\langle A(u),v\rangle=0$ for every
$v\in U_1$, then the definition of $U_1$ shows that the same equality
holds for $v=e_1,f_1$.  Hence $\langle A(u),v\rangle=0$ for every
$v\in U$, and therefore $u=0$.
Iterating gives a basis $e_1,f_1,\ldots,e_r,f_r$ of $U$ satisfying
\[
\eta(e_i,e_j)=\eta(f_i,f_j)=0,
\qquad
\eta(e_i,f_j)=\delta_{ij}.
\]
In particular, $\operatorname{rank}\eta=2r$.

Now, if $W\subset\R^n$ is $\eta$-isotropic and $d=\dim W$, then
$A(W)\subset W^\perp$.  Consequently,
\[
d-(n-2r)=d-\dim(K)\leq d-\dim(W\cap K)=\dim A(W)\leq n-d,
\]
where we used $\dim K=n-2r$.  Thus $d\leq n-r$.  Conversely,
$K\oplus\operatorname{span}\{e_1,\ldots,e_r\}$ is isotropic and has
dimension $n-r$.
\end{proof}

For the standard Heisenberg form $\omega$ on $\R^{2N}$, the associated map is
$A(x,y)=(-y,x)$ and is invertible.  Consequently,
\[
\operatorname{rank}\omega
=\operatorname{rank}\left(\tfrac12\omega\right)=2N,
\qquad I_{\frac12\omega}=N.
\]

\subsection{Balls, coordinate cubes, and cylinders}\label{ballnotation}
An open ball takes the form
\[
B_r(\X)=\{\Y\in\R^{n+1}:\|\Y^{-1}\cdot\X\|<r\}.
\]
In particular,
\[
B_r(\0)=\{(Y,s):|Y|^4+16s^2<r^4\}.
\]
We also use the coordinate cylinders
\[
C_r(\X):=\{\Y=(Y,s): |X-Y|<r,\ |t-s|^{1/2}<r\}
\]
and coordinate cubes
\[
Q_r(\X):=\{\Y=(Y,s): \|X-Y\|_\infty<r,\ |t-s|^{1/2}<r\},
\]
where $\|X\|_\infty=\max_{1\leq j\leq n}|X_j|$.

We call the coordinate axis corresponding to the last component of $\mathbb R^{n+1}$ the $t$-axis.  We refer to $\mathbb R^n$ as space and $\mathbb R$ as time.  For a cube $Q_r(\X)$ we write $\wt Q_r(\X)$, or equivalently $\wt Q_r(X)$, for its spatial projection,
\[
\wt Q_r(\X)=\wt Q_r(X)=\{Y\in\R^n:\|Y-X\|_\infty<r\}.
\]

Notice that, when $\eta=0$, every coordinate cube satisfies
\[
Q_r(\X_Q)\subset B_{c_nr}(\X_Q),
\qquad
c_n=(n^2+16)^{1/4}.
\]
For a nonzero form $\eta$, coordinate cubes centered far from the origin
need not have uniformly bounded metric diameter.  Balls centered at the
origin, however, can always be compared with coordinate cubes and
cylinders.  Namely,
\begin{equation}\label{origincubescyllinders} C_r(\0)\subset Q_r(\0)\subset B_{(n+4)r}(\0).\end{equation}

\subsection{Smooth and Lipschitz functions}\label{lipfncts}
We call $f:\R^{n+1}\to\R$ Lipschitz if
\[
|f(\X)-f(\Y)|\leq C\|\Y^{-1}\cdot\X\|,
\]
and denote the least such constant by $\|f\|_{\Lip}$.  For an open set
$U\subset\R^{n+1}$, let $\Lip_0(U)$ be the Lipschitz functions with
compact support contained in $U$.  With $\overline\X=(-X,t)$, put
\[
\Lipodd(U)=\{\varphi\in\Lip_0(U):
\varphi(\X)=-\varphi(\overline\X)\},
\]
and
\[
\Lipeven(U)=\{\varphi\in\Lip_0(U):
\varphi(\X)=\varphi(\overline\X)\}.
\]

If $\mathcal{F}\subset \Lip_0(\mathbb{R}^{n+1})$, we say $\mathcal{F}'\subset \mathcal{F}$ is dense if for every $f\in \mathcal{F}$ there exists $f_m\in \mathcal{F}'$ such that $f_m\to f$ uniformly, $\sup_m\|f_m\|_{\Lip}<\infty$ and there exists a compact set $K\subset \R^{n+1}$ such that $\supp(f_m)\subset K$ for every $m$.

By "smooth" we are always {referring} to ordinary smoothness in the Euclidean coordinates
$(X,t)\in\R^n\times\R$.  It is useful to express this convention in terms
of the left-invariant fields.  For $v\in\R^n$, let $Z_v$ be the
left-invariant field obtained by translating $(v,0)$ from the identity:
\[
Z_vF(X,t)
=
\left.\frac{d}{d\lambda}\right|_{\lambda=0}
F\big((X,t)\cdot(\lambda v,0)\big).
\]
Using the group law, this gives
\[
Z_vF(X,t)=\partial_vF(X,t)+\eta(X,v)\partial_tF(X,t).
\]
Put $T=\partial_t$.  Then
\[
[Z_u,Z_v]=2\eta(u,v)T,
\qquad [Z_v,T]=0,
\]
and
\[
\partial_v=Z_v-\eta(X,v)T.
\]
It follows that ordinary $C^\infty$ smoothness is equivalently described by
the existence and continuity of all iterated derivatives formed from
$Z_{e_1},\ldots,Z_{e_n}$ and $T$, where $(e_i)$ is any basis of $\R^n$.
The field $T$ must be included in this general statement.  If $\eta\ne0$,
then $T$ is a constant multiple of one commutator $[Z_u,Z_v]$, but when
$\eta=0$ the fields $Z_v$ do not detect regularity in the time variable.

Every $F\in C_c^\infty(\R^{n+1})$ is Lipschitz for $d$.  Indeed, write
$q^{-1}p=(H,\tau)$ and $r=d(p,q)$.  Then $|H|\leq r$ and
$|\tau|\leq r^2/4$.  When $r<1$ and one of $F(p),F(q)$ is nonzero, the
horizontal coordinates of $p$ and $q$ stay in a fixed bounded set.  The
ordinary mean value theorem along the coordinate segment
$\lambda\mapsto q\cdot(\lambda H,\lambda\tau)$ therefore gives
\[
|F(p)-F(q)|\lesssim_F |H|+|\tau+\eta(Q,H)|\lesssim_F r,
\]
where $Q$ is the horizontal component of $q$; it stays in a fixed bounded
set in the case under consideration.
When $r\geq1$, the same conclusion follows from the boundedness of $F$.

For a vertical plane \(P\), write
\[
B_R^P:=P\cap B_R(\mathbf 0).
\]  We use the following extension lemma.

\begin{lem}[Plane test functions]\label{planeextensionlem}
Let $P=W\times\R$ be a linear vertical plane.  If $F\in C_c^\infty(P)$ is supported in $B_R^P$, then for every $R'>R$ there is $\widetilde F\in C_c^\infty(\R^{n+1})$, supported in $B_{R'}$, whose restriction to $P$ is $F$.  If $F$ is horizontally even on $P$, then $\widetilde F$ may be chosen horizontally even.  

Moreover, restrictions of smooth compactly supported horizontally even
functions supported in $B_{R'}$ are uniformly dense among horizontally
even Lipschitz functions on $P$ supported in $B_R^P$.
\end{lem}

\begin{proof}
Write $\R^n=W\oplus\mathcal N$.  Since $\supp F$ is a compact subset of
$B_R^P$, choose an even cutoff $\chi\in C_c^\infty(\mathcal N)$, equal
to one near $0$, with sufficiently small support, and set
\[
\widetilde F(X,t)=F(\pi_WX,t)\chi(\pi_{\mathcal N}X).
\]
Since $\supp F\Subset B_R^P\Subset B_{R'}$,
the support of $\chi$ may be chosen sufficiently small that
$\widetilde F$ is supported in $B_{R'}$; moreover,
$\widetilde F|_P=F$.  If $F$ is horizontally even, then so is
$\widetilde F$.  

Finally, for the density assertion, convolve a Lipschitz
function $f$ on the group $P$ with a smooth approximate identity
$\rho_\epsilon$ invariant under $(w,t)\mapsto(-w,t)$.  Written in the form
\[
f_\epsilon(p)=\int_P\rho_\epsilon(q)f(q^{-1}p)\,dq,
\]
left invariance of $d$ shows that
$\|f_\epsilon\|_{\Lip}\leq\|f\|_{\Lip}$.  The convolution preserves
horizontal evenness, converges uniformly to $f$, and has support in any
prescribed neighborhood of $\supp f$ when $\epsilon$ is sufficiently
small.  Applying the extension just constructed completes the proof.
\end{proof}

\subsection{Admissible antisymmetric kernels}\label{CZOdefs}

Our definition of admissible kernels is informed by the parabolic and
Heisenberg Riesz transforms, which play a distinguished role in
applications.  These kernels are antisymmetric in the spatial variables
\cite{BHHLN3, CLY1, CM, Ho, LM}.

In parabolic space, following \cite[Definition~1]{BHHLN2}, a smooth
function $K:\R^{n+1}\setminus\{\0\}\to\R$ is an admissible kernel of
dimension $s$ if
\begin{equation}\label{spacialanti}
K(X,t)=-K(-X,t),\text{ i.e. }K(\X)=-K(\overline{\X}),
\end{equation}
and, for every spatial multiindex $\alpha$ and $j\geq0$,
\[
|\partial_X^\alpha\partial_t^jK(\X)|
\lesssim_{\alpha,j,K}\|\X\|^{-s-|\alpha|-2j}.
\]

In $\mathbb H^N$, we use the higher-dimensional version of
\cite[Definition~1.3]{Orp}.  A smooth function
$K:\mathbb H^N\setminus\{\0\}\to\R$ is an admissible kernel of
dimension $s$ if \eqref{spacialanti} holds,
and, for every $\ell\geq0$ and unit vectors
$v_1,\ldots,v_\ell\in\R^{2N}$,
\[
|Z_{v_1}\cdots Z_{v_\ell}K(\X)|
\lesssim_{\ell,K}\|\X\|^{-s-\ell}.
\]

For the unified framework, an ordinary smooth function
$K:\mathbb G_\eta\setminus\{\0\}\to\R$ is an admissible kernel of
dimension $s$ if \eqref{spacialanti} holds and, for every $\ell,j\geq0$
and unit vectors $v_1,\ldots,v_\ell\in\R^n$,
\begin{equation}\label{etaadmissiblederivatives}
|Z_{v_1}\cdots Z_{v_\ell}T^jK(\X)|
\lesssim_{\ell,j,K}\|\X\|^{-s-\ell-2j}.
\end{equation}
For $\eta=0$, one has $Z_v=\partial_v$, so this is the parabolic
definition.  If $\eta\ne0$, choose
unit vectors $u,v$ with $\eta(u,v)\ne0$.  Since
\[
T=\frac{1}{2\eta(u,v)}(Z_uZ_v-Z_vZ_u),
\]
the commutator identity shows that \eqref{etaadmissiblederivatives} is
equivalent to imposing the bounds using iterated horizontal derivatives
alone, with constants allowed to depend on the fixed form $\eta$.  In
particular it agrees with the Heisenberg definition when
$\eta=\frac12\omega$.

Given an admissible kernel $K$ of dimension $s$, the Calder\'on--Zygmund singular integral operator associated with $K$ and a measure $\mu$ is defined through the truncations
\[
T_\epsilon f(\X)=\int_{\|\Y^{-1}\cdot\X\|>\epsilon}K(\Y^{-1}\cdot\X)f(\Y)\,d\mu(\Y).
\]
We say that $T$ is bounded from $L^2(\mu)$ to $L^2(\mu)$ if there is a constant $C>0$ such that for every $f\in L^2(\mu)$,
\[
\int \sup_{\epsilon>0}|T_\epsilon f(\X)|^2\,d\mu(\X)\le C\|f\|_{L^2(\mu)}^2.
\]

\subsection{Measures}\label{measuresec}

We fix $C_0>0$ to be the Ahlfors regularity constant.  All constants may depend on the dimension, $\eta$, the ADR dimension $s$, and $C_0$.

Throughout the paper, by ``measure'' we mean a locally finite Borel measure.
We say that a nonzero measure $\mu$ with
$\diam(\supp\mu)>0$ is $C_0$-ADR of dimension $s$, abbreviated to
$s$-ADR, if
\begin{align}\label{ADRdef}
C_0^{-1} r^{s} \leq \mu(B_r(\X)) \leq C_0r^{s}
\end{align}
for all $\X \in \supp(\mu)$ and all
$r\in(0,\diam(\supp\mu))$. 

\begin{lem}\label{allscaleADR}
Let $\mu$ be $C_0$-ADR of dimension $s$, put
$E=\supp\mu$ and $D=\diam E$, and interpret
$\min\{r,D\}=r$ when $D=\infty$.  Then
\begin{equation}\label{adrsaturated}
\mu(B_r(\Y))\simeq \min\{r,D\}^s
\qquad(\Y\in E,\ r>0).
\end{equation}
Moreover,
\begin{equation}\label{globalADRupper}
\mu(B_R(\X))\lesssim R^s
\qquad(\X\in\mathbb G_\eta,\ R>0),
\end{equation}
and
\begin{equation}\label{adrscalecomparison}
\mu(B_R(\Y))\lesssim (R/r)^s\mu(B_r(\Y))
\qquad(\Y\in E,\ 0<r\leq R).
\end{equation}
In particular, $\mu$ is doubling on its support at every scale.
\end{lem}

\begin{proof}
If $D=\infty$, then \eqref{adrsaturated} is precisely
\eqref{ADRdef}.  Suppose that $D<\infty$ and fix $\Y_0\in E$.
Since $E\subset B_{2D}(\Y_0)$, geometric doubling gives a uniformly
bounded cover of this ball by balls of radius $D/4$.  Discarding the
balls that miss $E$ and recentering the others at points of $E$ gives
points $\Y_1,\ldots,\Y_N\in E$, with $N$ uniformly bounded, such that
\[
E\subset\bigcup_{i=1}^N B_{D/2}(\Y_i).
\]
Since $D/2<D$, the ADR bounds give
\[
C_0^{-1}(D/2)^s
\leq\mu(B_{D/2}(\Y_0))
\leq\mu(\mathbb G_\eta)
\leq\sum_{i=1}^N\mu(B_{D/2}(\Y_i))
\lesssim D^s.
\]
Together with \eqref{ADRdef}, this proves \eqref{adrsaturated},
including at the endpoint $r=D$.

For arbitrary $\X\in\mathbb G_\eta$, the estimate
\eqref{globalADRupper} is immediate if $B_R(\X)\cap E=\varnothing$.
Otherwise, choose $\Y\in B_R(\X)\cap E$.  Then
\[
B_R(\X)\cap E\subset B_{2R}(\Y),
\]
so \eqref{adrsaturated} gives
\[
\mu(B_R(\X))\leq\mu(B_{2R}(\Y))
\lesssim\min\{2R,D\}^s\lesssim R^s.
\]
Finally, \eqref{adrscalecomparison} follows from
\eqref{adrsaturated} and
\[
\frac{\min\{R,D\}}{\min\{r,D\}}\leq\frac Rr
\qquad(0<r\leq R).
\]
\end{proof}

Recall that a sequence $\mu_j$ converges {weakly} to a measure $\mu$ if
\[
\int_{\R^{n+1}}\varphi(\X)\,d\mu_j(\X)
\longrightarrow
\int_{\R^{n+1}}\varphi(\X)\,d\mu(\X)
\quad\text{for all }\varphi\in C_c(\R^{n+1}).
\]

We use the following well-known compactness lemma.

\begin{lem}\label{adrwl}
Suppose $s>0$ and $\mu_j$ is a sequence of $s$-ADR measures with constant $C_0$.
Assume that there are points $\X_j\in\supp(\mu_j)$ with
$\X_j\to\X$.  Then $(\mu_j)$ is {weakly precompact}.  Every
{weak limit} $\mu$ is either zero, or $\mu$ is $s$-ADR with the same
constant $C_0$ and $\X\in\supp\mu$.  Moreover, the limit cannot be zero if there are constants $R_*,c_*>0$ such that
\begin{equation}\label{adrnoncollapse}
\mu_j(B_{R_*}(\X_j))\geq c_*R_*^s
\qquad\text{for every }j.
\end{equation}
\end{lem}

\subsection{Vertical $\beta$-numbers and uniform rectifiability}\label{URdefn}

For a $(k+2)$-ADR measure $\mu$, put
\[
\beta_{\mu,k}(\X,r)^2
=
\inf_{L\in\VP_k}
\frac{1}{\mu(B_r(\X))}
\int_{B_r(\X)}
\left(\frac{\dist(\Y,L)}{r}\right)^2\,d\mu(\Y).
\]
When $k$ is fixed, we write $\beta_\mu=\beta_{\mu,k}$.  We call $\mu$
\emph{uniformly rectifiable} if
\begin{equation}\label{flatsquare}
\int_0^R\int_{B_R(\X)}
\beta_\mu(\Y,r)^2\,\frac{d\mu(\Y)\,dr}{r}
\leq CR^{k+2}
\end{equation}
for every $\X\in\supp\mu$ and $R>0$.  In parabolic codimension one, this
is the definition of Hofmann--Lewis--Nystr\"om \cite{HLN}.

\subsection{The classes $\mathcal M_m$ and $\alpha$-numbers}
\label{generalalphanumbers}

Fix $0\leq k\leq n-1$ and put $s=k+2$.  For each integer $m> k$, let $\mathcal M_m$ be the class of
measures
\begin{equation}\label{commonflatclass}
(L_p)_\#\bigl(\mathcal H^m\vert_W\times\gamma\bigr),
\end{equation}
where $p\in\mathbb G_\eta$, $W\subset\R^n$ is an $\eta$-isotropic linear $m$-plane, and
$\gamma$ is a Radon measure on $\R$ such that the product measure is
nonzero. 

For $m=k$, we set $\mathcal{M}_k$ to be the class of measures
\begin{equation}\label{commonflatclass2}
(L_p)_\#\bigl(\mathcal H^m\vert_W\times\gamma\bigr),
\end{equation}
where $p\in\mathbb G_\eta$, $W\subset\R^n$ is a linear $m$-plane, and
$\gamma$ is a Radon measure on $\R$ such that the product measure is
nonzero and is invariant under
the translations generated by $\eta(W,W)$:
\begin{equation}\label{flatshiftinvariance}
(\tau_c)_\#\gamma=\gamma
\qquad\text{for every }c\in\eta(W,W),
\end{equation}
where $\tau_c(t)=t+c$.  

The condition \eqref{flatshiftinvariance} is vacuous when $W$ is
$\eta$-isotropic.  Otherwise, since $\eta$ is scalar-valued,
$\eta(W,W)=\R$, so $\gamma$ is a nonzero translation-invariant Radon
measure and hence $\gamma=c\mathcal L^1$ for some $c>0$.  Therefore, if
$\sigma=\mathcal H^m\vert_W\times\gamma\in\mathcal M_k$ and $W$ is not
$\eta$-isotropic, then $\sigma$ is a positive constant multiple of
$\mathcal H^m\vert_W\times\mathcal L^1$.

We set $\mathcal M_m=\varnothing$ if
$m\notin\{k,\ldots,n\}$.  

The class $\mathcal M_m$ is stable under left
translations, homogeneous dilations, and multiplication by positive
constants.  For $\X\in\mathbb G_\eta$ and $r>0$, let
$\mathcal M_m(\X,r)$ be the collection of $C_0$-ADR measures $\lambda$
of dimension $s$
for which
\begin{equation}\label{localcomparisonclass}
\lambda\vert_{B_{3r}(\X)}
=\sigma\vert_{B_{3r}(\X)}
\qquad\text{for some }\sigma\in\mathcal M_m.
\end{equation}
Thus $\lambda$ is globally $C_0$-ADR, while the product identity is required
only on $B_{3r}(\X)$.  For a Radon
measure $\mu$, set
{
\begin{equation}\label{alphamodel}
\alpha_{m,\mu}(\X,r)
=\frac1{r^{s+1}}\inf_{\lambda\in\mathcal M_m(\X,r)}
\sup_{\substack{\supp f\subset B_{3r}(\X)\\\|f\|_{\Lip}\leq1}}
\left|\int f\,d(\mu-\lambda)\right|.
\end{equation}
}
If $\mathcal M_m(\X,r)=\varnothing$, then 
$\alpha_{m, \mu}(\X,r)$ is $+\infty$.  

Set
\[
T_{\X,r}(p)=\delta_{1/r}(\X^{-1}\cdot p),
\]
Then
\[
\lambda\in\mathcal M_m(\X,r)
\quad\Longleftrightarrow\quad
r^{-s}(T_{\X,r})_\#\lambda\in\mathcal M_m(\0,1),
\]
and the same normalization sends
$\alpha_{m,\mu}(\X,r)$ to
$\alpha_{m,r^{-s}(T_{\X,r})_\#\mu}(\0,1)$.

Next, we record the local compactness property of these product classes.  The
statement also incorporates the elementary fact that asymptotically equal
measures have the same {weak limit}.

\begin{lem}[Local product limits]\label{localflatlimit}
Let $0<R<R_0$, fix $m\in\{k,\ldots,n\}$, and suppose that $\mu_j$ converge
{weakly} in $B_{R_0}(\0)$ to a Radon measure $\mu$ with
$\0\in\supp\mu$.  Suppose that $\sigma_j\in\mathcal M_m$ and
\[
 d_j:=
 \sup_{\substack{\supp f\subset B_{R_0}(\0)\\
                  \|f\|_{\Lip}\leq1}}
 \left|\int f\,d(\mu_j-\sigma_j)\right|
 \longrightarrow0.
\]
Then $\sigma_j\rightharpoonup\mu$ {in $B_{R_0}(\0)$}.  After
passing to a subsequence, there is a linear $m$-plane
$W\subset\R^n$ and a Radon measure $\gamma$ on $\R$ such that
\begin{equation}\label{localflatproduct}
 \mu\vert_{B_R(\0)}
 =
 (\mathcal H^m\vert_W\times\gamma)\vert_{B_R(\0)},
 \qquad
 (\tau_c)_\#\gamma=\gamma
 \quad(c\in\eta(W,W)).
\end{equation}
If $m>k$, then $W$ is $\eta$-isotropic.
\end{lem}

\begin{proof}
First, $\sigma_j\rightharpoonup\mu$ {in $B_{R_0}(\0)$}.  Indeed,
if $K\Subset B_{R_0}(\0)$, choose
$0\leq\chi\in C_c^\infty(B_{R_0}(\0))$ with $\chi\geq1$ on $K$.
Then
\[
 \sigma_j(K)
 \leq\int\chi\,d\mu_j+\|\chi\|_{\Lip}d_j,
\]
so the measures $\sigma_j$ have uniformly bounded mass on compact
subsets of $B_{R_0}(\0)$.  Moreover, for
$f\in C_c^\infty(B_{R_0}(\0))$,
\[
 \left|\int f\,d(\mu_j-\sigma_j)\right|
 \leq\|f\|_{\Lip}d_j\longrightarrow0.
\]
Smooth approximation and the preceding local mass bounds give the same
convergence for every $f\in C_c(B_{R_0}(\0))$.

Choose $q_j\in\supp\sigma_j$ with $q_j\to\0$, which is possible by
{weak convergence} as $0\in\supp\mu$.  After translating by $q_j^{-1}$, the definition of
$\mathcal M_m$ gives
\[
 (L_{q_j^{-1}})_\#\sigma_j
 =\mathcal H^m\vert_{W_j}\times\gamma_j,
\]
where $W_j$ is a linear $m$-plane and $\gamma_j$ satisfies
\eqref{flatshiftinvariance} with $W_j$ in place of $W$.  Pass to a
subsequence so that $W_j\to W$, and choose orthogonal maps
$O_jW=W_j$ with $O_j\to I$.  Pulling back by
$(w,t)\mapsto(O_jw,t)$ reduces these measures to
$\mathcal H^m\vert_W\times\gamma_j$.  Since the translations and the
orthogonal maps tend to the identity, these product measures also
converge {weakly} to $\mu$ in $B_{R_0}(\0)$.

If $m>k$, then every $W_j$ is $\eta$-isotropic.  Passing to the limit
in $\eta(w_j,z_j)=0$, with $w_j,z_j\in W_j$ converging to arbitrary
$w,z\in W$, shows that $W$ is $\eta$-isotropic as well.

Select $a\in(R^2/4,R_0^2/4)$ and put $J=[-a,a]$.  Choose a sufficiently
small closed ball $D\subset W$ centered at $0$ such that
$D\times J\Subset B_{R_0}(\0)$, and let
$0\leq\chi\in C_c(B_{R_0}(\0))$ satisfy $\chi\geq1$ on $D\times J$.
Then
\[
 \mathcal H^m(D)\gamma_j(J)
 \leq\int\chi\,d(\mathcal H^m\vert_W\times\gamma_j)\lesssim1.
\]
After a further subsequence, $\gamma_j\vert_J$ converges weakly to a
finite measure $\gamma_J$ on $J$.  Testing the weak convergence against product test functions yields
\[
 \mu\vert_{B_R(\0)}
 =
 (\mathcal H^m\vert_W\times\gamma_J)\vert_{B_R(\0)}.
\]

If $W$ is not $\eta$-isotropic, then neither is $W_j$ for all large $j$.
Thus $\gamma_j=c_j\mathcal L^1$, and the {weak convergence} gives
$c_j\to c>0$ (positivity follows from $\0\in\supp\mu$).  In this case
take $\gamma=c\mathcal L^1$.  If $W$ is $\eta$-isotropic, extend
$\gamma_J$ by zero to $\R$.  In either case \eqref{localflatproduct}
holds.
\end{proof}

\section{Main results and reductions}

\subsection{The parabolic result}\label{parastate}

We first describe the classes in the parabolic $\HSDC$ condition.

Fix $k \in \{0,\ldots,n-1\}$.  For
$m\in\{k+1,k+2\}\cap\{0,\ldots,n\}$, consider $(k+2)$-ADR measures
$\mu=\mathcal H^m\vert_L\times\nu$, where $L\subset\R^n$ is an $m$-plane
and $\nu$ is a Radon measure on the time line $(\R,d_t)$.  Equivalently,
\[
\nu\left(\left(t-\frac{\rho^2}{4},t+\frac{\rho^2}{4}\right)\right)
\approx \rho^{k+2-m}
\]
for $t\in\supp\nu$ and $0<\rho<\diam(\supp\mu)$.  Such a measure $\mu$
satisfies:
\begin{itemize}
    \item $\mu$ is $(k+2)$-ADR,
    \item all admissible $(k+2)$-dimensional CZOs are bounded in $L^2(\mu)$, but
    \item the spatial projection of the support of $\mu$ is $L$, so
    $\beta_{\mu}(\X,R)\gtrsim1$ for every $\X\in\supp\mu$ and $R>0$.
    Consequently, \eqref{flatsquare} fails.
    \end{itemize}
These are precisely the $(k+2)$-ADR measures in $\mathcal M_m$ in the parabolic case with $m=k+1$ or $m=k+2$.  Every measure in
$\mathcal M_{k+1}\cup\mathcal M_{k+2}$ has spatial factor of dimension
greater than $k$.  Our main result says that the David--Semmes conclusion
holds provided the measure quantitatively avoids these two classes, measured by $\alpha_{m,\mu}$.

We say that $\mu$ satisfies the ($k$-dimensional) $\eps$-$\HSDC$
condition if, for every
$\X\in \supp(\mu)$ and $R>0$,
\begin{equation}\label{kdimpack}\int_0^R\int_{B_R(\X)}\mathbf{1}_{\{(\Y,r)\,:\, \alpha_{k+1,\mu}(\Y,r)<\eps \text{ or }\alpha_{k+2,\mu}(\Y,r)<\eps\}}d\mu(\Y)\frac{dr}{r}\leq CR^{k+2}.
\end{equation}

\begin{thm}\label{parathm} Suppose that $\mu$ is a $C_0$-ADR measure of dimension $k+2$, and there exist $\eps>0$ and $C>0$ such that $\mu$ satisfies the $\eps$-$\HSDC$ condition.

If all admissible $(k+2)$-dimensional CZOs are bounded in $L^2(\mu)$, then $\mu$ is uniformly rectifiable.
\end{thm}

In the codimension-one case $k=n-1$, \eqref{kdimpack} is just the condition
\begin{equation}\label{co1pack}
\begin{aligned}
&\int_0^R\int_{B_R(\X)}
\mathbf{1}_{\{(\Y,r)\,:\, \alpha_{n,\mu}(\Y,r)<\eps\}}
\,d\mu(\Y)\frac{dr}{r}\leq CR^{n+1}.
\end{aligned}
\end{equation}

Lemma~\ref{HSDCnecessary} shows that this condition is necessary for
uniform rectifiability.

Time-directed ADR measures form an important subclass of ADR measures.
{In codimension one, every time-backward or time-forward ADR measure satisfies the
$\varepsilon$-$\HSDC$ condition for some $\varepsilon>0$.}  Following
\cite[Definition~2.4]{BHMN2}, for $\X=(X,t)$ we put
{
\[
\begin{aligned}
B_r^-(\X)&=\{\Y=(Y,s)\in B_r(\X):s<t\},\\
B_r^+(\X)&=\{\Y=(Y,s)\in B_r(\X):s>t\}.
\end{aligned}
\]
}
An $(n+1)$-ADR measure $\mu$ on parabolic space is \emph{time-backward
ADR} if $\diam(\supp\mu)=\infty$ and, for some $c_->0$,
\[
\mu(B_r^-(\X))\geq c_-r^{n+1}
\qquad(\X\in\supp\mu,\ r>0).
\]
Time-forward ADR is defined analogously.

\begin{lem}\label{onesidedHSDC}
Suppose that $\mu$ is time-backward ADR or time-forward ADR.  There is
$\eps_0>0$, depending only on $n$, $C_0$, and the one-sided
ADR constant, such that
\[
\alpha_{n,\mu}(\X,r)\geq\eps_0
\qquad(\X\in\supp\mu,\ r>0).
\]
In particular, $\mu$ satisfies \eqref{co1pack}.
\end{lem}

\begin{cor}\label{timesym}
Suppose that $\mu$ is an $(n+1)$-ADR measure on parabolic space which is
either time-backward ADR or time-forward ADR.  If every
admissible $(n+1)$-dimensional CZO is bounded on $L^2(\mu)$, then $\mu$ is
uniformly rectifiable.
\end{cor}

\begin{proof}
Lemma~\ref{onesidedHSDC} gives the $\eps_0$-$\HSDC$, and
Theorem~\ref{parathm}, with $k=n-1$, applies.
\end{proof}

\subsection{The main results in the Heisenberg case}\label{heisenbergtheorems}

Let $N\geq1$ and $0\leq k\leq2N-1$.  For
$\eta=\frac12\omega$, one has $I_\eta=N$.  For
$d\in\{k+1,k+2\}$, the class $\mathcal M_d$ consists of left translates
of measures $\mathcal H^d\vert_L\times\nu$, where $L$ is an
isotropic $d$-plane.  Hence $\mathcal M_d=\varnothing$ when $d>N$.
The examples in Section~\ref{examples} show that, if $\mathcal{M}_d$ is non-empty (which occurs if and only if $d\leq N$), the admissible singular
integrals are bounded in $L^2(\mu)$ for any $(k+2)$-ADR measure in $\mathcal{M}_d$.

A $(k+2)$-ADR measure $\mu$ satisfies the $\eps$-Heisenberg-HSDC
condition if, for every
$\X\in\supp\mu$ and $R>0$,
\begin{equation}\label{heiscontinuousHSDC}
\int_0^R\int_{B_R(\X)}
\mathbf 1_{\left\{\begin{smallmatrix}
\alpha_{k+1,\mu}(\Y,r)<\eps\ \mathrm{or}\\
\alpha_{k+2,\mu}(\Y,r)<\eps
\end{smallmatrix}\right\}}
\,d\mu(\Y)\,\frac{dr}{r}
\leq C R^{k+2}.
\end{equation}
If one of the classes $\mathcal{M}_{k+1}(\0,1)$ or $\mathcal{M}_{k+2}(\0,1)$ is empty, then the corresponding coefficient is
identically $+\infty$ and contributes nothing to the condition.

\begin{thm}[Heisenberg groups]\label{heisenbergthm}
Let $N\geq1$ and $0\leq k\leq2N-1$.  Suppose that $\mu$ is a $C_0$-ADR measure of dimension $k+2$ on $\mathbb H^N$ and that $\mu$ satisfies the $\eps$-Heisenberg-HSDC condition for some $\eps>0$.  If every admissible $(k+2)$-dimensional CZO is bounded on $L^2(\mu)$, then $\mu$ satisfies the vertical $k$-plane $\beta$-Carleson estimate.  Equivalently, $\mu$ is uniformly rectifiable in the sense of Subsection~\ref{URdefn}.
\end{thm}

Lemma~\ref{HSDCnecessary} shows that this condition is necessary for the
vertical $k$-plane $\beta$-Carleson estimate.

\begin{cor}\label{lowcodimHeisCor}
If $N\leq k\leq2N-1$, then every $(k+2)$-ADR measure on $\mathbb H^N$ for which all admissible $(k+2)$-dimensional CZOs are bounded on $L^2(\mu)$ satisfies the vertical $k$-plane $\beta$-Carleson estimate.
\end{cor}

\subsection{Dyadic cubes, Carleson families, and HSDC} \label{dyadsec}
Following \cite{Chr,DS1,DS2}, fix a dyadic decomposition of $\supp\mu$.
We use $E,F$ for dyadic cubes and reserve $Q$ for coordinate cubes.

We use dyadic cubes merely as a bookkeeping device.  It is convenient
to normalize their scales so that
$\diam(E)\leq\frac12\ell(E)$.

Specifically, when $S$ is unbounded, for each $j \in \mathbb Z$ we can construct a family $\mathbb D_j$ of Borel subsets of $S = \supp(\mu)$ (the 
dyadic cubes of the $j$\textsuperscript{th} generation) such that
\begin{itemize}
		\item Each $\mathbb D_j$ is a partition of $S$, i.e. $S = \bigcup_{E \in \mathbb D_j} E$ and $E \cap E' = \emptyset$ whenever $E,E' \in \mathbb D_j$ and $E \neq E'$;
	\item if $E \in \mathbb D_j$ and $E' \in \mathbb D_k$, with $k \leq j$, then either $E \subseteq E'$ or $E \cap E' = \emptyset$;
	\item if $E \in \mathbb D_j$, we have $2^{-j} \lesssim \diam(E) \leq \frac{1}{2}2^{-j}$ and $\mu(E) \approx 2^{-js}$;
	\item if $E \in \mathbb D_j$, there is a point $\Z_E =(Z_E,\tau_E)\in E$ (called the center of $E$) such that $\dist\big(\Z_E,\supp(\mu)\setminus E\big) \gtrsim 2^{-j}$.
\end{itemize}
If $S$ is bounded, choose a top index $j_{\rm top}$ with
$2^{-j_{\rm top}}\approx\diam S$, put $\mathbb D_{j_{\rm top}}=\{S\}$,
and use the same construction only for $j\geq j_{\rm top}$.  In either
case, set
$\mathbb D(\mu)=\mathbb D:=\bigcup_j \mathbb D_j$, over the available
generations.  Given a non-top cube $E \in \mathbb D_j$, the unique dyadic
cube $E' \in \mathbb D_{j-1}$ which contains $E$ is called the parent of
$E$.  If $E$ is from generation $j$, we write $J(E) = j$.  If $J(E)=j$,
we write $\ell(E) = 2^{-j}$.

For $C>0$, we call a family $\mathcal Q\subset\dy(\mu)$
$C$-\emph{Carleson} if
\[
\sum_{E\in\mathcal Q:E\subset E_0}\mu(E)\leq C\mu(E_0)
\qquad\text{for every }E_0\in\dy(\mu).
\]
The family is called Carleson if it is $C$-Carleson for some $C>0$; the
least such $C$ is its Carleson norm.

In the bounded case, ancestor arguments stop at the top cube; the resulting
exceptions form a Carleson family.

The properties of the dyadic grid above imply that there are uniform
constants $c_*>0$ and $A_{\dy}>0$ such that, for every
$E \in \mathbb D(\mu)$,
\begin{align}\label{c*}
B_{c_*\ell(E)}(\Z_E) \cap \supp(\mu)\subset E,
\end{align}
and
\[
E\subset B_{A_{\dy}\ell(E)}(\Z_E)\cap\supp(\mu).
\]
Moreover, we may take $A_{\dy}=3/4<1$, since
$\diam(E)\leq\frac12\ell(E)$.

We use the following dyadic form of \eqref{flatsquare}.

\begin{lem}\label{dyadicUR} An $s$-ADR measure $\mu$ is uniformly rectifiable if and only if there are constants $A> A_{\dy}$ and $C>0$ such that for every $E_0\in \dy(\mu)$,
$$\sum_{E\in \dy(\mu): E\subset E_0} \beta_{\mu}(\Z_E,A\ell(E))^2\mu(E)\leq C \mu(E_0).$$
\end{lem}

For $\varepsilon>0$, we call
\begin{equation}\label{unifieddyadicHSDC}
\left\{E\in\dy(\mu):
\alpha_{k+1,\mu}(\Z_E,\ell(E))<\varepsilon
\quad\text{or}\quad
\alpha_{k+2,\mu}(\Z_E,\ell(E))<\varepsilon
\right\}
\end{equation}
the family of $\varepsilon$-$\HSDC$ cubes.  We say that $\mu$ satisfies
the dyadic $\varepsilon$-$\HSDC$ condition if this family is Carleson.
If $k\geq I_\eta$, then neither an $\eta$-isotropic $(k+1)$-plane nor an
$\eta$-isotropic $(k+2)$-plane exists.  Thus
$\mathcal M_{k+1}=\mathcal M_{k+2}=\varnothing$, the family in
\eqref{unifieddyadicHSDC} is empty, and the condition is vacuous.

\begin{lem}[Changing scales in alpha coefficients]\label{alphaenlarge}
Let $m$ be an integer and let $\mu$ be $s$-ADR.  There are constants
$c_\alpha\in(0,1)$ and
$C_\alpha>1$, depending only on $s$ and independent of $M$, such that the
following holds for every $M\geq2$.

If, for some $\varepsilon>0$, the cubes satisfying
\[
\alpha_{m,\mu}(\Z_E,\ell(E))<\varepsilon
\]
form a Carleson family, then the cubes satisfying
\[
\alpha_{m,\mu}(\Z_E,M\ell(E))<c_\alpha\varepsilon
\]
also form a Carleson family.  Similarly, if the cubes satisfying
\[
\alpha_{m,\mu}(\Z_E,\ell(E))>\varepsilon
\]
form a Carleson family, then the cubes satisfying
\[
\alpha_{m,\mu}(\Z_E,M\ell(E))>C_\alpha\varepsilon
\]
also form a Carleson family.  The Carleson norms in the conclusions may
depend on $M$.
\end{lem}

\begin{proof}
For $q\geq0$, let $E^{(q)}$ denote the $q$-th dyadic ancestor of $E$,
when it exists.  If $\supp\mu$ is bounded, the cubes for which
$E^{(q)}$ does not exist are the top cube and the first $q-1$
generations below it (with no exceptions when $q=0$).  They form a
Carleson family with norm $O(q+1)$.  If
$\mathcal C$ is Carleson, then
\begin{equation}\label{fixedancestors}
\{E:E^{(q)}\text{ exists and }E^{(q)}\in\mathcal C\}
\end{equation}
is Carleson.  Indeed, for a fixed $E_0$, the first $q$ generations below
$E_0$ contribute at most $q\mu(E_0)$.  Below those generations, grouping
the cubes by their $q$-th ancestors gives a bound by
\[
\sum_{\substack{F\in\mathcal C\\F\subset E_0}}\mu(F).
\]

Fix $M\geq2$, and choose integers $q_-,q_+\geq0$ such that
\[
\frac M4<2^{q_-}\leq\frac M2,
\qquad
2M\leq2^{q_+}<4M.
\]
Apart from the Carleson families of top-scale exceptions described above,
put $F_\pm(E)=E^{(q_\pm)}$.  Since $\Z_E\in F_\pm(E)$ and
$A_{\dy}<1$,
\[
B_{3\ell(F_-(E))}(\Z_{F_-(E)})
\subset B_{3M\ell(E)}(\Z_E)
\subset B_{3\ell(F_+(E))}(\Z_{F_+(E)}).
\]
Consequently,
\[
\begin{aligned}
\mathcal M_m(\Z_E,M\ell(E))
&\subset
\mathcal M_m(\Z_{F_-(E)},\ell(F_-(E))),
\\
\mathcal M_m(\Z_{F_+(E)},\ell(F_+(E)))
&\subset
\mathcal M_m(\Z_E,M\ell(E)).
\end{aligned}
\]
Restricting the test functions in \eqref{alphamodel} therefore gives
\[
\alpha_{m,\mu}(\Z_{F_-(E)},\ell(F_-(E)))
\leq 4^{s+1}
\alpha_{m,\mu}(\Z_E,M\ell(E)),
\]
and
\[
\alpha_{m,\mu}(\Z_E,M\ell(E))
\leq 4^{s+1}
\alpha_{m,\mu}(\Z_{F_+(E)},\ell(F_+(E))).
\]
The two assertions now follow from \eqref{fixedancestors}, with
\[
c_\alpha=4^{-(s+1)},
\qquad
C_\alpha=4^{s+1}.
\]
\end{proof}

\subsection{The Littlewood--Paley estimate}

For a function $\varphi$ and $\lambda>0$, put
\[
\varphi_\lambda(\X)=\varphi(\delta_{1/\lambda}\X).
\]
Then
\[
(\varphi_\lambda*\mu)(\Y)=\int \varphi_\lambda(\Z^{-1}\cdot\Y)\,d\mu(\Z).
\]
Throughout this subsection $s$ denotes the ADR dimension of the measure under consideration.

We say that a measure $\mu$ satisfies the $s$-dimensional horizontally
antisymmetric Littlewood--Paley estimate if, for each
$\varphi\in\Lipodd(\mathbb R^{n+1})\cap C^\infty(\R^{n+1})$, the following
holds for $R\in(0,\diam(\supp\mu))$ and $\X\in\supp\mu$:
\begin{align}\label{C2}
\sum_{2^{-j} \leq R}\int_{B_R(\X)} \big(2^{js}(\varphi_{2^{-j}} \ast \mu)(\Y)\big)^2 d\mu(\Y) \leq C(\varphi)R^{s}.
\end{align}
\begin{lem}\label{geoCZbdd}
Suppose that $\mu$ is an $s$-ADR measure and that every admissible
$s$-dimensional CZO is bounded in $L^2(\mu)$.  Then $\mu$ satisfies \eqref{C2}.
\end{lem}

\begin{proof}[Proof sketch]
Fix $\varphi\in \Lipodd(\R^{n+1})\cap C_c^{\infty}(\R^{n+1})$, and let $\eps_j$ be a random sequence of independent $\pm1$ signs with mean zero.  Then
$$K(\X) = \sum_{j\in \mathbb{Z}}\eps_j2^{js}\varphi_{2^{-j}}(\X)$$
is an admissible kernel, with kernel seminorms independent of the signs.
For a bounded scalar sequence $a=(a_j)$, let $K_a$ be the same sum with
$a_j$ in place of $\eps_j$.  A uniform boundedness principle argument
yields
\[
 \left\|\sup_{\delta>0}|T_{K_a,\delta}f|\right\|_{L^2(\mu)}
 \lesssim_\varphi
 \|a\|_\infty\|f\|_{L^2(\mu)}.\]
Fix a finite set $F\subset\mathbb Z$. For each $j\in F$, independently choose $a_j=1$ or $a_j=-1$, each with probability $1/2$, and set $a_j=0$ outside $F$. Applying the preceding estimate and taking expectations gives
$${\sum_{j\in F}}\int_{\R^{n+1}}\bigl(2^{js}(\varphi_{2^{-j}} *(fd\mu))\bigr)^2d\mu\lesssim \|f\|_{L^2(\mu)}^2 \text{ for every }f\in L^2(\mu),
$$
where the sum is finite.

Finally, testing with $f$ equal to the characteristic function of a fixed enlargement of $B_R(\X)$, and using the compact support of $\varphi$, yields \eqref{C2}.
\end{proof}

For $E \in \mathbb D(\mu)$ define
\begin{align}\label{thetaquant}
\Theta_{\mu,\varphi,A}(E) = \int_{B_{A\ell(E)}(\Z_E)} \big(\ell(E)^{-s}(\varphi_{\ell(E)}\ast \mu)(\Y)\big)^2 d\mu(\Y).
\end{align}
We refer to $\Theta_{\mu, \varphi, A}(E)$ as a square function coefficient.  Since $\mu$ is ADR, for a fixed generation $j$, the balls $\{B_{A\ell(E)}(\Z_E):E\in\dy_j(\mu)\}$ have bounded overlap.  Summing the resulting estimates over generations and using \eqref{C2}, we infer the following lemma.

\begin{lem}\label{CZcoef}
Suppose that $\mu$ satisfies \eqref{C2}.  Then, for every
$\varphi\in \Lipodd(\R^{n+1})\cap C^\infty(\R^{n+1})$, $A>0$ and
$E_0\in \dy(\mu)$,
\begin{align}\label{GSALP}
\sum_{E\in \dy(\mu):E\subset E_0}\Theta_{\mu, \varphi, A}(E)
\lesssim_{A,\varphi}\ell(E_0)^{s}.
\end{align}
\end{lem}

\begin{lem}\label{findCarleson}
Suppose that $\mu$ satisfies \eqref{C2} and that $\mathcal{Q}\subset \dy(\mu)$ is a
family of dyadic cubes with the following property: there exist
$\Delta>0$, $A>0$ and a finite family
$\mathcal{F}\subset \Lipodd(\R^{n+1})\cap C^\infty(\R^{n+1})$ such that
$$\max_{\varphi\in \mathcal{F}}\Theta_{\mu, \varphi, A}(E)\geq \Delta \mu(E)\text{ for all }E\in \mathcal{Q}. $$
Then $\mathcal{Q}$ is Carleson with norm depending on $\Delta, A$ and $\mathcal{F}$.
\end{lem}

\begin{proof} Appealing directly to the hypothesis and Lemma \ref{CZcoef}, and using $\mu(E_0)\approx \ell(E_0)^s$, we see that
$$\Delta\sum_{E\in \mathcal{Q}:E\subset E_0}\mu(E)\leq\sum_{\varphi\in \mathcal{F}}\sum_{E\in \mathcal{Q}:E\subset E_0}\Theta_{\mu, \varphi, A}(E)\lesssim \mu(E_0),$$
where the implicit constant depends on $A$ and $\mathcal{F}$.
\end{proof}

\subsection{The continuous and dyadic HSDC conditions}

\begin{lem}[Continuous HSDC implies dyadic HSDC]
\label{continuousdyadicHSDC}
Let $\mu$ be an $s$-ADR measure on $\mathbb G_\eta$, where $s=k+2$.
Suppose that, for some $\varepsilon>0$ and every
$\X\in\supp\mu$, $R>0$,
\begin{equation}\label{generalcontinuousHSDC}
\int_0^R\int_{B_R(\X)}
\mathbf 1_{\{(\Y,r):\,\alpha_{k+1,\mu}(\Y,r)<\varepsilon
\ \mathrm{or}\ \alpha_{k+2,\mu}(\Y,r)<\varepsilon\}}
\,d\mu(\Y)\,\frac{dr}{r}
\lesssim R^s.
\end{equation}
Then the dyadic $\varepsilon'$-$\HSDC$ condition holds for
\begin{equation}\label{continuousdyadicthreshold}
\varepsilon'=2^{-(s+2)}\varepsilon.
\end{equation}
\end{lem}

\begin{proof}
Let $E$ be an $\varepsilon'$-$\HSDC$ cube, write
$\ell=\ell(E)$, and choose the index $m\in\{k+1,k+2\}$ for which
$\alpha_{m,\mu}(\Z_E,\ell)<\varepsilon'$.  Fix
$c<\min\{c_*/2,1/4\}$.  If $\supp\mu$ is bounded, first discard the
uniformly bounded number of top generations for which
$c\ell(E)\geq\diam(\supp\mu)$; these cubes form a Carleson family.  For every remaining cube, lower ADR and
\eqref{c*} give
\[
\mu(B_{c\ell(E)}(\Z_E))\simeq\ell(E)^s\simeq\mu(E).
\]
If
\[
 \Y\in B_{c\ell}(\Z_E)\cap\supp\mu,
 \qquad \frac\ell2\le r\le\frac{3\ell}{4},
\]
then $B_{3r}(\Y)\subset B_{3\ell}(\Z_E)$ (this is why $c<1/4$), and hence
\[
\mathcal M_m(\Z_E,\ell)\subset\mathcal M_m(\Y,r).
\]
Restricting the test functions in \eqref{alphamodel} therefore gives
\[
 \alpha_{m,\mu}(\Y,r)
 \le\left(\frac\ell r\right)^{s+1}
 \alpha_{m,\mu}(\Z_E,\ell)
 <\varepsilon.
\]
The set of such $(\Y,r)\in B_{c\ell(E)}(\Z_E)\times [\ell/2, 3\ell/4]$ has $d\mu\times \frac{dr}{r}$ measure comparable to $\mu(E)$.

Notice that the intervals $[\ell/2,3\ell/4]$ are disjoint for
$\ell=2^{-j}$, $j\in\mathbb Z$, and that, for fixed $j$, the sets
$B_{c\ell(E)}(\Z_E)\cap\supp\mu$, $E\in\dy_j(\mu)$, are pairwise
disjoint.  Finally,
for $E\subset E_0$, each
$B_{c\ell(E)}(\Z_E)\times[\ell/2,3\ell/4]$ lies in
$B_{C\ell(E_0)}(\Z_{E_0})\times(0,C\ell(E_0))$ for a fixed constant $C>1$
depending only on the dyadic constants.  Summing and applying
\eqref{generalcontinuousHSDC} at radius $C\ell(E_0)$ proves the discrete Carleson
estimate.
\end{proof}

\subsection{The general theorem}

The following theorem contains Theorems~\ref{parathm}
and~\ref{heisenbergthm} as special cases.  The remainder of the paper is
devoted to its proof.

\begin{thm}[General $\eta$ theorem]\label{generaletathm}
Fix a bilinear form $\eta$ satisfying
\eqref{etahypothesis}, and let $0\leq k\leq n-1$.  Suppose that $\mu$ is
a $C_0$-ADR measure of dimension $k+2$ on $\mathbb G_\eta$, that $\mu$ satisfies the
$(k+2)$-dimensional horizontally antisymmetric Littlewood--Paley estimate
\eqref{C2}, and that, for some $\varepsilon>0$, the
$\varepsilon$-$\HSDC$ cubes in \eqref{unifieddyadicHSDC} form a
Carleson family.  Then $\mu$ satisfies the vertical $k$-plane
$\beta$-Carleson estimate \eqref{flatsquare}.

The same conclusion holds if the dyadic $\HSDC$ hypothesis is replaced
by the continuous condition \eqref{generalcontinuousHSDC}.  In either
formulation, the Littlewood--Paley hypothesis may be replaced by the
assumption that every admissible $(k+2)$-dimensional CZO is bounded on
$L^2(\mu)$.
\end{thm}

Taking $\eta=0$ or $\eta=\frac12\omega$ and applying
Lemmas~\ref{geoCZbdd} and~\ref{continuousdyadicHSDC} gives the parabolic
and Heisenberg theorems above.  {Lemma~\ref{onesidedHSDC} gives Corollary~\ref{timesym}, and \eqref{maximalisotropicdimension} gives Corollary~\ref{lowcodimHeisCor}.}

\subsection{Necessity of HSDC}

\begin{lem}[Necessity of HSDC]\label{HSDCnecessary}
Fix $\eta$ satisfying \eqref{etahypothesis}, let $0\leq k\leq n-1$, put
$s=k+2$, and let $\mu$ be $C_0$-ADR of dimension $s$ on $\mathbb G_\eta$.
There are constants $\varepsilon_0,b>0$, depending only on $n,k$, the
form $\eta$, and $C_0$, such that, for
$m\in\{k+1,k+2\}$, $\X\in\supp\mu$, and
$0<r<\diam(\supp\mu)$,
\[
 \alpha_{m,\mu}(\X,r)<\varepsilon_0
 \quad\Longrightarrow\quad
 \beta_{\mu,k}(\X,2r)\geq b.
\]
Consequently every $s$-ADR measure on $\mathbb G_\eta$ satisfying the
vertical $k$-plane $\beta$-Carleson estimate satisfies the continuous
$\varepsilon_0$-$\HSDC$ condition \eqref{generalcontinuousHSDC}.
\end{lem}

\begin{proof}
We use the all-scale estimates \eqref{adrsaturated} and
\eqref{globalADRupper} from Lemma~\ref{allscaleADR}.
By translation and dilation covariance, it suffices to prove the first
assertion with $\X=\0$ and $r=1$.  Thus
$\0\in\supp\mu$ and $\diam(\supp\mu)>1$.  If
$\mathcal M_m=\varnothing$, there is nothing to prove.  Otherwise,
suppose that
\[
 \alpha_{m,\mu}(\0,1)<\varepsilon_0.
\]
There are $\lambda\in\mathcal M_m(\0,1)$ and
$\sigma\in\mathcal M_m$ such that
\begin{equation}\label{alphaeps0}
 \lambda\vert_{B_3(\0)}=\sigma\vert_{B_3(\0)}
\end{equation}
and
\begin{equation}\label{alphaeps0test}
 \sup_{\substack{\supp f\subset B_3(\0)\\
                  \|f\|_{\Lip}\leq1}}
 \left|\int f\,d(\mu-\lambda)\right|<\varepsilon_0.
\end{equation}
For
\[
 f_0(q)=\left(\frac12-d(q,\0)\right)_+,
\]
the lower ADR bound gives
\[
 \int f_0\,d\mu
 \geq\frac14\mu(B_{1/4}(\0))
 \geq c_0,
\]
where $c_0>0$ depends only on $s$ and $C_0$.  Once
$\varepsilon_0<c_0$, \eqref{alphaeps0test} implies that
$\supp\lambda\cap B_{1/2}(\0)\neq\varnothing$.  Fix
\[
 z=(Z,\tau)\in\supp\lambda\cap B_{1/2}(\0).
\]
By \eqref{alphaeps0}, $z\in\supp\sigma$.

Write
\[
 \sigma=(L_p)_\#
 \bigl(\mathcal H^m\vert_W\times\gamma\bigr).
\]
Since $m>k$, the plane $W$ is $\eta$-isotropic.  Recentering the product
at the point $z$ gives a Radon measure $\widetilde\gamma$ such
that
\begin{equation}\label{recenteredsigmaHSDC}
 (L_{z^{-1}})_\#\sigma
 =\mathcal H^m\vert_W\times\widetilde\gamma,
 \qquad 0\in\supp\widetilde\gamma.
\end{equation}
Here we used isotropy to see that left translation by a point of
$W\times\R$ acts by ordinary translation on both factors.

Fix an arbitrary $P=(a+V)\times\R\in\VP_k$, where $V$ is a linear
$k$-plane.  Since $m>k$, choose a unit vector
\[
 u\in W\cap V^\perp.
\]
Put $\rho=1/2$.  For one of the two choices of sign,
\[
 \dist(Z\pm\rho u,a+V)\geq\rho.
\]
Indeed, if $c=\langle Z-a,u\rangle$, then
$\max\{|c+\rho|,|c-\rho|\}\geq\rho$.  Choose such a sign and set
\[
 y=z\cdot(\pm\rho u,0).
\]
It follows from \eqref{recenteredsigmaHSDC} that
$y\in\supp\sigma$.  Moreover,
\[
 d(y,z)=\rho,\qquad y\in B_1(\0).
\]
Consequently \eqref{alphaeps0} implies that $y\in\supp\lambda$.

Set $\delta=1/8$ and define
\[
 h_y(q)=\bigl(\delta-d(q,y)\bigr)_+.
\]
Then $\|h_y\|_{\Lip}\leq1$ and
\[
 \supp h_y\subset \overline{B}_\delta(y)\subset B_2(\0)\Subset B_3(\0).
\]
The points $z,y\in\supp\lambda$ are at distance $\rho$, so
$\diam(\supp\lambda)\geq\rho>\delta/2$.  Since $\lambda$ is globally
$C_0$-ADR,
\[
 \int h_y\,d\lambda
 \geq\frac{\delta}{2}\lambda(B_{\delta/2}(y))
 \geq c_1,
\]
where $c_1>0$ depends only on $s$ and $C_0$.  Taking
$\varepsilon_0<c_1/2$, \eqref{alphaeps0test} yields
\[
 \int h_y\,d\mu\geq\frac{c_1}{2},
 \qquad
 \mu(B_\delta(y))\geq\frac{c_1}{2\delta}.
\]
For $q\in B_\delta(y)$, the triangle inequality and
\eqref{verticalplanedistance} give
\[
 \dist(q,P)
 \geq\dist(y,P)-d(q,y)
 \geq\rho-\delta=\frac38.
\]
Hence, uniformly for every $P\in\VP_k$,
\[
 \int_{B_2(\0)}\dist(q,P)^2\,d\mu(q)\geq c_2
\]
for some $c_2>0$.  Combining this with
\eqref{globalADRupper} gives
\[
 \beta_{\mu,k}(\0,2)^2
 =\inf_{P\in\VP_k}
 \frac{1}{4\mu(B_2(\0))}
 \int_{B_2(\0)}\dist(q,P)^2\,d\mu(q)
 \geq c_3.
\]
The first assertion follows with $b=\sqrt{c_3}$.

It remains to prove the consequence.  Let
$D=\diam(\supp\mu)$ and write
\[
 \mathcal E
 =
 \left\{(\Y,r):
 \alpha_{k+1,\mu}(\Y,r)<\varepsilon_0
 \ \text{or}\
 \alpha_{k+2,\mu}(\Y,r)<\varepsilon_0
 \right\}.
\]
For $0<r<D$, the first assertion gives
\[
 \mathbf 1_{\mathcal E}(\Y,r)
 \leq b^{-2}\beta_{\mu,k}(\Y,2r)^2
 \qquad(\Y\in\supp\mu).
\]
Therefore, if the vertical $\beta$-Carleson estimate holds, then for
$\X\in\supp\mu$ and $R>0$,
{
\[
\begin{aligned}
&\int_0^{\min\{R,D\}}\int_{B_R(\X)}
\mathbf 1_{\mathcal E}(\Y,r)\,d\mu(\Y)\,\frac{dr}{r}\\
&\qquad\leq b^{-2}\int_0^{2R}\int_{B_{2R}(\X)}
\beta_{\mu,k}(\Y,\rho)^2\,d\mu(\Y)\,\frac{d\rho}{\rho}
\lesssim R^s.
\end{aligned}
\]
}
If $D<\infty$ and $R>D$, fix $\Y_0\in\supp\mu$.  Since
$\supp\mu\subset B_{2D}(\Y_0)$, \eqref{adrsaturated} gives
$\mu(\mathbb G_\eta)\lesssim D^s$, and hence
\[
 \int_D^R\int_{B_R(\X)}
 \mathbf 1_{\mathcal E}(\Y,r)\,d\mu(\Y)\,\frac{dr}{r}
 \leq\mu(\mathbb G_\eta)\log\frac RD
 \lesssim D^s\log\frac RD
 \lesssim R^s.
\]
This proves \eqref{generalcontinuousHSDC}.
\end{proof}

\subsection{The one-sided parabolic condition}
\label{separationsec}

\begin{proof}[Proof of Lemma~\ref{onesidedHSDC}]
We prove the time-backward case and put $s=n+1$.  Fix
$\X=(X,t)\in\supp\mu$ and $r>0$.  The part of
$B_r^-(\X)\cap\supp\mu$ lying in the slab
\[
\{(Y,u):t-a^2r^2<u<t\}
\]
can be covered by $O(a^{-n})$ balls of radius $C\,a r$ centered on
$\supp\mu$, so the upper ADR bound ensures that this set has mass
$\lesssim a r^s$.
Choosing $0<a<1/2$ sufficiently small in terms of the time-backward ADR
constant, we obtain $c>0$ such that
\begin{equation}\label{separatedpastmass}
\mu\bigl(B_r(\X)\cap\{(Y,u):u\leq t-a^2r^2\}\bigr)
\geq cr^s.
\end{equation}

Suppose that the asserted lower bound for $\alpha_{n,\mu}$ fails.
Translation, dilation, and normalization then give $C_0$-ADR,
time-backward ADR measures $\mu_j$ such that
\[
\0\in\supp\mu_j,
\qquad
\alpha_{n,\mu_j}(\0,1)\longrightarrow0.
\]
The lower ADR bound gives $\mu_j(B_1(\0))\geq C_0^{-1}$, so
Lemma~\ref{adrwl} gives, after passage to a subsequence, a {weak}
limit $\mu_\infty$ which is $s$-ADR and contains $\0$ in its support.

Choose $\lambda_j\in\mathcal M_n(\0,1)$ so that
\[
d_j:=
\sup_{\substack{\supp f\subset B_3(\0)\\ \|f\|_{\Lip}\leq1}}
\left|\int f\,d(\mu_j-\lambda_j)\right|
\to 0\text{ as }j\to \infty.
\]
For each $j$, choose $\sigma_j\in\mathcal M_n$ agreeing with $\lambda_j$
on $B_3(\0)$.  Thus $d_j$ is also the discrepancy between $\mu_j$ and
$\sigma_j$.  Since $d_j\to0$, Lemma~\ref{localflatlimit}, applied with
$m=n$, $R_0=3$, and $R=2$, gives a Radon measure $\gamma$ on $\R$ such that
\begin{equation}\label{onesidedlimitproduct}
\mu_\infty|_{B_2(\0)}
=
(\mathcal H^n\times\gamma)|_{B_2(\0)}.
\end{equation}
Here the spatial plane is necessarily $\R^n$.  In particular,
$0\in\supp\gamma$.

We claim that $\supp\gamma$ contains no nontrivial interval in a
neighborhood of $0$.  Indeed, suppose that
$J\subset\supp\gamma$ is such an interval with
$\{0\}\times J\Subset B_2(\0)$.  For $t$ in a fixed smaller interval
$J_0\Subset J$ and all sufficiently small $\rho$, the ADR bounds and
\eqref{onesidedlimitproduct} give
\[
 \rho^{n+1}
 \lesssim \mu_\infty(B_\rho(0,t))
 \leq C_n\rho^n
 \gamma\left(\left(t-\frac{\rho^2}{4},
                         t+\frac{\rho^2}{4}\right)\right).
\]
Thus each interval on the right has $\gamma$-mass $\gtrsim\rho$.
Taking $\gtrsim\rho^{-2}$ points of $J_0$ separated by a fixed multiple
of $\rho^2$ produces disjoint such intervals in a fixed compact subset
of $J$, forcing its $\gamma$-mass to be $\gtrsim\rho^{-1}$, a
contradiction.

Choose $u<0$, as close to $0$ as necessary, with
$u\notin\supp\gamma$, and let $b$ be the right endpoint of the component
of $\R\setminus\supp\gamma$ containing $u$.  Then
\[
 u<b\leq0,
 \qquad b\in\supp\gamma,
 \qquad \gamma((u,b))=0.
\]
Set $Z=(0,b)$.  By \eqref{onesidedlimitproduct},
$Z\in\supp\mu_\infty$.  Choose $\rho>0$ so small that
\[
\overline{B}_\rho(Z)\Subset B_2(\0),
\qquad
b-\frac{\rho^2}{4}>u.
\]
Choose $Z_j=(z_j,t_j)\in\supp\mu_j$ with $Z_j\to Z$.  Applying
\eqref{separatedpastmass} at $Z_j$ with radius $\rho/2$ gives
\[
\mu_j\left(
 B_{\rho/2}(Z_j)\cap
 \left\{(Y,v):v\leq t_j-\frac{a^2\rho^2}{4}\right\}
\right)
\geq c\,2^{-s}\rho^s.
\]
For all sufficiently large $j$, the set on the left is contained in the
compact set
\[
K:=
\overline{B}_{3\rho/4}(Z)\cap
\left\{(Y,v):v\leq b-\frac{a^2\rho^2}{8}\right\}.
\]
Every point of $K$ satisfies
\[
 v\geq b-\frac{9\rho^2}{64}>u,
 \qquad
 v\leq b-\frac{a^2\rho^2}{8}<b.
\]
Hence $K\subset B_2(\0)\cap(\R^n\times(u,b))$, and
\eqref{onesidedlimitproduct} gives $\mu_\infty(K)=0$.  Finally, upper
semicontinuity of mass on compact sets under {weak convergence}
yields
\[
0<c\,2^{-s}\rho^s
\leq\limsup_{j\to\infty}\mu_j(K)
\leq\mu_\infty(K)=0,
\]
a contradiction.

The time-forward case follows by reversing the time inequalities.
\end{proof}
\medskip

\section{The cylinder blow-up:  from qualitative control to quantitative control} \label{cylsection}

The main result of this section uses Littlewood--Paley coefficients to
transfer qualitative closeness to $\mathcal M_k$ into the Carleson measure
condition for the vertical $\beta$-numbers.

\subsection{The beta stopping lemma}

We first isolate the standard stopping-time reduction to cubes on which the $\beta$-number is comparable to its value at a fixed larger scale.

Set
\begin{equation}\label{dyadicbetaradius}
r_{\dy}:=3A_{\dy}.
\end{equation}
For $E\in\dy(\mu)$ set
\begin{align}\label{dyadicbeta}
\beta_\mu(E)=\beta_\mu(\Z_E,r_{\dy}\ell(E)).
\end{align}
We say that $E$ is $\Lambda$-$\beta$-doubling if
$$
\beta_\mu(\Z_E,\Lambda\ell(E))\leq \Lambda\beta_\mu(E),
$$
and denote the family of such cubes by $\mathbb D_{\operatorname{up},\Lambda}(\mu)$.  When $\Lambda$ is fixed we suppress it from the notation.

\begin{lem}\label{dyuplem}
Let $\mu$ be an ADR measure.  There is a sufficiently large dyadic number $\Lambda_0$, depending only on the ADR and dyadic constants, such that for every dyadic $\Lambda\geq\Lambda_0$ there is a constant $C(\Lambda)$ with the following property: for every $E_0\in\dy(\mu)$,
$$
\sum_{\substack{E\in\dy(\mu)\\E\subset E_0}}\beta_\mu(E)^2\mu(E)
\leq C(\Lambda)\left(\sum_{\substack{E\in\mathbb D_{\operatorname{up},\Lambda}(\mu)\\E\subset E_0}}\beta_\mu(E)^2\mu(E)+\mu(E_0)\right).
$$
\end{lem}

\begin{proof}
The proof is the usual ancestor argument.  Assume $\Lambda\geq 2A_{\dy}$.
For bounded support, cubes for which the ancestor requested below lies above the top
generation occur only within $O_\Lambda(1)$ generations of the top; their
total contribution is $O_\Lambda(\mu(E_0))$.  We discard them into the
terminal term.  For every remaining cube $E$, let $E^+$ be the first
dyadic ancestor such that
\[
\ell(E^+)\geq \frac{\Lambda}{2A_{\dy}}\ell(E).
\]
Then, by the choice of the first such ancestor,
\[
\ell(E^+)<\frac{\Lambda}{A_{\dy}}\ell(E).
\]
Since $E\subset E^+$, both $\Z_E$ and $\Z_{E^+}$ belong to $E^+$.  Therefore
\[
d(\Z_E,\Z_{E^+})\leq A_{\dy}\ell(E^+).
\]
Consequently, if $\X\in B_{\Lambda\ell(E)}(\Z_E)$, then
\[
d(\X,\Z_{E^+})
\leq \Lambda\ell(E)+A_{\dy}\ell(E^+)
\leq 3A_{\dy}\ell(E^+)=r_{\dy}\ell(E^+).
\]
Thus
\[
B_{\Lambda\ell(E)}(\Z_E)
\subset
B_{r_{\dy}\ell(E^+)}(\Z_{E^+}).
\]
The two radii are comparable by the preceding inequalities, with
constants depending only on $A_{\dy}$.  Hence
\eqref{adrscalecomparison} and the monotonicity of the least-squares
error under enlargement of balls give a constant $C_{\rm anc}$,
independent of $\Lambda$, such that
$$
\beta_\mu(\Z_E,\Lambda\ell(E))\leq C_{\rm anc}\beta_\mu(E^+).
$$
If $E$ is not $\Lambda$-$\beta$-doubling, then
$$
\Lambda\beta_\mu(E)<\beta_\mu(\Z_E,\Lambda\ell(E))\leq C_{\rm anc}\beta_\mu(E^+),
$$
and hence
$$
\beta_\mu(E)\leq \frac{C_{\rm anc}}{\Lambda}\beta_\mu(E^+).
$$
Choose $\Lambda$ sufficiently large that $\theta:=C_{\rm anc}/\Lambda<1$.  For each cube $G$, choose one ancestor $G^+$ for which the preceding estimate holds.  Let $M=M(\Lambda)$ be an upper bound for the number of dyadic generations between $G$ and $G^+$, over all possible $G$.  Then
\[
\beta_\mu(G)\leq \theta\,\beta_\mu(G^+)
\]
whenever $G$ is not $\Lambda$-$\beta$-doubling.

Now fix a top cube $E_0$.  For every cube $E\subset E_0$, define a chain by
\[
E^{(0)}=E,
\qquad
E^{(i+1)}=(E^{(i)})^+.
\]
Let $N(E)$ be the first integer $i\geq0$ for which either $E^{(i)}$ is $\Lambda$-$\beta$-doubling, or $E^{(i+1)}$ is not contained in $E_0$.  In the first case set $\widetilde E=E^{(i)}$.  In the second case set $\widetilde E=E_0$.  

If $\widetilde E=F\neq E_0$, then all cubes before $F$ in the chain are not $\Lambda$-$\beta$-doubling, so iteration gives
\begin{equation}\label{dyancest}
\beta_\mu(E)
\leq
\theta^{N(E)}\beta_\mu(F)
\leq
\left(\frac{\ell(E)}{\ell(F)}\right)^c\beta_\mu(F),
\qquad
c:=\frac{-\log_2\theta}{M}>0.
\end{equation}
If $\widetilde E=E_0$, either the chain reaches the doubling cube $E_0$,
or it stops at a cube whose next chosen ancestor exits $E_0$.  In the
second case that last cube has side length comparable, with constants
depending only on $\Lambda$, to $\ell(E_0)$.  In both cases the same
iteration, together with the bound $\beta_\mu\lesssim1$, gives
\begin{equation}\label{dyancestterminal}
\beta_\mu(E)
\lesssim_\Lambda
\left(\frac{\ell(E)}{\ell(E_0)}\right)^c.
\end{equation}

For a fixed stopping cube $F\neq E_0$,
$$
\sum_{\substack{E:\,\widetilde E=F}}\beta_\mu(E)^2\mu(E)
\lesssim_\Lambda
\sum_{m\geq0}2^{-2cm}\beta_\mu(F)^2
\sum_{\substack{E\subset F\\ \ell(E)=2^{-m}\ell(F)}}\mu(E)
\lesssim_\Lambda \beta_\mu(F)^2\mu(F).
$$
The cubes assigned terminally to $E_0$ ($\widetilde{E}=E_0$) are estimated in the same way, using \eqref{dyancestterminal} in place of \eqref{dyancest}, and contribute at most $C(\Lambda)\mu(E_0)$.  Summing over all genuine stopping cubes $F\in\mathbb D_{\operatorname{up},\Lambda}(\mu)$ contained in $E_0$, and adding the terminal contribution, proves the claim.
\end{proof}

\subsection{Pruning by odd square coefficients}\label{pruningsec}

Let $P=W\times\R$ be a linear vertical plane and let $\mathcal N=W^\perp$.  If $v\in \mathcal N$ and $\psi$ is horizontally even, put
$$
\varphi_{v,\psi}(x,t)=\langle v,x\rangle\psi(x,t).
$$
These are horizontally odd kernels.
For an affine vertical plane $P$ and $\rho>0$, write
\[
P_\rho=\{\X\in\mathbb G_\eta:\dist(\X,P)<\rho\}.
\]

Fix
$$
\tau_0=\frac18,\qquad \tau_1=\frac14,\qquad \tau_2=\frac12.
$$
{The constants $\tau_0,\tau_1,\tau_2$ fixed in Subsection~\ref{pruningsec} satisfy $2\tau_0=\tau_1$ and $\tau_0+\tau_2<1$. These relations are used in the proof of Lemma~\ref{prunalt}.}
Later, we also use the point-scale notation
\[
\Theta_{\mu,\varphi,A}(\Z,r)
:=
\int_{B_{Ar}(\Z)}\left|r^{-s}(\varphi_r*\mu)(\X)\right|^2\,d\mu(\X).
\]
Thus $\Theta_{\mu,\varphi,A}(E)=
\Theta_{\mu,\varphi,A}(\Z_E,\ell(E))$ for every dyadic cube $E$.

\begin{lem}[The pruning lemma]\label{prunalt}
Let $A\geq1$, {let $r>0$,} let $\mu$ be $s$-ADR, let $P=(a+W)\times\R$ be an affine vertical plane of homogeneous dimension $s$, and let $\Z\in\supp\mu$.  Let $\psi$ be a horizontally even Lipschitz function such that $0\leq\psi\leq1$, supported in $B_{\tau_2A}(\0)$ and equal to $1$ on $B_{\tau_1A}(\0)$.  Suppose
\[
\mu(B_{\tau_0Ar}(\Z))\geq c_{\mathrm{mass}}r^s
\]
for some $c_{\mathrm{mass}}>0$, and suppose
$$
\frac1{\mu(B_{Ar}(\Z))}\int_{B_{Ar}(\Z)}\left(\frac{\dist(\X,P)}{Ar}\right)^2d\mu(\X)\leq\beta^2.
$$
Put $m:=\dim W^\perp$, and let $\{v_1,\ldots,v_m\}$ be an orthonormal basis of the horizontal normal space $\mathcal N=W^\perp$.  There are constants $K_A$ and $C_A$, depending only on $A$, $c_{\mathrm{mass}}$, the ADR constant, and the dimensions, such that
\[
\int_{B_{\tau_0Ar}(\Z)\setminus P_{K_A\beta r}}
\left(\frac{\dist(\X,P)}r\right)^2d\mu(\X)
\leq C_A\sum_{i=1}^m
\Theta_{\mu,\varphi_{v_i,\psi},A}(\Z,r).
\]
\end{lem}

\begin{proof}
By translation and dilation, assume $\Z=\0$ and $r=1$, and put
\[
M_0:=\mu(B_{\tau_0A}(\0)).
\]
By \eqref{adrscalecomparison}, comparison of the two concentric balls
and the local mass hypothesis gives
\begin{equation}\label{pruningmasscomparison}
M_0\geq c_{\mathrm{mass}},
\qquad
\mu(B_A(\0))\lesssim_AM_0.
\end{equation}
If $\beta=0$, the hypothesis implies that $\mu$ is supported on $P$ inside
$B_A(\0)$, and the conclusion is immediate.  We may therefore assume that
$\beta>0$.
Let $H_i$ be the
vertical hyperplane containing $P$ whose horizontal unit normal is $v_i$,
and write $u_i(\X)$ for the signed distance from $\X$ to $H_i$.  We first
claim that there are constants $K_A^{(1)}$ and $C_A^{(1)}$ such that
\begin{equation}\label{onecoordinateprune}
\int_{B_{\tau_0A}(\0)\cap\{|u_i|>K_A^{(1)}\beta\}}
u_i(\X)^2\,d\mu(\X)
\leq C_A^{(1)}\Theta_{\mu,\varphi_{v_i,\psi},A}(\0,1).
\end{equation}

Write $H_i=\{(x,t):\langle v_i,x\rangle=c_i\}$ and set
$x_i=\langle v_i,x\rangle-c_i$ and
$y_i=\langle v_i,y\rangle-c_i$.  Then
$x_i-y_i=\langle v_i,x-y\rangle$.  Since $P\subset H_i$, the hypothesis
and \eqref{pruningmasscomparison} give
\[
\int_{B_A(\0)}|y_i|^2\,d\mu
\leq A^2\beta^2\mu(B_A(\0))
\lesssim_{A,c_{\mathrm{mass}}}\beta^2M_0.
\]
Consequently, after choosing $D_A$ sufficiently large, the set
\[
G_i=B_{\tau_0A}(\0)\cap\{|y_i|\leq D_A\beta\}
\]
satisfies $\mu(G_i)\geq c_AM_0$.

If $\X=(x,t)\in B_{\tau_0A}(\0)$ and $x_i>K_A^{(1)}\beta$, the contribution
of $G_i$ to
\[
(\varphi_{v_i,\psi}*\mu)(\X)
=\int (x_i-y_i)\psi(\Y^{-1}\X)\,d\mu(\Y)
\]
is at least $c_Ax_iM_0$, provided
$K_A^{(1)}\geq2D_A$.  Here we used
\[
\X,\Y\in B_{\tau_0A}(\0)
\quad\Longrightarrow\quad
\Y^{-1}\X\in B_{2\tau_0A}(\0)=B_{\tau_1A}(\0),
\]
so that $\psi(\Y^{-1}\X)=1$ on $G_i$.  The possible negative
contribution is bounded by
\[
\begin{aligned}
\int_{B_A(\0)\cap\{|y_i|>x_i\}}|y_i|\,d\mu(\Y)
&\leq \frac1{x_i}\int_{B_A(\0)}|y_i|^2\,d\mu(\Y)\\
&\lesssim_{A,c_{\mathrm{mass}}}\frac{\beta^2}{x_i}M_0.
\end{aligned}
\]
Indeed,
\[
\X\in B_{\tau_0A}(\0),\quad \psi(\Y^{-1}\X)\ne0
\quad\Longrightarrow\quad
\Y\in B_{(\tau_0+\tau_2)A}(\0)\subset B_A(\0).
\]
Increasing $K_A^{(1)}$ makes this error at
most half of the preceding lower bound.  The same argument with $-v_i$
treats $x_i<-K_A^{(1)}\beta$.  Thus
\[
|(\varphi_{v_i,\psi}*\mu)(\X)|
\gtrsim_{A,c_{\mathrm{mass}}}|u_i(\X)|M_0
\]
whenever $\X\in B_{\tau_0A}(\0)$ and
$|u_i(\X)|>K_A^{(1)}\beta$.  Squaring, integrating, and using
\eqref{pruningmasscomparison} gives
\[
\Theta_{\mu,\varphi_{v_i,\psi},A}(\0,1)
\gtrsim_{A,c_{\mathrm{mass}}}
M_0^2
\int_{B_{\tau_0A}(\0)\cap\{|u_i|>K_A^{(1)}\beta\}}
u_i(\X)^2\,d\mu(\X).
\]
Since $M_0\geq c_{\mathrm{mass}}$, this proves
\eqref{onecoordinateprune}.

Set $K_A=2\sqrt m\,K_A^{(1)}$.  Since
\[
\dist(\X,P)^2=\sum_{i=1}^m u_i(\X)^2,
\]
if $\X\notin P_{K_A\beta}$, then
\[
\dist(\X,P)^2
\leq2\sum_{i=1}^m u_i(\X)^2
\mathbf 1_{\{|u_i(\X)|>K_A^{(1)}\beta\}}.
\]
Summing \eqref{onecoordinateprune} over $i$ proves the lemma.
\end{proof}

\subsection{Affine cylinders}\label{flatlimitsec}

\begin{lem}\label{coordproductcylrigidlem}
Let $R>0$, let $P=W\times\R$ be a linear vertical subgroup, and put
$\mathcal N=W^\perp$.  Let
\[
\nu=\mathcal H^k\vert_W\times\gamma,
\]
where $\gamma$ is a Radon measure satisfying
\[
(\tau_c)_\#\gamma=\gamma
\qquad\text{for every }c\in\eta(W,W).
\]

Suppose that $\sigma$ is a Radon measure on $\mathcal N\times P$ and that,
for some $C_\sigma<\infty$,
\begin{equation}\label{cylprojidentity}
\supp\sigma\cap(\mathcal N\times B_{20R}^P)
\subset B^{\mathcal N}_{C_\sigma}\times B_{20R}^P,
\qquad
(\pi_P)_\#\bigl(\sigma\vert_{\mathcal N\times B_{20R}^P}\bigr)
=
\nu\vert_{B_{20R}^P}.
\end{equation}
Assume also that, for every horizontally even Lipschitz function $\psi$ on
$P$ with $\supp\psi\subset B_{5R}^P$,
\begin{equation}\label{cylcancel}
\int_{\mathcal N\times B_{20R}^P}
(a-b)\psi(q^{-1}p)\,d\sigma(b,q)=0
\end{equation}
for $\sigma$-a.e. $(a,p)\in\mathcal N\times B_{5R}^P$.  Then there are a
linear map $L:W\to\mathcal N$ and $a_0\in\mathcal N$ such that
\[
\supp\sigma\cap(\mathcal N\times B_{2R}^P)
\subset
\{(Lw+a_0,w,t):(w,t)\in B_{2R}^P\cap(W\times\supp\gamma)\}.
\]
\end{lem}

\begin{proof}
It suffices to prove the result
for $R=1$.

We write points of $P$ as $p=(w,t)$ and $q=(z,u)$, with $w,z\in W$.  Let $\pi_P:\mathcal N\times P\to P$ denote the projection.

We first upgrade the cancellation identity \eqref{cylcancel} from an almost-everywhere statement to a pointwise statement on the support.  Let $\psi$ be a horizontally even Lipschitz function on $P$, with $\supp\psi\subset B_5^P$, and define
\[
\mathcal V_\psi(a,p)
=
\int_{\mathcal N\times B_{20}^P}
(a-b)\psi(q^{-1}p)\,d\sigma(b,q).
\]
The support condition in \eqref{cylprojidentity} gives
\[
\supp\sigma\cap(\mathcal N\times B_{20}^P)
\subset B_{C_\sigma}^{\mathcal N}\times B_{20}^P.
\]
Therefore, dominated
convergence shows that $\mathcal V_\psi$ is continuous in $(a,p)\in \mathcal{N}\times B^P_5$.

The defining identity gives $\mathcal V_\psi=0$ for $\sigma$-a.e. $(a,p)\in\mathcal N\times B_5^P$.  Hence
\begin{equation}\label{pointwisecylcancel}
\mathcal V_\psi(a,p)=0
\qquad
\text{for every }(a,p)\in\supp\sigma\cap(\mathcal N\times B_5^P).
\end{equation}

We now prove that (within $\mathcal{N}\times B^P_5$) $\sigma$ is supported on a graph over $P$.  Choose a nonnegative horizontally even Lipschitz function $\psi_0$, supported in a sufficiently small neighborhood of the identity in $P$, and positive in a smaller neighborhood of the identity.  For $p\in B_5^P$, set
\[
S(p)=\int_{B_{20}^P}\psi_0(q^{-1}p)\,d\nu(q)
\]
and
\[
T(p)=
\int_{\mathcal N\times B_{20}^P}
 b\,\psi_0(q^{-1}p)\,d\sigma(b,q).
\]
Again, dominated convergence shows that $S$ and $T$ are continuous.

If $p=(w,t)\in\supp\nu\cap B_5^P$, then $S(p)>0$.  Indeed,
\[
q^{-1}p=(w-z,t-u+\eta(w-z,w)).
\]
For $z$ sufficiently close to $w$, the horizontal component $w-z$ is small.  Moreover $\eta(w-z,w)\in\eta(W,W)$, and $\gamma$ is invariant under the corresponding time translation.  Since $t\in\supp\gamma$, the definition of support and the positivity of $\psi_0$ near the identity give positive $\nu$-mass to the set of such $q$.

Now take $(a,p)\in\supp\sigma\cap(\mathcal N\times B_5^P)$.  The projection identity implies that $p\in\supp\nu\cap B_5^P$, and hence $S(p)>0$.  Applying \eqref{pointwisecylcancel} with $\psi_0$, and using $(\pi_P)_\#(\sigma\vert_{\mathcal N\times B_{20}^P})=\nu\vert_{B_{20}^P}$, gives
\[
0
=
\int_{\mathcal N\times B_{20}^P}
(a-b)\psi_0(q^{-1}p)\,d\sigma(b,q)
=
aS(p)-T(p).
\]
Thus
\[
a=G(p):=\frac{T(p)}{S(p)}.
\]
Therefore
\begin{equation}\label{supportgraphG}
\supp\sigma\cap(\mathcal N\times B_5^P)
\subset
\{(G(p),p):p\in\supp\nu\cap B_5^P\}.
\end{equation}
The function $G$ is continuous on $\supp\nu\cap B_5^P$.

(\ref{supportgraphG}) and (\ref{cylprojidentity}) imply the identity
\begin{equation}\label{graphpushforward}
\sigma\vert_{\mathcal N\times B_5^P}
=
(G,\operatorname{id})_\#\bigl(\nu\vert_{B_5^P}\bigr).
\end{equation}
Since $G$ is continuous and $p\in\supp\nu$,  (\ref{graphpushforward}) implies
that $(G(p),p)\in\supp\sigma$ for every $p\in\supp\nu\cap B_5^P$.
Consequently, whenever $p=(w,t)\in\supp\nu\cap B_2^P$, and $\psi\in \Lip_0(B_3^P)$ is horizontally even ($\psi\in \Lip_0(B_3^P)$ ensures that $p\in B_2^P$ and $q^{-1}p\in\supp\psi$ imply $q\in B_5^P$), the identity \eqref{pointwisecylcancel} becomes the  equation
\begin{equation}\label{coordproductgrapheqn}
\int
\bigl(G(w,t)-G(z,u)\bigr)
\psi(w-z,t-u+\eta(w-z,w))
\,d\mathcal H^k(z)d\gamma(u)=0.
\end{equation}

We next remove the time variable.  Fix
$p=(w,t)\in\supp\nu\cap B_2^P$.  Let $\rho_\epsilon$ be an even approximate identity on $W$ supported in $B_{\epsilon}^W$, and let $\theta\in \Lip_0((-9/4,9/4))$.  Writing $\xi=w-z$ and taking
\[
\psi(\xi,\tau)=\rho_\epsilon(\xi)\theta(\tau)
\]
in \eqref{coordproductgrapheqn}, we get for sufficiently small $\epsilon>0$,
\[
\int
\bigl(G(w,t)-G(w-\xi,u)\bigr)
\rho_\epsilon(\xi)
\theta(t-u+\eta(\xi,w))
\,d\xi\,d\gamma(u)=0.
\]
Letting $\epsilon\downarrow0$, using the continuity of $G$ on $\supp\nu$ and the identity $\eta(0,w)=0$, gives
\begin{equation}\label{timegrapheqn}
\int
\bigl(G(w,t)-G(w,u)\bigr)\theta(t-u)\,d\gamma(u)=0
\end{equation}
for every $(w,t)\in\supp\nu\cap B_2^P$.

We claim that $G(w,t)$ is independent of $t$.  Fix $w$ and two points $t,t_0\in\supp\gamma$ such that $(w,t),(w,t_0)\in B_2^P$.  Let $\vartheta\geq0$ be a Lipschitz function on $\R$, supported in $(-1,1)$ and positive in a neighborhood of $0$.  For $r>0$ sufficiently small, set
\[
\theta_r(s)=\vartheta\left(\frac{s-(t-t_0)}{r}\right).
\]
Since $|t|,|t_0|<1$, for sufficiently small $r$ we have
$\theta_r\in\Lip_0((-9/4,9/4))$, so this is an admissible choice in
\eqref{timegrapheqn}.
Then
\[
\theta_r(t-u)=\vartheta\left(\frac{t_0-u}{r}\right),
\]
and \eqref{timegrapheqn} gives
\[
G(w,t)\int \vartheta\left(\frac{t_0-u}{r}\right)\,d\gamma(u)
=
\int G(w,u)\vartheta\left(\frac{t_0-u}{r}\right)\,d\gamma(u).
\]
Since $t_0\in\supp\gamma$ and $\vartheta$ is positive near $0$, the denominator
\[
m_r:=\int \vartheta\left(\frac{t_0-u}{r}\right)\,d\gamma(u)
\]
is finite and strictly positive.  Hence
\[
G(w,t)
=
\frac{1}{m_r}
\int G(w,u)\vartheta\left(\frac{t_0-u}{r}\right)\,d\gamma(u).
\]
As $r\downarrow0$, the weights are supported where $u\to t_0$.  Since $G$ is continuous on $\supp\nu$, the right-hand side tends to $G(w,t_0)$.  Therefore
\[
G(w,t)=G(w,t_0).
\]
Thus there is a well-defined function $B$ on
\[
\Omega:=\pi_W\bigl(B_2^P\cap(W\times\supp\gamma)\bigr)
\]
such that
\begin{equation}\label{GequalsB}
G(w,t)=B(w)
\qquad
\text{for every }(w,t)\in\supp\nu\cap B_2^P.
\end{equation}
Since
\[
B_2^P
=
\{(w,t)\in P:|w|^4+16t^2<16\},
\]
we have
\[
\Omega
=
\bigcup_{\substack{t\in\supp\gamma\\ |t|<1}}
\left\{
w\in W:
|w|<(16-16t^2)^{1/4}
\right\}.
\]
Thus $\Omega$ is either empty or {an open ball} in $W$.  If $\Omega$
is empty, the conclusion is immediate, so we assume otherwise.

It remains to prove that $B$ is affine.  Fix $w_0\in\Omega$, and choose
$t_0\in\supp\gamma$ such that $(w_0,t_0)\in B_2^P$.  Since $B_2^P$ is open,
there is a neighborhood $U$ of $w_0$ such that $(w,t_0)\in B_2^P$ for
$w\in U$.  Choose a nonnegative Lipschitz function $\theta$, positive near
$0$ and supported sufficiently close to $0$.  After shrinking $U$ if necessary, we can choose an
even compactly supported Lipschitz function $\phi:W\to\mathbb R$, supported
sufficiently close to $0$, so that
\[
\psi(\xi,\tau)=\phi(\xi)\theta(\tau)
\]
is supported in $B_3^P$ and, whenever $w\in U$, $\xi\in\supp\phi$, and
$\theta(t_0-u+\eta(\xi,w))\neq0$, one has $(w-\xi,u)\in B_2^P$.
Using \eqref{GequalsB} in \eqref{coordproductgrapheqn}, with
$p=(w,t_0)$, yields
\[
0=
\int_W
\bigl(B(w)-B(w-\xi)\bigr)\phi(\xi)
\left[
\int\theta(t_0-u+\eta(\xi,w))\,d\gamma(u)
\right]d\xi.
\]
Since $\eta(\xi,w)\in\eta(W,W)$, the time-translation invariance of $\gamma$ ensures that
\[
\int\theta(t_0-u+\eta(\xi,w))\,d\gamma(u)
=
\int\theta(t_0-u)\,d\gamma(u)>0.
\]
We therefore obtain the identity
\begin{equation}\label{spatialconveqn}
\int_W
\bigl(B(w)-B(w-\xi)\bigr)\phi(\xi)\,d\xi=0
\end{equation}
for every $w\in U$.  Since $w_0$ was arbitrary, this identity holds locally
throughout $\Omega$ for every {even Lipschitz $\phi$ supported in a sufficiently small ball centered at $0$}.

If the dimension $k=0$, then $W=\{0\}$ and $B$ is already affine.  We may therefore
assume $k\geq1$ in the remaining argument.

We now use a convolution argument.  Let $U'\Subset\Omega$.  By compactness,
there is $\epsilon_0>0$ such that \eqref{spatialconveqn} holds for every
$w\in U'$ and every even $\phi$ supported in $B_{\epsilon_0}^W$.  Let
$\rho_\epsilon$ be a standard even smooth approximate identity on $W$ with
$0<\epsilon<\epsilon_0$.  Taking
$\phi=\partial_{ij}\rho_\epsilon$ in \eqref{spatialconveqn} is allowed,
since two derivatives preserve evenness and the support remains in
$B_{\epsilon_0}^W$.  Moreover,
$\int_W\partial_{ij}\rho_\epsilon=0$, so the term containing $B(w)$
vanishes.  Hence
\[
\partial_{ij}(B*\rho_\epsilon)=0
\qquad\text{on }U'.
\]
Letting $\epsilon\downarrow0$ shows that $\partial_{ij}B=0$ in the sense of
distributions on $U'$.  Since $U'\Subset\Omega$ was arbitrary, the Hessian
of $B$ vanishes distributionally on the connected open set $\Omega$.
Therefore $B$ is affine on $\Omega$.

Thus there are a linear map $L:W\to\mathcal N$ and a vector $a_0\in\mathcal N$ such that
\[
B(w)=Lw+a_0
\qquad\text{for }w\in\Omega.
\]
Combining this with \eqref{supportgraphG} and \eqref{GequalsB}, we conclude that
\[
\supp\sigma\cap(\mathcal N\times B_2^P)
\subset
\{(Lw+a_0,w,t):(w,t)\in B_2^P\cap(W\times\supp\gamma)\},
\]
as required.
\end{proof}

\subsection{The cylinder blow-up}\label{mastercylproofsec}

\begin{prop}[Cylinder blow-up]\label{mastercylprop}
Let $s=k+2$.  There exist constants $\rho,\Delta>0$, a large power of two $A\gg1$, and a finite family $\mathcal K$ of smooth compactly supported horizontally odd Lipschitz kernels supported in $B_A(\0)$ such that the following holds.  If $E\in\mathbb D_{\operatorname{up},A}(\mu)$ is an $A$-$\beta$-doubling cube for an $s$-ADR measure $\mu$ and
$$
\alpha_{k,\mu}(\Z_E,A\ell(E))\leq\rho,
$$
then
$$
\max_{\varphi\in\mathcal K}\Theta_{\mu,\varphi,A}(E)
\geq \Delta\beta_\mu(E)^2\mu(E).
$$
\end{prop}

\begin{proof}
Set $R_{\dy}:=2r_{\dy}$.  Take $A$ to be a power of $2$ with $A\geq\Lambda_0$ and large enough that
\[
20R_{\dy}<\tau_0A.
\]
{This guarantees that the inclusions between balls needed for \eqref{mastertail}, \eqref{masterprojidentity}, and \eqref{limitcancelgoal} hold.}

Choose a smooth horizontally even cutoff $\psi_1$ with $0\leq\psi_1\leq1$, supported in $B_{\tau_2A}(\0)$ and equal to one on $B_{\tau_1A}(\0)$.  Choose a countable family $\{\psi_i\}_{i\geq2}\subset C_c^\infty(B_A(\0))\cap\Lipeven(B_A(\0))$ which is dense, in the uniform norm, in the set of even Lipschitz functions supported in $B_A(\0)$.  Construct the sequence so that, for every $R'<R<A$, the collection of those $\psi_j$ supported in $B_R(\0)$ is {dense in $\Lipeven(B_{R'}(\0))$ with respect to the supremum norm}.  For $j\geq1$, define
\begin{equation}\label{masterfinitefamilies}
\mathcal K_j=
\{x_\alpha\psi_i:1\leq i\leq j,\ 1\leq\alpha\leq n\}.
\end{equation}
Thus the kernels \(x_\alpha\psi_1\), \(1\leq\alpha\leq n\), needed
{to apply Lemma~\ref{prunalt}} belong to every \(\mathcal K_j\).  More generally,
for any horizontal normal vector \(v\), the kernel
\(\langle v,x\rangle\psi_1\) is a linear combination of these coordinate
kernels.  Likewise, every kernel \(x_\alpha\psi_i\) belongs
to \(\mathcal K_j\) for all sufficiently large \(j\).

Assume the proposition fails.  For each $j\geq1$, apply its negation with
\[
\rho=\Delta=j^{-1},
\qquad
\mathcal K=\mathcal K_j,
\]
and let $\widetilde E_j,\widetilde\mu_j$ be the resulting set and measure for which the statement fails.
Put $\ell_j=\ell(\widetilde E_j)$ and normalize $\widetilde E_j$ by
\[
T_j(p)=\delta_{\ell_j^{-1}}
\bigl(\Z_{\widetilde E_j}^{-1}\cdot p\bigr),
\qquad
\mu_j=\ell_j^{-s}(T_j)_\#\widetilde\mu_j,
\]
and transport the dyadic grid by $T_j$.  Writing
$E_j=T_j(\widetilde E_j)$, we get
\[
\alpha_{k,\mu_j}(\0,A)
=\alpha_{k,\widetilde\mu_j}
  (\Z_{\widetilde E_j},A\ell_j),
\qquad
\beta_{\mu_j}(\0,R)
=\beta_{\widetilde\mu_j}
  (\Z_{\widetilde E_j},R\ell_j),
\]
and
\[
\Theta_{\mu_j,\varphi,A}(E_j)
=\ell_j^{-s}\Theta_{\widetilde\mu_j,\varphi,A}(\widetilde E_j),
\qquad
\mu_j(E_j)=\ell_j^{-s}\widetilde\mu_j(\widetilde E_j).
\]
Thus, for every $j$, we obtain an $s$-ADR measure
$\mu_j$ and an $A$-$\beta$-doubling cube $E_j$ with
$\Z_{E_j}=\0$ and $\ell(E_j)=1$ such that
$$
\alpha_{k,\mu_j}(\0,A)\leq j^{-1}
$$
and
\begin{equation}\label{masterfailcoeffunit}
\max_{\varphi\in\mathcal K_j}\Theta_{\mu_j,\varphi,A}(E_j)
\leq j^{-1}\widehat\beta_j^2\mu_j(E_j),
\qquad
\widehat\beta_j:=\beta_{\mu_j}(E_j).
\end{equation}
The dyadic mass and containment estimates give
\begin{equation}\label{mastercubemass}
\mu_j(E_j)\simeq1,
\qquad
E_j\subset B_{A_{\dy}}(\0).
\end{equation}
Since $A_{\dy}<\tau_0A$, it follows that
\begin{equation}\label{masterpruningmass}
\mu_j(B_{\tau_0A}(\0))\geq\mu_j(E_j)\gtrsim1.
\end{equation}
Put
$$
\beta_j=\beta_{\mu_j}(\0,A),
$$
and let $P_j$ be a minimizing vertical plane for $\beta_j$.   Since
$\ell(E_j)=1$, \eqref{dyadicbeta} gives
\[
\widehat\beta_j=\beta_{\mu_j}(E_j)=\beta_{\mu_j}(\0,r_{\dy}).
\]
The $A$-$\beta$-doubling condition gives
$A^{-1}\beta_j\leq \widehat\beta_j$.  Conversely, using $P_j$ as a
competitor for $\widehat\beta_j$ gives
{
\[
\begin{aligned}
\widehat\beta_j^2
&\leq\frac{1}{\mu_j(B_{r_{\dy}}(\0))}
\int_{B_{r_{\dy}}(\0)}
\left(\frac{\dist(\Y,P_j)}{r_{\dy}}\right)^2d\mu_j(\Y)\\
&\leq\frac{A^2\mu_j(B_A(\0))}
{r_{\dy}^2\mu_j(B_{r_{\dy}}(\0))}\beta_j^2
\leq C(A,A_{\dy})^2\beta_j^2,
\end{aligned}
\]
}
where the last inequality follows from \eqref{adrscalecomparison} for
the two concentric balls.  Hence
\begin{equation}\label{masterbetacomparison}
A^{-1}\beta_j\leq \widehat\beta_j\leq C(A,A_{\dy})\beta_j.
\end{equation}
We have $\beta_j>0$ for every $j$.  Indeed, if $\beta_j=0$, then
\eqref{masterbetacomparison} gives $\widehat\beta_j=0$, so $E_j$ cannot be
a counterexample to the asserted nonnegative lower bound.

Let $\Lambda_j\in\mathcal M_k(\0,A)$ be almost-minimizing competitors
for $\alpha_{k,\mu_j}(\0,A)$, and choose $\nu_j\in\mathcal M_k$
such that
\[
\Lambda_j\vert_{B_{3A}(\0)}
=\nu_j\vert_{B_{3A}(\0)}.
\]
Then
\[
A^{-(s+1)}
\sup_{\substack{\supp f\subset B_{3A}(\0)\\ \|f\|_{\Lip}\leq1}}
\left|\int f\,d(\mu_j-\nu_j)\right|
\leq
\alpha_{k,\mu_j}(\0,A)+j^{-1}.
\]
By \eqref{masterpruningmass} and Lemma~\ref{adrwl}, after passing to
a subsequence,
\[
\mu_j\rightharpoonup\mu_0
\]
{on \(\R^{n+1}\)}, where $\mu_0$ is $s$-ADR and
$\0\in\supp\mu_0$.

Since $A$ is fixed, the preceding comparison distance tends to zero.
Lemma~\ref{localflatlimit}, with $m=k$, $R=2A$, $R_0=3A$, and
$\sigma_j=\nu_j$, gives
\begin{equation}\label{modelweaklimit}
\nu_j\rightharpoonup\mu_0,
\qquad
\mu_j\rightharpoonup\mu_0
\end{equation}
{in $B_{3A}(\0)$}, and gives a linear $k$-plane $W$ and a Radon
measure $\gamma$ such that, with $P_0=W\times\R$,
\begin{equation}\label{limitingflatproduct}
\mu_0\vert_{B_{2A}(\0)}
=
(\mathcal H^k\vert_W\times\gamma)\vert_{B_{2A}(\0)},
\qquad
(\tau_c)_\#\gamma=\gamma
\quad(c\in\eta(W,W)).
\end{equation}
Let $0\leq\zeta\leq1$ belong to $\Lip_0(B_{2A}(\0))$ and equal one on
$B_A(\0)$.  Since
$\mu_0$ is supported on $P_0$ in $B_{2A}(\0)$, {weak convergence} gives
\[
\int \zeta(\Y)\dist(\Y,P_0)^2\,d\mu_j(\Y)\longrightarrow0.
\]
Moreover, \eqref{mastercubemass} gives
$\mu_j(B_A(\0))\geq\mu_j(E_j)\gtrsim1$.  Using $P_0$ as a competitor in
the definition of $\beta_j$ therefore gives $\beta_j\to0$.

We apply Lemma~\ref{prunalt}, with $r=1$, cutoff $\psi_1$, and the uniform
local mass bound \eqref{masterpruningmass}, to the plane $P_j$.  Write $W_j$
for the horizontal direction of $P_j$, and choose an
orthonormal basis $\{v_{j,1},\ldots,v_{j,n-k}\}$ of $W_j^\perp$.  The
kernels $x_\alpha\psi_1$, $1\leq\alpha\leq n$, belong to
$\mathcal K_j$.  Expressing each normal vector in the coordinate basis,
\eqref{masterfailcoeffunit}, \eqref{mastercubemass}, and
\eqref{masterbetacomparison} give
\[
\Theta_{\mu_j,\varphi_{v_{j,i},\psi_1},A}(E_j)
\lesssim \sum_{\alpha=1}^n
\Theta_{\mu_j,x_\alpha\psi_1,A}(E_j)
\lesssim j^{-1}\beta_j^2.
\]
Thus the right hand side in the conclusion of Lemma~\ref{prunalt},
divided by $\beta_j^2$, tends to zero.  Hence
\begin{equation}\label{abstractpruned}
\int_{B_{\tau_0A}(\0)\setminus (P_j)_{K_A\beta_j}}
\left(\frac{\dist(\Y,P_j)}{\beta_j}\right)^2d\mu_j(\Y)\to0.
\end{equation}

After passing to a subsequence, write the minimizing planes as
\[
P_j=(\xi_j+W_j)\times\R,
\qquad
\xi_j\in W_j^\perp,
\qquad
W_j\to W'.
\]
Since $P_j$ minimizes $\beta_j$,
\begin{equation}\label{minimizingplanesecondmoment}
\int_{B_A(\0)}\dist(\Y,P_j)^2\,d\mu_j(\Y)
=A^2\beta_j^2\mu_j(B_A(\0))\longrightarrow0.
\end{equation}
We first claim that $\xi_j\to0$.  Otherwise, after passing to a
subsequence, $|\xi_j|\geq\varepsilon>0$.  Put
$r_\varepsilon=\min\{\varepsilon/4,A/2\}$, and choose
$\chi\in C_c(B_{r_\varepsilon}(\0))$ with $0\leq\chi\leq1$ such that
$\int\chi\,d\mu_0>0$; this is possible because $\0\in\supp\mu_0$.  On
$\supp\chi$,
\[
\dist(\Y,P_j)\geq\dist(\0,P_j)-d(\Y,\0)
\geq\frac{3\varepsilon}{4}.
\]
Consequently,
\[
\int\chi(\Y)\dist(\Y,P_j)^2\,d\mu_j(\Y)
\geq\frac{9\varepsilon^2}{16}\int\chi\,d\mu_j,
\]
whose right-hand side has a positive limit.  This contradicts
\eqref{minimizingplanesecondmoment}, and hence $\xi_j\to0$.

Set $P'=W'\times\R$.  The functions $\dist(\,\cdot\,,P_j)$ converge
uniformly on compact sets to $\dist(\,\cdot\,,P')$.  Hence, for every
$0\leq\chi\in C_c(B_A(\0))$,
\[
\int\chi(\Y)\dist(\Y,P')^2\,d\mu_0(\Y)
=\lim_{j\to\infty}
\int\chi(\Y)\dist(\Y,P_j)^2\,d\mu_j(\Y)=0.
\]
Therefore
\[
\supp\mu_0\cap B_A(\0)\subset P'.
\]
On the other hand, \eqref{limitingflatproduct} and
$\0\in\supp\mu_0$ imply that $\0\in\supp\gamma$ and that a neighborhood of
$\0$ in $W\times\{0\}$ lies in $\supp\mu_0$.  Hence $W\subset W'$.  Since
both spaces have dimension $k$, $W=W'$.  Thus the minimizing planes converge
to $P_0$, and we put
\[
\mathcal N=W^\perp.
\]
Let $\pi_{P_0}:\mathcal N\times P_0\to P_0$ be the coordinate
projection, and let
\[
\Pi_{P_0}:\mathbb G_\eta\to P_0,
\qquad
\Pi_{P_0}(Y,s)=(\pi_WY,s),
\]
be the spatial projection onto $P_0$.

Choose orthogonal maps $O_j\in O(n)$ such that
\[
O_jW=W_j,
\qquad
O_j\mathcal N=W_j^\perp,
\qquad
O_j\to I .
\]
(These maps are used only as Euclidean coordinate transformations: they
are not assumed to preserve $\eta$.)
Define the coordinate map
$$
S_j:\mathcal N\times P_0\to \R^{n+1},
\qquad
S_j(a,w,t)=(\xi_j+O_j(w+\beta_j a),t),
$$
and set
$$
\sigma_j=(S_j^{-1})_\#\bigl(\mu_j\vert_{B_{\tau_0A}(\0)}\bigr).
$$
Here $p=(w,t)$ is the $P_0$-coordinate and $a$ is the normal coordinate, measured in units of $\beta_j$ relative to $P_j$.  For $\Y=(Y,s)$, the inverse coordinates are
\[
a_j(\Y)=\beta_j^{-1}\pi_{\mathcal N}O_j^{-1}(Y-\xi_j),
\qquad
p_j(\Y)=\bigl(\pi_WO_j^{-1}(Y-\xi_j),s\bigr),
\]
where $S_j^{-1}\Y=(a_j(\Y),p_j(\Y))$.  Since $O_j\to I$ and $\xi_j\to0$, we have
\[
p_j(\Y)\to (\pi_WY,s)=\Pi_{P_0}(\Y)
\]
uniformly for $\Y$ in compact sets.  Also, for $\Y=S_j(a,w,t)$,
\[
\dist(\Y,P_j)
=
\dist_{\R^n}(\xi_j+O_j(w+\beta_j a),\xi_j+W_j)
=
\beta_j |a| .
\]
If $C\Subset\mathcal N\times P_0$, then the explicit formula for $S_j$,
together with $\xi_j\to0$, $O_j\to I$, and $\beta_j\to0$, shows that
the sets $S_j(C)$ lie in one fixed compact subset
$C'\Subset\mathbb G_\eta$.  Consequently
\[
\sigma_j(C)\leq\mu_j(C'),
\]
which is uniformly bounded.  Thus, after passing to a subsequence,
\[
\sigma_j\rightharpoonup\sigma
\]
{in $\mathcal{N}\times P_0$}.  The pruning estimate gives, for every compact $K\Subset B_{20R_{\dy}}^{P_0}$,
\begin{equation}\label{mastersecondtail}
\int_{(\mathcal N\setminus B_{\mathcal N}(0,K_A))\times K}|a|^2\,d\sigma_j(a,p)\to0.
\end{equation}
Indeed, on the support of $\sigma_j$ we have $\dist(S_j(a,p),P_j)/\beta_j=|a|$.  Therefore the left-hand side is bounded by the integral in \eqref{abstractpruned}.  In particular,
\begin{equation}\label{mastertail}
\sigma_j\bigl((\mathcal N\setminus B_{\mathcal N}(0,K_A))\times K\bigr)\to0.
\end{equation}
{We now pass to the limit} in (\ref{mastertail}).  Let
\(0\leq\vartheta\leq1\) be a function in \(C_c(\mathcal N\times B_{20R_{\dy}}^{P_0})\) supported in the set
where \(|a|>K_A+\varepsilon\).  If \(K\Subset B_{20R_{\dy}}^{P_0}\) contains
the projection of \(\supp\vartheta\), then \eqref{mastertail} gives
\[
\int \vartheta\,d\sigma_j\to0.
\]
Passing to the limit and then letting \(\varepsilon>0\) be arbitrary gives
\begin{equation}\label{masterslabsupport}
\supp\sigma\cap(\mathcal N\times B_{20R_{\dy}}^{P_0})
\subset \overline B_{\mathcal N}(0,K_A)\times B_{20R_{\dy}}^{P_0}.
\end{equation}

We next verify the projection and cancellation hypotheses of
Lemma~\ref{coordproductcylrigidlem} at the scale $R_{\dy}$.  We first claim
that
\begin{equation}\label{masterprojidentity}
(\pi_{P_0})_\#\bigl(\sigma\vert_{\mathcal N\times B_{20R_{\dy}}^{P_0}}\bigr)
=
\mu_0\vert_{B_{20R_{\dy}}^{P_0}}.
\end{equation}
Let $f\in C_c(B_{20R_{\dy}}^{P_0})$ be nonnegative, and choose a compact set $K\Subset B_{20R_{\dy}}^{P_0}$ containing $\supp f$.  Let $M>K_A$ and choose $\chi_M\in C_c(\mathcal N)$ with $0\leq\chi_M\leq1$, with $\chi_M=1$ on $\overline B_{\mathcal N}(0,K_A)$, and with $\supp\chi_M\subset B_{\mathcal N}(0,M)$.  By \eqref{mastertail},
$$
\int f(p)\,d\sigma_j(a,p)
=
\int f(p)\chi_M(a)\,d\sigma_j(a,p)+o_j(1).
$$
For any $\Y$ write
$$
S_j^{-1}\Y=(a_j(\Y),p_j(\Y)),
\qquad p_j(\Y)\in P_0.
$$
Since $K\Subset B_{20R_{\dy}}^{P_0}$, $20R_{\dy}<\tau_0A$, and
$S_j(a,p)\to p$ uniformly on
$\overline B_{\mathcal N}(0,M)\times K$, the conditions
$p_j(\Y)\in K$ and $|a_j(\Y)|\le M$ imply
$\Y\in B_{\tau_0A}(\0)$ for all large $j$.
Thus the restriction in the definition of $\sigma_j$ is harmless for this
integrand, and pulling back by $S_j$ gives
$$
\int f(p)\chi_M(a)\,d\sigma_j(a,p)
=
\int f(p_j(\Y))\chi_M(a_j(\Y))\,d\mu_j(\Y).
$$
Since $20R_{\dy}<\tau_0A$ and $S_j(a,p)\to p$ uniformly for
$p\in K$ and $|a|\leq M$, choose
$\zeta\in\operatorname{Lip}_0(B_{\tau_0A}(\0))$ which equals one on a
neighborhood of $K$ containing all points $S_j(a,p)$ with $p\in K$ and
$|a|\leq M$, for all large $j$.  Inserting $\zeta$ does not change the last integral.  Using \eqref{mastertail} again, and the fact that $\chi_M=1$ on $\overline B_{\mathcal N}(0,K_A)$, we may remove $\chi_M$ at the cost of another $o_j(1)$ error:
$$
\int f(p)\chi_M(a)\,d\sigma_j(a,p)
=
\int f(p_j(\Y))\zeta(\Y)\,d\mu_j(\Y)+o_j(1).
$$
The functions $\Y\mapsto f(p_j(\Y))\zeta(\Y)$ have the common compact
support $\supp\zeta$ and converge uniformly there to
$\Y\mapsto f(\Pi_{P_0}(\Y))\zeta(\Y)$.  The masses
$\mu_j(\supp\zeta)$ are uniformly bounded, so
{weak convergence} of $\mu_j$ gives
{
\[
\begin{aligned}
&\lim_{j\to\infty}\int f(p_j(\Y))\zeta(\Y)\,d\mu_j(\Y)\\
&\qquad=\int f(\Pi_{P_0}(\Y))\zeta(\Y)\,d\mu_0(\Y)
=\int_{P_0}f(p)\,d\mu_0(p),
\end{aligned}
\]
}
because $\supp\zeta\subset B_{\tau_0A}(\0)\subset B_{2A}(\0)$,
$\mu_0$ is supported on $P_0$ there, and $\zeta=1$ on $\supp f$.  On the other hand, {weak convergence} of $\sigma_j$ to $\sigma$ gives
$$
\int f(p)\chi_M(a)\,d\sigma_j(a,p)
\to
\int f(p)\chi_M(a)\,d\sigma(a,p).
$$
By \eqref{masterslabsupport}, $\chi_M\equiv1$ for $\sigma$-a.e.\ $(a,p)\in\mathcal N\times B_{20R_{\dy}}^{P_0}$
with $p\in\supp f$, so that $f(p)\chi_M(a)=f(p)$ holds $\sigma$-a.e.; thus
$$
\int f\,d(\pi_{P_0})_\#\bigl(\sigma\vert_{\mathcal N\times B_{20R_{\dy}}^{P_0}}\bigr)
=
\int f\,d\mu_0.
$$
By linearity this proves \eqref{masterprojidentity}.  

It remains to prove that, for every compactly supported
horizontally even Lipschitz function \(\psi\) on \(P_0\) with
\(\supp\psi\subset B_{5R_{\dy}}^{P_0}\)
\begin{equation}\label{limitcancelgoal}
\int_{\mathcal N\times B_{20R_{\dy}}^{P_0}}
(a-b)\,\psi(q^{-1}p)\,d\sigma(b,q)
=
0
\end{equation}
for every
\((a,p)\in\supp\sigma\cap(\mathcal N\times B_{5R_{\dy}}^{P_0})\).  Write
$$
(a,p)=(a,w,t),
\qquad
(b,q)=(b,z,u),
$$
with $a,b\in\mathcal N$ and $p=(w,t),q=(z,u)\in P_0$.  The definition of
$S_j$ gives $S_j(a,p)\to p$ uniformly on compact sets.  Since inversion and
multiplication are continuous in $\mathbb G_\eta$, it follows immediately
that
\[
S_j(b,q)^{-1}\cdot S_j(a,p)\longrightarrow q^{-1}p
\]
uniformly on compact subsets of
$(\mathcal N\times P_0)\times(\mathcal N\times P_0)$.  We shall also use
the exact formula for its horizontal component:
\begin{equation}\label{coordinategroupdiff}
\pi_H\bigl(S_j(b,q)^{-1}\cdot S_j(a,p)\bigr)
=O_j\bigl(w-z+\beta_j(a-b)\bigr).
\end{equation}

Fix $v\in\mathcal N$ and put $v_j=O_jv\in W_j^\perp$.  Fix an index $i$
such that $\psi_i$ is supported in $B_{6R_{\dy}}(\0)$.  For all large $j$,
the kernel
$$
\varphi_{j,i,v}(X,t):=\langle v_j,X\rangle\psi_i(X,t)
$$
is a linear combination of members of $\mathcal K_j$.  By \eqref{masterfailcoeffunit} and \eqref{masterbetacomparison},
\[
\Theta_{\mu_j,\varphi_{j,i,v},A}(E_j)
\lesssim
\sum_{\alpha=1}^n
\Theta_{\mu_j,x_\alpha\psi_i,A}(E_j)
\lesssim j^{-1}\widehat\beta_j^2\mu_j(E_j)
=o(\beta_j^2),
\]
where we used \eqref{mastercubemass}.  Equivalently,
\begin{equation}\label{fixedkernelL2small}
\int_{B_A(\0)}
\left|\int \varphi_{j,i,v}(\Y^{-1}\X)\,d\mu_j(\Y)\right|^2d\mu_j(\X)
=o(\beta_j^2).
\end{equation}
For these coordinates we have the exact identity
\begin{equation}\label{normalizedpullbackkernel}
\frac{1}{\beta_j}\varphi_{j,i,v}\bigl(S_j(b,q)^{-1}\cdot S_j(a,p)\bigr)
=
\langle v,a-b\rangle\psi_i\bigl(S_j(b,q)^{-1}\cdot S_j(a,p)\bigr).
\end{equation}
Indeed, by \eqref{coordinategroupdiff} and the orthogonality of $O_j$,
\[
\begin{aligned}
\frac{1}{\beta_j}
\left\langle v_j,O_j(w-z+\beta_j(a-b))\right\rangle
&=
\frac{1}{\beta_j}
\left\langle O_jv,O_j(w-z+\beta_j(a-b))\right\rangle \\
&=
\frac{1}{\beta_j}
\left\langle v,w-z+\beta_j(a-b)\right\rangle \\
&=
\langle v,a-b\rangle,
\end{aligned}
\]
because $v\in\mathcal N=W^\perp$ and $w-z\in W$.

Let \(K_0\Subset\mathcal N\times B_{5R_{\dy}}^{P_0}\) be compact.  Choose
\(\chi_0\in C_c(\mathcal N\times B_{5R_{\dy}}^{P_0})\), with
$0\leq\chi_0\leq1$, such that \(\chi_0\equiv 1\) on a neighborhood of
\(K_0\).  There is a compact set
\(K_1\Subset B_{20R_{\dy}}^{P_0}\) such that, for all large $j$,
\((a,p)\in\supp\chi_0\) and
\[
\psi_i\bigl(S_j(b,q)^{-1}\cdot S_j(a,p)\bigr)\neq0
\]
imply $q\in K_1$.  Indeed, uniform convergence on $\supp\chi_0$ gives
\[
S_j(\supp\chi_0)\subset B_{6R_{\dy}}(\0)\subset B_A(\0)
\]
for all large $j$.  Since $\supp\psi_i\subset B_{6R_{\dy}}(\0)$, a nonzero
kernel factor then forces $S_j(b,q)\in B_{12R_{\dy}}(\0)$.  On that fixed
compact set, $q=p_j(S_j(b,q))$ and $p_j\to\Pi_{P_0}$ uniformly, which gives
the claim.  Enlarge $K_1$, if necessary, so that it also contains every
$q$ which can satisfy $\psi_i(q^{-1}p)\ne0$ for
$(a,p)\in\supp\chi_0$.

Fix \(M>K_A+1\).  Choose
\[
\chi_{1,M}\in C_c(\mathcal N\times B_{20R_{\dy}}^{P_0})
\]
with \(0\leq \chi_{1,M}\leq1\), such that
\(\chi_{1,M}\equiv1\) on a neighborhood of
\(\overline B_{\mathcal N}(0,M)\times K_1\), and such that
\(\supp\chi_{1,M}\subset B_{\mathcal N}(0,M+1)\times K_1'\) for some compact
\(K_1'\Subset B_{20R_{\dy}}^{P_0}\).  For \(v\in\mathcal N\), define
\[
G_{j,i,v,M}((a,p),(b,q))
:=
\chi_0(a,p)\chi_{1,M}(b,q)
\langle v,a-b\rangle
\psi_i\bigl(S_j(b,q)^{-1}\cdot S_j(a,p)\bigr).
\]
We claim first that
\begin{equation}\label{truncatedpulledbackL2}
\int
\left|
\int
G_{j,i,v,M}((a,p),(b,q))\,d\sigma_j(b,q)
\right|^2
d\sigma_j(a,p)
\longrightarrow 0 .
\end{equation}
Indeed, let
\[
H_{j,i,v}((a,p),(b,q))
=
\chi_0(a,p)\langle v,a-b\rangle
\psi_i\bigl(S_j(b,q)^{-1}\cdot S_j(a,p)\bigr).
\]
{By \eqref{normalizedpullbackkernel},
\[\int H_{j,i,v}((a,p),(b,q))\,d\sigma_j(b,q)
=\frac{\chi_0(a,p)}{\beta_j}\left[\varphi_{j,i,v}*
\bigl(\mu_j\vert_{B_{\tau_0A}(\0)}\bigr)\right](S_j(a,p)).\]} 
{For \((a,p)\in\supp\chi_0\), the preceding inclusions give
\(S_j(a,p)\in B_A(\0)\). Whenever
\(\psi_i(S_j(b,q)^{-1}\cdot S_j(a,p))\ne0\), they also give
\[
S_j(b,q)\in B_{12R_{\dy}}(\0)
\subset B_{20R_{\dy}}(\0)
\subset B_{\tau_0A}(\0).
\]}
Thus, the restriction on $\mu_j$ may be removed without changing the
convolution on $\supp\chi_0$.  Pulling back by $S_j$ and applying
\eqref{fixedkernelL2small} gives
\[
\int
\left|
\int H_{j,i,v}((a,p),(b,q))\,d\sigma_j(b,q)
\right|^2 d\sigma_j(a,p)=o_j(1).
\]
It remains only to compare \(H_{j,i,v}\) with \(G_{j,i,v,M}\).  For
\((a,p)\in\supp\chi_0\), the support of the difference
\(H_{j,i,v}-G_{j,i,v,M}\) is contained, for all large \(j\), in the set where
\(q\in K_1'\) and \(|b|>M\).  On this set,
\[
|H_{j,i,v}((a,p),(b,q))-G_{j,i,v,M}((a,p),(b,q))|
\lesssim_{\chi_0,\psi_i,v} 1+|b| .
\]
The local mass bounds, together with \eqref{mastertail}, give
\[
\sup_j\sigma_j(\supp\chi_0)<\infty,
\qquad
\sup_j\sigma_j(\mathcal N\times K_1')<\infty.
\]
Therefore, by the Cauchy--Schwarz inequality, \eqref{mastersecondtail}, and \eqref{mastertail}
\[
\begin{aligned}
&\int
\left|
\int (H_{j,i,v}-G_{j,i,v,M})((a,p),(b,q))\,d\sigma_j(b,q)
\right|^2 d\sigma_j(a,p)  \\
&\qquad\lesssim_{\chi_0,\psi_i,v}
\sigma_j(\supp\chi_0)
\int_{(\mathcal N\setminus B_{\mathcal N}(0,M))\times K_1'}(1+|b|^2)\,d\sigma_j(b,q)
=o_j(1),
\end{aligned}
\]
because \(M>K_A\).  This proves \eqref{truncatedpulledbackL2}.

For fixed \(M\), the kernels \(G_{j,i,v,M}\) have common compact support in
\((\mathcal N\times P_0)^2\).  Moreover, by the computation of
\(S_j(b,q)^{-1}\cdot S_j(a,p)\),
\[
G_{j,i,v,M}((a,p),(b,q))
\longrightarrow
\chi_0(a,p)\chi_{1,M}(b,q)
\langle v,a-b\rangle
\psi_i(q^{-1}p)
\]
uniformly on this common compact support.  Writing $x=(a,p)$,
$y=(b,q)$, and $z=(c,r)$, the left hand side of
\eqref{truncatedpulledbackL2} is
\[
\iiint G_{j,i,v,M}(x,y)G_{j,i,v,M}(x,z)
\,d\sigma_j(x)d\sigma_j(y)d\sigma_j(z).
\]
The {weak convergence} $\sigma_j\rightharpoonup\sigma$ implies
\[
\sigma_j\times\sigma_j\times\sigma_j
\rightharpoonup
\sigma\times\sigma\times\sigma
\]
{on $(\mathcal N\times P_0)^3$}.  Hence the uniform convergence above
and \eqref{truncatedpulledbackL2} give
\[
\int
\left|
\int
\chi_0(a,p)\chi_{1,M}(b,q)
\langle v,a-b\rangle
\psi_i(q^{-1}p)\,d\sigma(b,q)
\right|^2
d\sigma(a,p)
=0.
\]
For $(a,p)\in\mathcal N\times B_{5R_{\dy}}^{P_0}$, put
\[
\mathcal V_{i,v}(a,p)
:=
\int_{\mathcal N\times B_{20R_{\dy}}^{P_0}}
\langle v,a-b\rangle
\psi_i(q^{-1}p)\,d\sigma(b,q).
\]
By \eqref{masterprojidentity}, {the cylinder $\mathcal N\times B_{20R_{\dy}}^{P_0}$} has finite
$\sigma$-mass, while \eqref{masterslabsupport} gives $|b|\leq K_A$ on its
support.  Dominated convergence therefore shows that $\mathcal V_{i,v}$ is
continuous.  By the choice of $K_1$ and because $M>K_A+1$, the inner
integral in the preceding $L^2$ identity equals
$\chi_0(a,p)\mathcal V_{i,v}(a,p)$.  Hence the continuous function
$\chi_0\mathcal V_{i,v}$ vanishes on
$\supp\sigma\cap(\mathcal N\times B_{5R_{\dy}}^{P_0})$.  Since $\chi_0=1$ near
the arbitrary compact set $K_0$, we conclude that
\[
\mathcal V_{i,v}(a,p)=0
\]
for every
$(a,p)\in\supp\sigma\cap(\mathcal N\times B_{5R_{\dy}}^{P_0})$.

{For each fixed $i$ with $\supp\psi_i\subset B_{6R_{\dy}}(\0)$, the preceding identity holds for every $v\in\mathcal N$. Hence}
\begin{equation}\label{limitcancelvector}
\int_{\mathcal N\times B_{20R_{\dy}}^{P_0}}
(a-b)\psi_i(q^{-1}p)\,d\sigma(b,q)=0
\end{equation}
{at every point of this support.}  By
Lemma~\ref{planeextensionlem}, for every vertical plane $P$
the restrictions of the chosen functions supported in $B_{6R_{\dy}}(\0)$
are dense among the horizontally even Lipschitz functions supported in
$B_{5R_{\dy}}^P$.  Applying this with $P=P_0$, choose
$\psi_{i_\ell}\vert_{P_0}\to\psi$ uniformly.  At every point in the displayed
support, \eqref{masterslabsupport} and \eqref{masterprojidentity} give
{
\begin{multline*}
\left|\int_{\mathcal N\times B_{20R_{\dy}}^{P_0}}
(a-b)\bigl(\psi_{i_\ell}(q^{-1}p)-\psi(q^{-1}p)\bigr)
\,d\sigma(b,q)\right|\\
\leq 2K_A\,\mu_0(B_{20R_{\dy}}^{P_0})
\|\psi_{i_\ell}\vert_{P_0}-\psi\|_\infty.
\end{multline*}
}
Letting $\ell\to\infty$ proves \eqref{limitcancelgoal} pointwise on the
support for every compactly supported horizontally even Lipschitz function
$\psi$ on $P_0$ with $\supp\psi\subset B_{5R_{\dy}}^{P_0}$.

\eqref{masterslabsupport}, 
\eqref{masterprojidentity}, 
\eqref{limitingflatproduct}, and {\eqref{limitcancelgoal}} are the
hypotheses of Lemma~\ref{coordproductcylrigidlem}, with $R=R_{\dy}$ and
$C_\sigma=K_A+1$.  Hence
$$
\supp\sigma\cap(\mathcal N\times B_{2R_{\dy}}^{P_0})
\subset
\mathcal G:=\{(Lw+a_0,w,t):(w,t)\in B_{2R_{\dy}}^{P_0}\}
$$
for some linear map $L:W\to\mathcal N$ and some $a_0\in\mathcal N$.
Here we have harmlessly enlarged the graph supplied by the lemma by
dropping the condition $t\in\supp\gamma$.  Let $\widetilde P_j$ be the
vertical plane obtained by applying $S_j$ to the full affine graph
$\{(Lw+a_0,w,t):(w,t)\in W\times\R\}$.  Equivalently,
$$
\widetilde P_j=\{(\xi_j+w_j+\beta_j(L_jw_j+a_{0,j}),t):w_j\in W_j,\ t\in\R\},
$$
where $L_j=O_jLO_j^{-1}:W_j\to W_j^\perp$ and $a_{0,j}=O_ja_0$.

Let $D(a,w,t)=|a-(Lw+a_0)|^2$.  Choose a compactly supported cutoff
$0\leq\Phi\leq1$ on $\mathcal N\times B_{2R_{\dy}}^{P_0}$ such that
$$
\Phi\equiv 1
\quad\text{on}\quad
B_{\mathcal N}(0,2K_A)\times B_{3r_{\dy}/2}^{P_0}.
$$
Since $\sigma$ is supported on $\mathcal G$ in
$\mathcal N\times B_{2R_{\dy}}^{P_0}$ and $\Phi D$ is continuous and
compactly supported,
$$
\int \Phi D\,d\sigma=0.
$$
By {weak convergence},
$$
\lim_{j\to\infty}\int \Phi D\,d\sigma_j=0.
$$
We now compare this cut-off integral with the integral over
\(B_{r_{\dy}}(\0)\) in the original variables.  Write
\[
S_j^{-1}\Y=(a_j(\Y),p_j(\Y)).
\]
The explicit formula for $p_j$ shows, independently of the normal
coordinate, that
\[
p_j\bigl(B_{r_{\dy}}(\0)\bigr)
\subset B_{3r_{\dy}/2}^{P_0}
\]
for all large $j$.  Indeed, orthogonal projection does not increase the
horizontal norm, while $O_j\to I$ and $\xi_j\to0$.  Put
\[
\mathcal B_j=S_j^{-1}\bigl(B_{r_{\dy}}(\0)\bigr),
\qquad
K=\overline B_{3r_{\dy}/2}^{P_0}.
\]
For all sufficiently large $j$,
$\mathcal B_j\subset\mathcal N\times K$.  On
$\mathcal B_j\cap\{|a|\leq K_A\}$ one has $\Phi=1$.  Therefore
\[
\int_{\mathcal B_j\cap\{|a|\leq K_A\}}D\,d\sigma_j
\leq\int\Phi D\,d\sigma_j\longrightarrow0.
\]
On the complementary part, $p\in K$ and
$D(a,w,t)\lesssim_K1+|a|^2$.  Therefore
\eqref{mastertail} and \eqref{mastersecondtail} give
\[
\int_{\mathcal B_j\cap\{|a|>K_A\}}D\,d\sigma_j
\lesssim_K
\int_{(\mathcal N\setminus B_{\mathcal N}(0,K_A))\times K}
(1+|a|^2)\,d\sigma_j(a,p)\longrightarrow0.
\]
Thus
\[
\int_{\mathcal B_j}D\,d\sigma_j\longrightarrow0.
\]
If $\Y=S_j(a,w,t)$, then 
\[
\dist(\Y,\widetilde P_j)\leq\beta_j|a-Lw-a_0|.
\]
Because $B_{r_{\dy}}(\0)\subset B_{\tau_0A}(\0)$, the definition of
$\sigma_j$ gives
\[
\int_{B_{r_{\dy}}(\0)}
\left(\frac{\dist(\Y,\widetilde P_j)}{\beta_j}\right)^2
\,d\mu_j(\Y)
\leq\int_{\mathcal B_j}D\,d\sigma_j\longrightarrow0.
\]
By \eqref{mastercubemass},
$\mu_j(B_{r_{\dy}}(\0))\geq\mu_j(E_j)\gtrsim1$.  Therefore, since
$\widetilde P_j$ is an admissible
competitor in the definition of $\beta_{\mu_j}(\0,r_{\dy})$,
$$
\beta_{\mu_j}(\0,r_{\dy})^2
\leq
\frac{1}{\mu_j(B_{r_{\dy}}(\0))}
\int_{B_{r_{\dy}}(\0)}
\left(\frac{\dist(\Y,\widetilde P_j)}{r_{\dy}}\right)^2d\mu_j(\Y)
=o(\beta_j^2).
$$
Since $\widehat\beta_j=\beta_{\mu_j}(E_j)
=\beta_{\mu_j}(\0,r_{\dy})$, we have
$\widehat\beta_j=o(\beta_j)$.  This contradicts
\eqref{masterbetacomparison}, which gives
$A^{-1}\beta_j\leq\widehat\beta_j$.  The contradiction proves the
proposition.
\end{proof}

\subsection{The flatness criterion}\label{modelcriterionsec}

\begin{thm}[Flatness criterion]\label{modelstepthm}
Let $s=k+2$, and let $\mu$ be an $s$-ADR measure satisfying the
horizontally antisymmetric Littlewood--Paley estimate \eqref{C2}.  Suppose
that, for every $\lambda>0$, the family
$$
\{E\in\dy(\mu):\alpha_{k,\mu}(\Z_E,\ell(E))>\lambda\}
$$
is Carleson.  Then $\mu$ satisfies the vertical $\beta$-Carleson estimate.
\end{thm}

\begin{proof}
Fix $A$ and the constants $\rho,\Delta$ from Proposition \ref{mastercylprop}.  Lemma~\ref{alphaenlarge}, applied with $M=A$ and
$\varepsilon=\rho/C_\alpha$, shows that the cubes for which
$$
\alpha_{k,\mu}(\Z_E,A\ell(E))>\rho
$$
form a Carleson family.  For the cubes such that $\alpha_{k,\mu}(\Z_E,A\ell(E)) \leq \rho$ and which are $A$-$\beta$-doubling cubes, Proposition \ref{mastercylprop} gives
$$
\beta_\mu(E)^2\mu(E)\lesssim
\max_{\varphi\in\mathcal K}\Theta_{\mu,\varphi,A}(E).
$$
Summing over $E\subset E_0$ and using Lemma \ref{CZcoef} for the finite family $\mathcal K$ gives a bound by $C\mu(E_0)$.  The contribution of the $\alpha_{k,\mu}$-bad cubes is bounded by their Carleson norm and the trivial estimate $\beta_\mu(E)\lesssim1$.  Finally Lemma \ref{dyuplem}, used with stopping scale $A$, removes the $A$-$\beta$-doubling restriction.  We have therefore proved
\[
\sum_{E\subset E_0}
\beta_\mu(\Z_E,r_{\dy}\ell(E))^2\mu(E)
\lesssim\mu(E_0).
\]
Since $r_{\dy}=3A_{\dy}>A_{\dy}$, Lemma~\ref{dyadicUR} gives the
continuous vertical $\beta$-Carleson estimate \eqref{flatsquare}.
\end{proof}

For $\eta=0$, the conclusion of Theorem~\ref{modelstepthm} is the
uniform rectifiability condition \eqref{flatsquare}; for
$\eta=\frac12\omega$, it is the corresponding vertical-plane estimate
in $\mathbb H^N$.

\section{Vampiric measures}\label{vampsec}

The flatness criterion reduces the main theorem to Carleson packing
estimates for the relevant $\alpha$-coefficients.  These packing estimates
are obtained from the structure of \emph{vampiric measures}.  We first
analyze such measures in the product case $\eta=0$, which forms the core of the general case.

A Borel measure $\mu$ on $\mathbb G_\eta$ is called \emph{vampiric}\footnote{These measures correspond to the local symmetry condition introduced by David and Semmes \cite{DS1}, and so it might be natural to refer to them as symmetric measures.  However, Mattila has already introduced a wider class of measures, also called symmetric measures, in the study of principal values of singular integrals.  We therefore use different terminology.} if, for every $\psi\in\Lipodd(\mathbb R^{n+1})$,
\[
\psi*\mu(\X)=\int_{\R^{n+1}}\psi(\Y^{-1}\cdot\X)\,d\mu(\Y)=0
\qquad\text{for every }\X\in\supp\mu.
\]

We shall use the following compactness observation repeatedly.  If $\mu_j$ are uniformly ADR vampiric measures and $\mu_j\rightharpoonup\mu$ {weakly}, then $\mu$ is vampiric.  Indeed, every point of $\supp\mu$ is the limit of points $\X_j\in\supp\mu_j$, and convolutions with compactly supported Lipschitz kernels pass to the limit.

We first derive an elementary consequence of the vampiric measure condition.  For $\X = (X,t)\in \R^{n+1}$, put $\overline{\X} = (-X,t)$.  If $U\subset \R^{n+1}$ is a bounded open set, then $\chi_U$ can be pointwise approximated by $\varphi_n\in \Lip_0(U)$ with $0\leq\varphi_n\leq \chi_U$.  Hence $f=\chi_U-\chi_{\overline U}$ is pointwise approximated by $\psi_n(\X)=\varphi_n(\X)-\varphi_n(\overline\X)$, with $|\psi_n|\leq\chi_U+\chi_{\overline U}$.  Dominated convergence gives, for vampiric $\mu$,
$$\mu(\{\Y:\Y^{-1}\cdot \X\in U\}) = \mu(\{\Y:\Y^{-1}\cdot \X\in \overline{U}\}),$$
or,
$$\mu(\X\cdot \{\Z^{-1}:\Z\in U\}) = \mu(\X\cdot \{\Z^{-1}:\Z\in \overline{U}\}).$$
Since $\X\mapsto\overline\X$ is a group automorphism of $\mathbb G_\eta$, we conclude that
\begin{equation}\label{vampcor}
    \mu(\X\cdot \{\Z^{-1}:\Z\in U\}) = \mu(\X\cdot \{\overline{\Z}^{-1}:\Z\in U\}).
\end{equation}

Finally, we record a simple result used in compactness arguments.

\begin{lem}\label{vampdense} Suppose $\mu$ satisfies
$\varphi*\mu\equiv 0\text{ on }\supp(\mu)$ for every $\varphi$ belonging to a dense subset of $\Lipodd(\R^{n+1})$.  Then $\mu$ is vampiric.
\end{lem}

\begin{lem}\label{squaretovampcompact}
Let $\{\varphi_i\}_{i=1}^\infty$ be a dense subset of $\Lipodd(\R^{n+1})$.  Suppose that $s>0$, that $\mu_j$ are uniformly $s$-ADR measures, $\0\in\supp\mu_j$, and $\mu_j\rightharpoonup\mu$ {weakly}.  Assume that for each fixed $i$ and each $R>0$,
$$
\int_{B_R(\0)} |\varphi_i*\mu_j(\X)|^2\,d\mu_j(\X)\longrightarrow0.
$$
Then $\mu$ is vampiric.
\end{lem}

\begin{proof}
Fix $i$ and $\X\in\supp\mu$.  We can find points $\X_j\in\supp\mu_j$ with $\X_j\to\X$, because, if such a sequence does not exist, some ball centered at $\X$ would have zero $\mu_j$-mass along a subsequence, contradicting {weak convergence} and the fact that $\X\in\supp\mu$.  On every fixed compact set the functions
$$
F_j(\Z)=\varphi_i*\mu_j(\Z)
$$
are uniformly bounded and uniformly Lipschitz, because $\varphi_i$ is compactly supported and Lipschitz and the measures are uniformly ADR.  If
$|F_j(\X_j)|>\delta$ along a subsequence, then $|F_j|>\delta/2$ on
$B_{c\delta}(\X_j)$, with $c$ independent of $j$.  For all large $j$,
\[
B_{c\delta/2}(\X)\subset B_{c\delta}(\X_j).
\]
Since $\X\in\supp\mu$, weak semicontinuity yields
\[
\liminf_j\mu_j(B_{c\delta}(\X_j))
\geq\mu(B_{c\delta/2}(\X))>0.
\]
It follows that, for $R$ large enough,
\[
\liminf_j\int_{B_R(\0)}|F_j|^2\,d\mu_j
\geq\frac{\delta^2}{4}\mu(B_{c\delta/2}(\X))>0,
\]
contradicting the assumed convergence to zero.  Hence $F_j(\X_j)\to0$.

On the other hand, the functions $\Y\mapsto\varphi_i(\Y^{-1}\cdot\X_j)$ converge uniformly on a fixed compact set to $\Y\mapsto\varphi_i(\Y^{-1}\cdot\X)$, and {weak convergence} gives
$$
F_j(\X_j)\longrightarrow \varphi_i*\mu(\X).
$$
Thus $\varphi_i*\mu(\X)=0$ for every $i$ and every $\X\in\supp\mu$.  Lemma \ref{vampdense} now implies that $\mu$ is vampiric.
\end{proof}

\subsection{The product case $\eta=0$}

In this subsection the group addition is the Euclidean addition, and $\mu$ is a vampiric measure which is $(k+2)$-ADR in the sense of \eqref{ADRdef}, where $0\leq k\leq n-1$.  We also fix a numerical constant $C_{\rm loc}\geq 24$ large enough that $B_{3r}(\0)\subset Q_{C_{\rm loc}r/6}(\0)$ for every $r>0$.  

The main goal of this subsection is the following proposition.  It says
that if $\alpha_{k,\mu}(\0,r)>0$, then at one of two controlled larger
scales $\mu$ is close to $\mathcal M_{k+1}$ or $\mathcal M_{k+2}$.

\begin{prop}[Product vampiric alternative]\label{vampprop} Let $0\leq k\leq n-1$.  For every $\lambda>0$ there exist $M_1,M_2\geq2$, depending also on $C_0$, such that the following holds.  Suppose that a vampiric measure $\mu$ is $C_0$-ADR of dimension $k+2$, $\0\in \supp(\mu)$, $r>0$, and
\begin{equation}\label{notplaneunit}
\alpha_{k,\mu}(\0,r)>0.
\end{equation}
Then, either
\begin{equation}\label{alphak1small}
\alpha_{k+1,\mu}(\0,M_1r)<\lambda
\end{equation}
or
\begin{equation}\label{alphak2small}
\alpha_{k+2,\mu}(\0,M_2r)<\lambda.
\end{equation}
When $k=n-1$, the second alternative is void under the convention $\alpha_{n+1,\mu}\equiv+\infty$, and hence the conclusion is the first alternative.
\end{prop}

The proposition also holds with $\0$ replaced by any support point, since all properties in the statement are translation invariant.  We use this observation routinely in this section.

The proof of this proposition will require a fair amount of preparatory work.  

\begin{lem}\label{reflections}
    Suppose $\X=(X,t)$ and $\Y=(Y,s)$ belong to $\supp(\mu)$.
    With $Z=Y-X$, the points $(X + 2\ell Z,t)$, $\ell \in \mathbb Z$,
    are contained in $\supp(\mu)$.
    Moreover, for all $0<r<\infty$ we have
    $$\mu(C_r(X+2\ell(Y-X),t)) = \mu(C_r(X,t)).$$
\end{lem}

\begin{proof}

For any cylinder $C_r(\xi,\tau)$ and any $\X=(X,t)\in\supp(\mu)$, we first have
\begin{equation}\label{measconv}
    \mu((X,t)+C_r(\xi,\tau)) = \mu((X,t)+C_r(-\xi,\tau)).
\end{equation}
This is a consequence of \eqref{vampcor} with $U=C_r(-\xi,-\tau)$.

With a view to arguing by induction, consider the following statement.  Suppose that $j\in \{0,1,2,\dots\}$ and that
$$
(X+2\ell(Y-X),t),\ (Y+2\ell(Y-X),s)\in\supp(\mu)
$$
for every $\ell\in\mathbb Z$ with $|\ell|\leq j$, and also
$$
\mu(C_r(X+2\ell(Y-X),t))=\mu(C_r(X,t))
$$
for every $\ell\in\mathbb Z$ with $|\ell|\leq j$.

This statement holds trivially for $j=0$.  Suppose now it holds for some fixed $j$.  Then
$$
C_r(X+2(j+1)(Y-X),t)=(Y+2j(Y-X),s) + C_r((Y-X),t-s),
$$
so equation \eqref{measconv}, with $(X,t)$ replaced by $(Y+2j(Y-X),s)\in \supp(\mu)$, $\xi=Y-X$, and $\tau= t-s$, yields
{
\begin{multline*}
\mu(C_r(X+2(j+1)(Y-X),t))\\
=\mu((Y+2j(Y-X),s)+C_r((X-Y),t-s))\\
=\mu(C_r(X+2j(Y-X),t))>0.
\end{multline*}
}
But now
$$C_r(X+2(j+1)(Y-X),t) = (X,t)+C_r(2(j+1)(Y-X),0),$$
so employing (\ref{measconv}), we deduce that
$\mu(C_r(X+2(j+1)(Y-X),t))=\mu(C_r(X-2(j+1)(Y-X),t))=\mu(C_r(X,t))>0.$
Interchanging the roles of $\X$ and $\Y$ in the above argument now yields that $(Y\pm 2(j+1)(Y-X),s)\in \supp(\mu)$.  This completes the induction step, and proves the lemma.
\end{proof}

\begin{lem}\label{reflections2}
    Suppose $\X=(X,t)$ and $\Y=(Y,t)$ belong to $\supp(\mu)$.
    With $Z=Y-X$, the points $(X + \ell Z,t)$, $\ell \in \mathbb Z$,
    are contained in $\supp(\mu)$.
\end{lem}
\begin{proof}
For $\ell\in \mathbb{Z}$, Lemma \ref{reflections} ensures that
$(X+2\ell(Y-X),t)\in \supp(\mu)$  and $(Y+2\ell(Y-X),t)=(X+(2\ell+1)(Y-X),t)\in \supp(\mu)$.
This proves the lemma.
\end{proof}

We need the following variant of Lemma \ref{reflections}.

\begin{lem}\label{reflections3}
Suppose $\X=(X,t)$ and $\Y=(Y,s)$ belong to $\supp(\mu)$, let $Z=Y-X$, let $r>0$, and let $I\subset\mathbb R$ be a bounded open interval.  With $\wt Q_r(W):=\{W'\in\mathbb R^n:\|W-W'\|_\infty<r\}$, for every $W\in\mathbb R^n$ and every $\ell\in\mathbb Z$,
$$
\mu(\wt Q_r(W+2\ell Z)\times I)=\mu(\wt Q_r(W)\times I).
$$
In particular,
$$
\mu(\wt Q_r(X+2\ell(Y-X))\times I)=\mu(\wt Q_r(X)\times I),
$$
and equivalently
$$
\mu(\wt Q_r(X+(2\ell+1)(Y-X))\times I)=\mu(\wt Q_r(Y)\times I).
$$
\end{lem}

\begin{proof}
Fix $(P,\tau)\in\supp(\mu)$ and set
$$
F(W):=\mu(\wt Q_r(W)\times I).
$$
We first record the identity
\begin{equation}\label{spatialreflectionidentity}
F(W)=F(2P-W),\qquad W\in\mathbb R^n.
\end{equation}
Indeed, writing $\xi=W-P$ and applying \eqref{vampcor} at $(P,\tau)$ with the relative time interval $I-\tau$ gives
$$
\mu((P,\tau)+\wt Q_r(\xi)\times(I-\tau))
=
\mu((P,\tau)+\wt Q_r(-\xi)\times(I-\tau)),
$$
which is exactly \eqref{spatialreflectionidentity}.

Apply \eqref{spatialreflectionidentity} first with $P=X$ and then with $P=Y$.  Then
$$
F(W)=F(2X-W)=F(2Y-(2X-W))=F(W+2(Y-X)).
$$
Iteration gives $F(W)=F(W+2\ell(Y-X))$ for every $\ell\in\mathbb Z$.  Taking $W=X$ gives the first displayed special case.  Taking $W=Y$ gives $F(Y)=F(Y+2\ell(Y-X))$, and $Y+2\ell(Y-X)=X+(2\ell+1)(Y-X)$, which gives the second.
\end{proof}

\begin{lem}\label{lattice}
Given an integer $\ell>0$ and $(\ell+1)$ points
$\Y_i=(Y_i,s_i)\in \supp(\mu)$, $i=0,\ldots,\ell$, for any
$a_1,\ldots,a_{\ell} \in 2^\ell \mathbb Z$ the point
$$\Bigl(Y_0+\sum_{i=1}^\ell a_i(Y_i-Y_0),s_0\Bigr)$$
is in $\supp(\mu)$.

\end{lem}
\begin{proof}
We use induction.  Without
loss of generality, we may assume $\Y_0=(Y_0,s_0)=\0$.  Notice that
the base case $\ell=1$ is a direct consequence of Lemma 
\ref{reflections}.

Assume the statement is known for $\ell=j$.  We show it holds
for $\ell=j+1$.  Choose $a_1,\ldots,a_{j+1} \in 2^{j+1} \mathbb Z$.  
By our induction hypothesis, the point 
$\bigl(\frac{1}{2}\sum_{i=1}^j a_iY_i,0\bigr)$ belongs to $\supp(\mu)$.  Moreover, by Lemma \ref{reflections}, 
$(-a_{j+1}Y_{j+1},0) \in \supp(\mu)$.  Thus an application of 
Lemma \ref{reflections} yields that
$\bigl(\sum_{i=1}^{j+1} a_iY_i,0 \bigr) \in \supp(\mu),$ as required.
\end{proof}

\begin{lem}\label{orthogplan}
Suppose a (spatial) linear subspace $V\subset \R^n$ and 
$\X_0=(X_0,t_0) \in \supp(\mu)$ are such that
$(V+X_0) \times \{t_0\} \subseteq \supp(\mu)$.
Then, for any $\Y=(Y,t_0) \in \supp(\mu)$, 
$(V+Y) \times \{t_0\} \subseteq \supp(\mu)$.
\end{lem}
\begin{proof}
Without loss of generality, assume $\X_0=\0$.  Fix $v\in V$.  Since $(v/2,0)$ and $(Y,0)$ belong to $\supp(\mu)$, Lemma \ref{reflections2} gives $(v-Y,0)\in\supp(\mu)$.  Thus $(V-Y)\times\{0\}\subset\supp(\mu)$.  In particular $(-v-Y,0)\in\supp(\mu)$; applying Lemma \ref{reflections2} to $(0,0)$ and $(-v-Y,0)$ gives $(v+Y,0)\in\supp(\mu)$.  Hence $(V+Y)\times\{0\}\subset\supp(\mu)$.
\end{proof}

Analyzing time slices will play an important role in our analysis.  Let $\pi_{sp}$ be the projection of 
$\mathbb R^n \times \mathbb R$ onto the spatial component $\mathbb R^n$.  Given $t \in \mathbb R$ and $S \subset \mathbb R^{n+1}$, we let
$S_t = \pi_{sp}(S \cap (\mathbb R^{n} \times \{t\}))$.
In particular, for $t\in \R$,
\[
\supp(\mu)_t = \{X\in \R^n: \,(X,t)\in \supp(\mu)\}.
\]

\begin{lem}\label{zerodistancelem2}
Suppose $\X_0=(X_0,t_0) \in \supp(\mu)$ and $V$ is a (spatial) linear 
subspace such that $(V + X_0) \times \{t_0\} \subseteq \supp(\mu)$.
If there exists a sequence of points
$\{\X_j=(X_j,t_j)\}_{j \geq 1} \subset \supp(\mu)$
such that $X_j \notin V + X_0$ and
\begin{equation}\label{zerodistance2}
\lim_{j \rightarrow \infty}
\dist\big((V+X_0),X_j\big)=0,
\end{equation}
then there exists a unit vector $e$ orthogonal to $V$ such that 
$$\big( X_0 + \operatorname{span}(e,V)\big)\times \{t_0\}\subseteq
\supp(\mu).$$
\end{lem}
\begin{proof}
    Without loss of generality, suppose $\X_0=(X_0,t_0)=\0$.  For each $j \geq 1$, let $
    Z_j$ be the point of $V$ closest to $X_j$, and
    set $d_j = |X_j-Z_j|$, $e_j:=(X_j-Z_j)/d_j.$  Then, perhaps after passing to a subsequence if necessary, $e_j \rightarrow e$ for some
    unit vector $e$ such that $e \perp V$.  By Lemma
    \ref{reflections}, reflecting $\X_j=(X_j,t_j)$ across points of $V\times\{0\}$ produces points of the form $(Z_j+2kd_je_j,0)$ in $\supp(\mu)$ for every $k\in\mathbb Z$.  {Applying Lemma~\ref{orthogplan} at
$(Z_j+2kd_je_j,0)$ and using $Z_j\in V$, we obtain
\[
(V+2kd_je_j)\times\{0\}\subset\supp(\mu)
\]
for every $k\in\mathbb Z$ and every $j\geq1$.}  Because
    $d_j \xrightarrow{j \rightarrow \infty} 0$ and
    $e_j \xrightarrow{j \rightarrow \infty} e$, we 
    conclude that $(V+\lambda e) \times \{0\} \subset \supp(\mu)$ for every $\lambda \in \mathbb R$.
\end{proof}

\begin{lem}\label{planeallslices}
Suppose that $\mu$ is a $(k+2)$-ADR vampiric measure, that
$\X=(X,t)\in\supp(\mu)$, and that $L$ is an $m$-dimensional spatial
plane with $X\in L$ and $k\leq m\leq k+2$.  Assume that
\[
\supp(\mu)_{t}\cap \wt{Q}_{3r}(X)=L\cap \wt{Q}_{3r}(X).
\]
Then there is a Radon measure $\nu$ on $\R$ such that
\[
\mu|_{Q_r(\X)}
=
\bigl(\mathcal H^m\vert_L\times\nu\bigr)|_{Q_r(\X)}.
\]
In particular,
$\mathcal H^m\vert_L\times\nu\in\mathcal M_m$, and hence
\[
\mu\in\mathcal M_m(\X,\rho)
\qquad\text{whenever }B_{3\rho}(\X)\subset Q_r(\X).
\]
\end{lem}

\begin{proof}
By translating, assume $X=0\in L$ and $t=0$.  Put
$$
T_r=\{\tau\in(-r^2,r^2):\supp(\mu)_\tau\cap \wt Q_r(0)\neq\varnothing\}.
$$
We first show that, for every $\tau\in T_r$,
\begin{equation}\label{localmut}
\supp(\mu)_\tau\cap \wt Q_r(0)=L\cap\wt Q_r(0).
\end{equation}
Indeed, if $Y\in\supp(\mu)_\tau\cap\wt Q_r(0)$ and $Y\notin L$, then reflecting $(0,0)$ through $(Y,\tau)$ gives $(2Y,0)\in\supp(\mu)$.  {Since $2Y\in\wt Q_{3r}(0)$ and $\supp(\mu)_0\cap\wt Q_{3r}(0)=L\cap\wt Q_{3r}(0)$, we have $2Y\in L$, hence $Y\in L$, a contradiction.}  Thus the left side of \eqref{localmut} is contained in $L\cap\wt Q_r(0)$.

Conversely, choose $X_\tau\in\supp(\mu)_\tau\cap\wt Q_r(0)$.  By the preceding paragraph, $X_\tau\in L$.  If $Z\in L\cap\wt Q_r(0)$, then $Y=(Z+X_\tau)/2$ belongs to $L\cap\wt Q_{3r}(0)\subset\supp(\mu)_0$.  Reflecting $(X_\tau,\tau)$ through $(Y,0)$ gives $(Z,\tau)\in\supp(\mu)$.  This proves \eqref{localmut}.

For $p\in L\cap\wt Q_{3r}(0)$, put
\[
\Sigma_p(y,s)=(2p-y,s).
\]
Equation~\eqref{vampcor}, first for bounded open sets and then for Borel
sets by regularity, gives $(\Sigma_p)_\#\mu=\mu$.  Consequently, if
$p,p'\in L\cap\wt Q_{3r}(0)$, then
\[
\Sigma_{p'}\circ\Sigma_p(y,s)=(y+2(p'-p),s)
\]
also preserves $\mu$.  The vectors $2(p'-p)$ obtained in this way contain
a neighborhood of the origin in $L$.  Since the set of spatial
translations preserving $\mu$ is an additive subgroup, every translation
\[
\Theta_h(y,s)=(y+h,s),
\qquad h\in L,
\]
preserves $\mu$.

For every bounded Borel set $I\subset(-r^2,r^2)$, define a locally finite
Borel measure on $L$ by
\[
\sigma_I(C)=\mu(C\times I).
\]
The preceding translation invariance implies
\[
\sigma_I=c(I)\mathcal H^m\vert_L
\]
for some $c(I)\geq0$.  Fix a compact Borel set $C_L\subset L$ with
$0<\mathcal H^m(C_L)<\infty$ and define
\[
\nu(I)=c(I)
=\frac{\mu(C_L\times I)}{\mathcal H^m(C_L)}.
\]
Countable additivity of $\mu$ shows that $\nu$ is a Radon measure on
$(-r^2,r^2)$.  Moreover,
\[
\nu((-r^2,r^2))
=\frac{\mu(C_L\times(-r^2,r^2))}{\mathcal H^m(C_L)}<\infty,
\]
so extending $\nu$ by zero outside $(-r^2,r^2)$ gives a Radon measure
on $\R$, which we continue to denote by $\nu$.  We have
\[
\mu(C\times I)=\mathcal H^m(C)\nu(I)
\]
{for bounded Borel rectangles in $L\times(-r^2,r^2)$.
Since locally finite Borel measures are determined by their
values on these rectangles, it follows that}
\[
\mu\vert_{L\times(-r^2,r^2)}
=\mathcal H^m\vert_L\times\nu.
\]
Since \eqref{localmut} gives
$\supp\mu\cap Q_r(\0)\subset L\times\mathbb R$, the product identity
follows in $Q_r(\0)$, and undoing the initial translation gives the
asserted identity in $Q_r(\X)$.

Since $\eta=0$ in this subsection, every spatial plane is
$\eta$-isotropic.  Thus the nonzero product measure
$\mathcal H^m\vert_L\times\nu$ belongs to $\mathcal M_m$.  Finally,
if $B_{3\rho}(\X)\subset Q_r(\X)$, then
\[
\mu\vert_{B_{3\rho}(\X)}
=
(\mathcal H^m\vert_L\times\nu)\vert_{B_{3\rho}(\X)}.
\]
{Since $\mu$ is globally $C_0$-ADR,
\eqref{localcomparisonclass} gives
$\mu\in\mathcal M_m(\X,\rho)$.}
\end{proof}

\begin{lem}\label{maxplane} Suppose $\mu$ is a $(k+2)$-ADR vampiric measure, and $\0\in \supp(\mu)$. Let $L\subset\supp(\mu)_0$ be the spatial linear subspace through $0$ with $m = \dim(L)$ maximal.  Then $k\leq m\leq k+2$.
\end{lem}

\begin{proof}
Suppose, for contradiction, that
$m<k$.  We first claim that every ball $B_r(\0)$ contains a point $\X=(X,t)\in\supp(\mu)$ with $X\notin L$.  Otherwise, for some $r>0$, $\supp \mu\cap B_r(\0)$ would be contained in $L\times\R$.  Covering $(L\times\R)\cap B_r(\0)$ by balls of radius $\rho$ gives
$$
\mu(B_r(\0))\lesssim \left(\frac r\rho\right)^{m+2}\rho^{k+2}=r^{m+2}\rho^{k-m}.
$$
Letting $\rho\downarrow0$ contradicts the lower ADR bound because $m<k$.  Hence we can produce a sequence $\{\X_j=(X_j,t_j)\}_{j \geq 1}
\subseteq \supp(\mu)$ such that $X_j \notin L$ for each $j$, and so that $\X_j$ converges to $\0$.  An application of Lemma 
\ref{zerodistancelem2} then shows that $L$ cannot be of maximal dimension, a contradiction.

If instead $m:=\dim L>k+2$, then for each integer $j\geq0$ we can pack in $B_1(\0)\cap\supp(\mu)$ approximately $2^{mj}$ disjoint balls $\{B^j_\ell\cap\supp(\mu)\}_{\ell\geq1}$ of radius $2^{-j}$ centered on $\supp(\mu)$.  Hence
$$\mu(B_1(\0)) \geq \sum_\ell \mu(B^j_\ell )\gtrsim 2^{m j} 2^{-j(k+2)}.$$
Because $m > k+2$, the right-hand side tends to infinity as $j\to \infty$, which is absurd.
\end{proof}

\begin{lem}\label{planestructure}
Suppose $\mu$ is a $(k+2)$-ADR vampiric measure with $\0\in \supp(\mu)$.  Then there is an $m$-dimensional linear subspace
$$
L\subset\supp(\mu)_0,
\qquad m\in\{k,k+1,k+2\}\cap\{0,\ldots,n\},
$$
and there is a possibly empty linearly independent family $e_1,\ldots,e_q\in L^\perp$, with $m+q\leq k+2$, such that
\begin{enumerate}
\item[(a)] $\supp(\mu)_0\subset \operatorname{span}(L,e_1,\dots,e_q)$;
\item[(b)] for each $j=1,\ldots,q$,
$$
\supp(\mu)_0\cap \wt Q_{|e_j|/\sqrt n}(0)
\subset \operatorname{span}(L,e_1,\ldots,e_{j-1})\cap \wt Q_{|e_j|/\sqrt n}(0),
$$
where the span is interpreted as $L$ when $j=1$;
\item[(c)] $L+\sum_{j=1}^q a_je_j\subset\supp(\mu)_0$ for every $(a_1,\ldots,a_q)\in(2^q\mathbb Z)^q$.
\end{enumerate}
\end{lem}

\begin{proof}
Let $L \subset \mathbb R^n$ 
be the spatial linear subspace of maximal dimension such that $L \subset \supp(\mu)_0$.  
By Lemma \ref{maxplane}, $\dim(L)=m\in \{k,k+1,k+2\}\cap\{0,\ldots,n\}$.

If $\supp(\mu)_0=L$, we take $q=0$.  Then (a), (b), and (c) are trivial.  Since $m\leq k+2$, this also gives $m+q\leq k+2$.
Assume from now on that $\supp(\mu)_0\setminus L\neq \varnothing$.  We first record the following fact: if $F\subset L^\perp$ is a linear subspace and
$$
A(F):=(\supp(\mu)_0\cap L^\perp)\setminus F
$$
is nonempty, then the function $e\mapsto |e|$ attains its minimum on $A(F)$.

Indeed, let $e_i\in A(F)$ be a sequence such that $|e_i|$ tends to $\inf_{e \in A(F)} |e|$.  Passing to a subsequence, we may assume that $e_i\to e_\infty\in\supp(\mu)_0\cap L^\perp$.  If $e_\infty\notin F$, then the minimum is attained.  Suppose instead that $e_\infty\in F$.  Since $e_\infty\in\supp(\mu)_0$, Lemma \ref{orthogplan} gives
$$
(L+e_\infty)\times\{0\}\subset\supp(\mu).
$$
Moreover, $e_i\notin L+e_\infty$ and $\dist(e_i,L+e_\infty)=|e_i-e_\infty|\to0$.  Lemma \ref{zerodistancelem2}, applied to the plane $L+e_\infty$, gives a unit vector $u\perp L$ such that
$$
(e_\infty+\operatorname{span}(L,u))\times\{0\}\subset\supp(\mu).
$$
Applying Lemma \ref{orthogplan} once more, now translating this plane to the origin using $\0\in\supp(\mu)$, gives
$$
\operatorname{span}(L,u)\times\{0\}\subset\supp(\mu),
$$
contradicting the maximality of $L$.  Hence $e_\infty\notin F$, and the minimum is attained.

We now choose the vectors successively.  Let $F_0=\{0\}$.  If $A(F_0)=\varnothing$, set $q=0$.  Otherwise choose $e_1\in A(F_0)$ of minimal length.  Having chosen $e_1,\ldots,e_\ell$, put $F_\ell=\operatorname{span}(e_1,\ldots,e_\ell)$.  If $A(F_\ell)=\varnothing$, stop; otherwise choose $e_{\ell+1}\in A(F_\ell)$ of minimal length.  Since the vectors are linearly independent and in $L^\perp$, this process stops after at most $n-m$ steps.  Denote the resulting vectors by $e_1,\ldots,e_q$.

The minimality of the $|e_j|$ gives property (b).  Indeed, suppose $X\in\supp(\mu)_0\cap\wt Q_{|e_j|/\sqrt n}(0)$ and $X\notin\operatorname{span}(L,e_1,\ldots,e_{j-1})$.  By Lemma \ref{orthogplan}, $X+L\subset\supp(\mu)_0$, so the orthogonal projection $X_{L^\perp}$ also belongs to $\supp(\mu)_0\cap L^\perp$.  This projection is not in $\operatorname{span}(e_1,\ldots,e_{j-1})$, while
$$
|X_{L^\perp}|\leq |X|<|e_j|,
$$
because $X\in\wt Q_{|e_j|/\sqrt n}(0)$.  But then we could have chosen
$X_{L^\perp}$ instead of $e_j$, a contradiction.

Lemma \ref{orthogplan} ensures that if $X\in\supp(\mu)_0$, then $X+L\subset\supp(\mu)_0$.  Hence, if some $X\in\supp(\mu)_0$ were outside $\operatorname{span}(L,e_1,\dots,e_q)$, its projection to $L^\perp$ would give a point of $\supp(\mu)_0$ linearly independent of $e_1,\dots,e_q$, contradicting the termination of the algorithm described above.  Thus (a) holds.

Since $L$ is a linear subspace, Lemma \ref{lattice} ensures that
$$
L+\sum_{j=1}^q a_je_j \subset \supp(\mu)_0
\qquad\text{for all }a_j\in2^q\mathbb Z.
$$

It remains to check that $q\leq k+2-m$.  There is nothing to prove if
$q=0$, so assume $q\geq1$.  Since $e_1,\ldots,e_q$ are linearly
independent, the lattice
\[
\Lambda_e=
\left\{\sum_{j=1}^q a_je_j:(a_1,\ldots,a_q)\in(2^q\mathbb Z)^q\right\}
\subset L^\perp
\]
{is unbounded, and there exists $d_e>0$ such that
\[
|x-y|\geq d_e
\qquad\text{for all distinct }x,y\in\Lambda_e.
\]}  Fix
$0<\kap<d_e/10$.  For $K>1$, put
\[
A_K=(2^q\mathbb Z)^q\cap[-K,K]^q.
\]
Then $\#A_K\simeq K^q$.

Let $\mathcal Z_K$ be a maximal $\kap$-separated subset of
$L\cap B_L(0,K)$.  Maximality gives
\[
\#\mathcal Z_K\gtrsim(K/\kap)^m.
\]
By property~(c), all points in
\[
\mathcal C_K=
\left\{
\left(z+\sum_{j=1}^q a_je_j,0\right):
z\in\mathcal Z_K,\ (a_1,\ldots,a_q)\in A_K
\right\}
\]
belong to $\supp\mu$.  Distinct points of $\mathcal C_K$ are separated
by at least $\kap$, and $\mathcal C_K\subset B_{C_eK}(\0)$, where $C_e$
depends only on $e_1,\ldots,e_q$.  Hence the balls
\[
\{B_{\kap/3}(\X):\X\in\mathcal C_K\}
\]
are pairwise disjoint and contained in $B_{C'_eK}(\0)$.  The lower ADR
bound gives
\[
\mu(B_{C'_eK}(\0))
\gtrsim K^{m+q}\kap^{-m}\kap^{k+2}
=K^{m+q}\kap^{k+2-m}.
\]
On the other hand, the upper ADR bound gives
\[
\mu(B_{C'_eK}(\0))\lesssim_eK^{k+2}.
\]
For fixed $\kap$, these inequalities are incompatible as $K\to\infty$
if $m+q>k+2$.  Thus $q\leq k+2-m$.
\end{proof}

We now have the ingredients for the proof of Proposition \ref{vampprop}.  The cases depend on the number of transverse vectors $\{e_1,\dots,e_q\}$ in Lemma \ref{planestructure}, and the compactness arguments for one and two transverse directions are recorded separately.

We need the following lemma about stability of the alpha numbers under
{weak convergence}.  {The following lemma holds for every $\eta$ satisfying
\eqref{etahypothesis}.}

\begin{lem}[{Weak continuity of \(\alpha\)-numbers}]\label{alphaweak}
Suppose that \(\{\mu_j\}\) is a sequence of $C_0$-ADR measures of
dimension $s$, \(0\in \supp\mu_j\), and
\[
   \mu_j \rightharpoonup \mu
\]
{weakly} as Radon measures.  Fix an integer $m$.  Then, for every
\(\X\) and every \(r>0\) such that
$\mathcal M_m(\X,r)\neq\varnothing$,
\[
   \alpha_{m,\mu_j}(\X,r)
   \longrightarrow
   \alpha_{m,\mu}(\X,r).
\]
\end{lem}

\begin{proof}
Fix \(\X\) and \(r>0\). Set
\[
d_{\X,r}(\nu_1,\nu_2)
=
\sup_{\substack{\supp f\subset B_{3r}(\X)\\
                 \|f\|_{\operatorname{Lip}}\le 1}}
\left|\int f\,d(\nu_1-\nu_2)\right|.
\]
$\mathcal M_m(\X,r)$ is the same fixed comparison class for every $j$.
From the definition of \(\alpha_{m,\nu}(\X,r)\), we have
\[
\left|
\alpha_{m,\mu_j}(\X,r)
-\alpha_{m,\mu}(\X,r)
\right|
\leq
r^{-s-1} d_{\X,r}(\mu_j,\mu).
\]
It is therefore enough to prove that \(d_{\X,r}(\mu_j,\mu)\to0\).

Choose a compact set $K$ whose interior contains
$\overline{B_{3r}(\X)}$.  {Weak convergence} gives
\[
\sup_j\mu_j(K)+\mu(K)<\infty.
\]

Let \(\varepsilon>0\). Choose a finite \(\varepsilon\)-net
\(f_1,\dots,f_N\) in the collection of $f\in \Lip_0(B_{3r}(\X))$ with $\|f\|_{\Lip}\leq 1$. By {weak convergence},
\[
   \int f_i\,d\mu_j\to \int f_i\,d\mu
   \qquad\text{for each }i=1,\dots,N.
\]
Thus, for all sufficiently large \(j\),
\[
   \left|\int f_i\,d(\mu_j-\mu)\right|<\varepsilon
   \qquad\text{for all }i=1,\dots,N.
\]
Given any admissible test function \(f\), choose \(f_i\) with
\(\|f-f_i\|_\infty<\varepsilon\). Then
\[
\left|\int f\,d(\mu_j-\mu)\right|
\leq
\left|\int f_i\,d(\mu_j-\mu)\right|
+
\varepsilon\bigl(\mu_j(K)+\mu(K)\bigr).
\]
The right-hand side is \(O(\varepsilon)\), uniformly in \(f\), for all large
\(j\). Therefore
\[
d_{\X,r}(\mu_j,\mu)\to0.
\]
Consequently,
\[
   \alpha_{m,\mu_j}(\X,r)
   \longrightarrow
   \alpha_{m,\mu}(\X,r),
\]
as claimed.
\end{proof}
\begin{lem}\label{oneorthogonal}
For every $\lambda>0$ and $C_0\geq1$ there exists $M_1\geq2$ such that the following holds for every $M\geq M_1$.  Suppose $\mu$ is a $C_0$-ADR vampiric measure of dimension $k+2$ such that $\0\in\supp(\mu)$, $L\subset\supp(\mu)_0$ is an $m$-dimensional linear subspace with $m\in\{k,k+1\}$, and $e\perp L$ is nonzero with $|e|\leq C_{\rm loc}\sqrt n$.  Assume also that
\begin{equation}\begin{split}
\label{localplansupport}
\wt Q_{C_{\rm loc}M}(0)\cap \operatorname{span}(L,e)
&\supset \wt Q_{C_{\rm loc}M}(0)\cap\supp(\mu)_0 \\
&\supset \wt Q_{C_{\rm loc}M}(0)\cap\bigcup_{a\in4\mathbb Z}(L+ae).
\end{split}\end{equation}
Then
$$
\alpha_{m+1,\mu}(\0,M)<\lambda.
$$
Here $M_1=M_1(\lambda,C_0)$.
\end{lem}

\begin{proof}
Suppose not.  Then there exists $\lambda>0$ and a sequence of $(k+2)$-ADR vampiric measures $\{\wt\mu_\ell\}_{\ell \in \mathbb N}$ with the following properties.
\begin{itemize}
	\item Each measure in the sequence is $C_0$-ADR.
    \item There are integers $m_\ell\in\{k,k+1\}$, $m_\ell$-dimensional subspaces $L_\ell\subset\supp(\wt\mu_\ell)_0$, and nonzero vectors $e_\ell\perp L_\ell$ with $|e_\ell|\leq C_{\rm loc}\sqrt n$.
    \item There is an increasing sequence $M_\ell\to\infty$ such that $\alpha_{m_\ell+1,\wt\mu_\ell}(\0,M_\ell)\geq \lambda$.
    \item For each $\ell$, display \eqref{localplansupport} holds with the corresponding objects indexed by $\ell$.
\end{itemize}

Since $m_{\ell}\in \{k, k+1\}$, we may pass to a subsequence and assume that $m_{\ell}$ is constant, equal to $m$.  Passing to a further subsequence, we may also assume that $e_{\ell}/|e_{\ell}|\to e\in \R^n$ as $\ell \to \infty$, and that the planes $L_{\ell}$ converge to some plane $L$ in the local Hausdorff metric.

Now consider the measures $\mu_{\ell}=M_\ell^{-(k+2)}(\delta_{1/M_\ell})_\#\wt\mu_{\ell}$.  These measures are $(k+2)$-ADR, vampiric, and satisfy $\0\in\supp(\mu_\ell)$.  Since $4|e_\ell|/M_\ell\to0$, for all sufficiently large $\ell$ we may choose $a_\ell\in4\mathbb Z$ so that
\[
1\leq \frac{|a_\ell e_\ell|}{M_\ell}\leq2.
\]
The corresponding point $(a_\ell/M_\ell)e_\ell$ belongs to
$\wt Q_{C_{\rm loc}}(\0)\cap\supp(\mu_\ell)_0$ by
\eqref{localplansupport}.  Since the support contains both this point and
$\0$ at mutual distance at least one, its diameter is at least one and
the lower ADR estimate at $\0$ gives
\[
\mu_\ell(B_{1/2}(\0))\geq C_0^{-1}2^{-(k+2)}.
\]
By passing to another subsequence and using Lemma~\ref{adrwl}, we may assume that they converge {weakly} to a $(k+2)$-ADR vampiric measure $\mu$.  Put
$$
P_\ell=\operatorname{span}(L_\ell,e_\ell).
$$
After passing to a further subsequence, the planes $P_\ell$ converge to $P=\operatorname{span}(L,e)$ in the Grassmannian.  For the rescaled measures, \eqref{localplansupport} gives
{
\[
\begin{aligned}
\wt Q_{C_{\rm loc}}(\0)\cap\supp(\mu_\ell)_0
&\subset\wt Q_{C_{\rm loc}}(\0)\cap P_\ell,\\
\wt Q_{C_{\rm loc}}(\0)\cap\supp(\mu_\ell)_0
&\supset\wt Q_{C_{\rm loc}}(\0)\cap
\bigcup_{a\in4\mathbb Z}\left(L_\ell+\frac{a}{M_\ell}e_\ell\right).
\end{aligned}
\]
}
We claim that
$$
\wt Q_{C_{\rm loc}/2}(\0)\cap\supp(\mu)_0=\wt Q_{C_{\rm loc}/2}(\0)\cap P.
$$
First let $Z\in \wt Q_{C_{\rm loc}/2}(\0) \cap P$.  Choose $Z_\ell\in P_\ell$ with $Z_\ell\to Z$.  Since the sets
$$
\bigcup_{a\in4\mathbb Z}\left(L_\ell+\frac{a}{M_\ell}e_\ell\right)
$$
have mesh tending to zero in the $e_\ell$ direction, and $L_\ell\to L$, the bound $|e_\ell|\leq C_{\rm loc}\sqrt n$ allows us to choose $Y_\ell\in\supp(\mu_\ell)_0$ with $Y_\ell\to Z$.  The uniform lower ADR bounds and {weak convergence} then imply $Z\in\supp(\mu)_0$.

Conversely, suppose $Z\in\supp(\mu)_0\cap\wt Q_{C_{\rm loc}/2}(\0)$.  Choose $(Z_\ell,t_\ell)\in\supp(\mu_\ell)$ with $(Z_\ell,t_\ell)\to(Z,0)$.  By Lemma \ref{reflections}, applied to $(\0,0)$ and $(Z_\ell,t_\ell)$, the point $(2Z_\ell,0)$ belongs to $\supp(\mu_\ell)$.  For $\ell$ large, $2Z_\ell\in\wt Q_{C_{\rm loc}}(\0)$, and therefore the containment above gives $2Z_\ell\in P_\ell$.  Hence $Z_\ell\in P_\ell$, and passing to the limit yields $Z\in P$.  This proves the claim.

By Lemma~\ref{planeallslices}, applied with radius $C_{\rm loc}/6$,
there is a measure $\sigma\in\mathcal M_{m+1}$ which agrees with $\mu$
on $B_3(\0)$.  Since $\mu$ is $C_0$-ADR, this gives
\[
\mu\in\mathcal M_{m+1}(\0,1),
\qquad
\alpha_{m+1,\mu}(\0,1)=0.
\]
Lemma \ref{alphaweak} therefore gives
$$
\alpha_{m+1,\mu_\ell}(\0,1)\to0,
$$
contradicting the fact that $\alpha_{m+1,\mu_\ell}(\0,1)\geq\lambda$.
\end{proof}

\begin{lem}\label{twoorthogonal}
For every $\lambda>0$ and $C_0\geq1$ there exists $M_2\geq2$ such that the following holds for every $M\geq M_2$.  Suppose $\mu$ is a $C_0$-ADR vampiric measure of dimension $k+2$ with $\0\in\supp(\mu)$, $L\subset\supp(\mu)_0$ is $k$-dimensional, and $e_1,e_2\perp L$ are linearly independent with $|e_1|\leq C_{\rm loc}\sqrt n$.  Assume that
\begin{equation}\label{2veccontain}
\supp(\mu)_0\cap \wt Q_{|e_2|/\sqrt n}(0)
\subset \operatorname{span}(L,e_1)\cap \wt Q_{|e_2|/\sqrt n}(0),
\end{equation}
\begin{equation}\label{2vecplane}
\supp(\mu)_0\subset \operatorname{span}(L,e_1,e_2),
\end{equation}
and
\begin{equation}\label{2vecstruc2}
\supp(\mu)_0\supset \bigcup_{a_1,a_2\in4\mathbb Z}(L+a_1e_1+a_2e_2).
\end{equation}
Then
\begin{equation}\label{2vecconc}
\alpha_{k+1,\mu}(\0,M_1)<\lambda
\quad\text{or}\quad
\alpha_{k+2,\mu}(\0,M)<\lambda,
\end{equation}
where $M_1=M_1(\lambda,C_0)$ is the constant from Lemma \ref{oneorthogonal}.
\end{lem}
    
\begin{proof}
Suppose that the statement is false.  Then there exists $\lambda>0$ and a sequence of $(k+2)$-ADR vampiric measures $\{\wt\mu_\ell\}_{\ell \in \mathbb N}$ with the following properties.
\begin{itemize}
    \item Each $\wt \mu_\ell$ is $C_0$-ADR.
    \item There is an increasing sequence $\{M_{2,\ell}\}_{\ell \in \mathbb N}$ tending to infinity such that, for each $\ell \in \mathbb N$,
    $$\alpha_{k+1,\wt\mu_\ell}(\0,M_1(\lambda)) \geq \lambda \text{ and } \alpha_{k+2,\wt\mu_\ell}(\0,M_{2,\ell}) \geq \lambda.$$
    \item There are $k$-dimensional subspaces $L_\ell\subset\supp(\wt\mu_\ell)_0$ and vectors $e_{1,\ell},e_{2,\ell}\perp L_\ell$ with $|e_{1,\ell}|\leq C_{\rm loc}\sqrt n$.
    \item Displays \eqref{2veccontain}--\eqref{2vecstruc2} hold with $\mu$, $L$, $e_1$, and $e_2$ replaced by $\wt \mu_\ell$, $L_{\ell}$, $e_{1,\ell}$, $e_{2,\ell}$.
\end{itemize}
If $|e_{2,\ell}| > C_{\rm loc}M_1 \sqrt{n}$ for infinitely many $\ell$, then, for such $\ell$, display \eqref{2veccontain} gives
$$
\supp(\wt \mu_\ell)_0 \cap \wt Q_{C_{\rm loc}M_1}(\0) \subset \operatorname{span}(L_\ell,e_{1,\ell}) \cap \wt Q_{C_{\rm loc}M_1}(\0),
$$
{while \eqref{2vecstruc2}, with $a_2=0$, gives
\[
\supp(\wt\mu_\ell)_0
\supset
\bigcup_{a_1\in4\mathbb Z}(L_\ell+a_1e_{1,\ell}).
\]}  Lemma \ref{oneorthogonal} would then imply
$$
\alpha_{k+1,\wt\mu_\ell}(\0,M_1)<\lambda,
$$
contradicting the first inequality in the second item above.  Hence, after passing to a subsequence, we may assume that $|e_{2,\ell}|\leq C_{\rm loc}M_1\sqrt n$ for every $\ell$.

We now repeat the compactness argument from Lemma \ref{oneorthogonal}.  Consider the measures
$$
\mu_{\ell}=M_{2,\ell}^{-(k+2)}(\delta_{1/M_{2,\ell}})_\#\wt\mu_{\ell}.
$$
These measures are $(k+2)$-ADR, vampiric, and satisfy $\0\in\supp(\mu_\ell)$.  Since $4|e_{1,\ell}|/M_{2,\ell}\to0$, for all sufficiently large $\ell$ we may choose $a_\ell\in4\mathbb Z$ so that
\[
1\leq \frac{|a_\ell e_{1,\ell}|}{M_{2,\ell}}\leq2.
\]
By \eqref{2vecstruc2},
\[
\frac{a_\ell}{M_{2,\ell}}e_{1,\ell}\in\supp(\mu_\ell)_0.
\]
{Since $\0\in\supp(\mu_\ell)$ and
$|a_\ell e_{1,\ell}|/M_{2,\ell}\geq1$,
the support has diameter at least one.
The lower ADR estimate therefore gives}
\[
\mu_\ell(B_{1/2}(\0))\geq C_0^{-1}2^{-(k+2)}.
\]
By passing to a subsequence and using Lemma~\ref{adrwl}, we may assume that they converge {weakly} to a $(k+2)$-ADR vampiric measure $\mu$.  Set
$$
P_\ell=\operatorname{span}(L_\ell,e_{1,\ell},e_{2,\ell}).
$$
After passing to a further subsequence, the $(k+2)$-planes $P_\ell$ converge in the Grassmannian to a $(k+2)$-plane $P$.  Then, \eqref{2vecplane} gives
$$
\supp(\mu_\ell)_0\subset P_\ell,
$$
and \eqref{2vecstruc2} gives
$$
\supp(\mu_\ell)_0\supset
\bigcup_{a_1,a_2\in4\mathbb Z}\left(L_\ell+\frac{a_1}{M_{2,\ell}}e_{1,\ell}+\frac{a_2}{M_{2,\ell}}e_{2,\ell}\right).
$$
The vectors $e_{1,\ell}$ and $e_{2,\ell}$ are uniformly bounded, and $M_{2,\ell}\to\infty$.  We claim that
$$
\supp(\mu)_0\cap \wt Q_{C_{\rm loc}/2}(\0)=P\cap \wt Q_{C_{\rm loc}/2}(\0).
$$
If $Z\in P\cap\wt Q_{C_{\rm loc}/2}(\0)$, choose $Z_\ell\in P_\ell$
with $Z_\ell\to Z$.  Write
\[
Z_\ell=V_\ell+b_{1,\ell}e_{1,\ell}+b_{2,\ell}e_{2,\ell},
\qquad V_\ell\in L_\ell.
\]
Choose $a_{i,\ell}\in4\mathbb Z$ so that
\[
\left|\frac{a_{i,\ell}}{M_{2,\ell}}-b_{i,\ell}\right|
\leq\frac2{M_{2,\ell}},
\qquad i=1,2.
\]
Then
\[
Y_\ell
=V_\ell+\frac{a_{1,\ell}}{M_{2,\ell}}e_{1,\ell}
+\frac{a_{2,\ell}}{M_{2,\ell}}e_{2,\ell}
\in\supp(\mu_\ell)_0
\]
and, by the uniform bounds on $e_{1,\ell},e_{2,\ell}$,
\[
|Y_\ell-Z_\ell|
\leq\frac{2(|e_{1,\ell}|+|e_{2,\ell}|)}{M_{2,\ell}}
\longrightarrow0.
\]
Thus $Y_\ell\to Z$, and the uniform lower ADR bounds and {weak convergence}
give $Z\in\supp(\mu)_0$.  Conversely, if
$Z\in\supp(\mu)_0\cap\wt Q_{C_{\rm loc}/2}(\0)$, choose
$(Z_\ell,t_\ell)\in\supp(\mu_\ell)$ with
$(Z_\ell,t_\ell)\to(Z,0)$.  Lemma \ref{reflections}, applied to
$(\0,0)$ and $(Z_\ell,t_\ell)$, gives
$(2Z_\ell,0)\in\supp(\mu_\ell)$.  For $\ell$ large,
$2Z_\ell\in\wt Q_{C_{\rm loc}}(\0)$, and therefore the containment
$\supp(\mu_\ell)_0\subset P_\ell$ gives $2Z_\ell\in P_\ell$.  Hence
$Z\in P$ after passing to the limit.  This proves the claim.

By Lemma~\ref{planeallslices}, applied with radius $C_{\rm loc}/6$,
there is a measure $\sigma\in\mathcal M_{k+2}$ which agrees with $\mu$
on $B_3(\0)$.  Since $\mu$ is $C_0$-ADR, this gives
\[
\mu\in\mathcal M_{k+2}(\0,1),
\qquad
\alpha_{k+2,\mu}(\0,1)=0.
\]
Lemma \ref{alphaweak} therefore gives
$$
\alpha_{k+2,\mu_\ell}(\0,1)\to0,
$$
contradicting that $\alpha_{k+2,\mu_\ell}(\0,1)\geq\lambda$.
\end{proof}

We may now complete the proof of Proposition \ref{vampprop}.

\begin{proof}[Proof of Proposition \ref{vampprop}] 
The proof is a case analysis based on Lemma \ref{planestructure}.  We rescale to assume, without loss of generality, that $r=1$.

Suppose $\mu$ is a $(k+2)$-ADR vampiric measure with $\alpha_{k,\mu}(\0,1)>0$.  Let $L\subset\supp(\mu)_0$, of dimension $m\in\{k,k+1,k+2\}\cap\{0,\ldots,n\}$, and the family $e_1,\ldots,e_q$ be supplied by Lemma \ref{planestructure}.  When $k=n-1$, the second alternative in the proposition is void.  In that case, if $\dim L=k$, there cannot be two linearly independent vectors orthogonal to $L$ in $\mathbb R^n$; if $\dim L=k+1=n$, then no nonzero vector is orthogonal to $L$.  Thus the argument below only produces the first alternative, as claimed.

Suppose first that $\dim L=k$ (so $L=\{0\}$ if $k=0$).  We must have
$q\geq1$.  Indeed, if $q=0$, property~(a) gives $\supp(\mu)_0=L$, and
Lemma~\ref{planeallslices}, applied with radius $C_{\rm loc}/6$, gives a
measure $\sigma\in\mathcal M_k$ agreeing with $\mu$ on $B_3(\0)$.
Thus $\mu\in\mathcal M_k(\0,1)$ and
$\alpha_{k,\mu}(\0,1)=0$.  Since $m+q\leq k+2$, we have
$q\in\{1,2\}$.  Moreover,
\[
|e_1|<C_{\rm loc}\sqrt n.
\]
Otherwise property~(b) would give
\[
\supp(\mu)_0\cap\wt Q_{C_{\rm loc}}(\0)
=L\cap\wt Q_{C_{\rm loc}}(\0),
\]
and the same application of Lemma~\ref{planeallslices} would again give
$\mu\in\mathcal M_k(\0,1)$, contradicting
$\alpha_{k,\mu}(\0,1)>0$.

If $q=1$, properties~(a) and~(c) give \eqref{localplansupport} for $e_1$;
the $2\mathbb Z$ lattice in property~(c) contains the required
$4\mathbb Z$ lattice.  Lemma~\ref{oneorthogonal} therefore gives
\[
\alpha_{k+1,\mu}(\0,M_1)<\lambda.
\]
If $q=2$, properties~(a)--(c), with the vectors $e_1,e_2$,
are precisely the hypotheses \eqref{2veccontain}--\eqref{2vecstruc2} of
Lemma~\ref{twoorthogonal}.  That lemma gives the required alternative.

Next suppose $\dim L=k+1$.  If
$\supp(\mu)_0\cap \wt{Q}_{C_{\rm loc}M_1}(\0)
=L\cap \wt{Q}_{C_{\rm loc}M_1}(\0)$, then
Lemma~\ref{planeallslices}, applied with radius $C_{\rm loc}M_1/6$,
gives a measure $\sigma\in\mathcal M_{k+1}$ agreeing with $\mu$ on
$B_{3M_1}(\0)$.  Thus $\mu\in\mathcal M_{k+1}(\0,M_1)$ and
$\alpha_{k+1,\mu}(\0,M_1)=0$.  Otherwise,
$\supp(\mu)_0 \cap \wt Q_{C_{\rm loc}M_1}(\0)
\neq L\cap \wt Q_{C_{\rm loc}M_1}(\0)$.  {In this case the family provided by
Lemma~\ref{planestructure} consists of a single vector $e_1$.}
Property~(b) and the failure of the preceding equality give
\[
|e_1|\leq C_{\rm loc}M_1\sqrt n.
\]
Apply Lemma \ref{oneorthogonal} to the measure
$$
\widetilde\mu=(C_{\rm loc}M_1)^{-(k+2)}
(\delta_{(C_{\rm loc}M_1)^{-1}})_\#\mu
$$
and to $\widehat e_1=(C_{\rm loc}M_1)^{-1}e_1$.  We obtain
$$
\alpha_{k+2,\widetilde\mu}(\0,M)<\lambda
\qquad\text{for every }M\geq M_1.
$$
Scaling back gives
$$
\alpha_{k+2,\mu}(\0,C_{\rm loc}MM_1)<\lambda
\qquad\text{for every }M\geq M_1.
$$

Finally, suppose $\dim L=k+2$.  Then $\supp(\mu)_0=L$.  For every
$r>0$, Lemma~\ref{planeallslices}, applied at a sufficiently large
radius, gives a measure $\sigma_r\in\mathcal M_{k+2}$ agreeing with
$\mu$ on $B_{3r}(\0)$.  Hence
\[
\mu\in\mathcal M_{k+2}(\0,r),
\qquad
\alpha_{k+2,\mu}(\0,r)=0.
\]

Choose $M_2^\ast\geq\max\{M_2,C_{\rm loc}M_1^2\}$.  Then, in every case, either
$$
\alpha_{k+1,\mu}(\0,M_1)<\lambda
\quad\text{or}\quad
\alpha_{k+2,\mu}(\0,M_2^\ast)<\lambda,
$$
and the proposition is proved, after replacing $M_2$ by $M_2^\ast$ in the statement.
\end{proof}

\subsection{The general $\eta$ case}\label{generaletavampsec}

We now return to the fixed skew-symmetric form $\eta$ on $\R^n$.  We write $\pi_H(X,t)=X$.

\begin{lem}[Reflections and intrinsic translations]\label{higherHreflectionlem}
Let $\mu$ be a vampiric measure on $\mathbb G_\eta$, and set
$E=\supp\mu$.  If $\X=(X,t)\in E$, define
\[
\Sigma_\X(Z,v)=\bigl(2X-Z,v-2\eta(X,Z)\bigr).
\]
Then, for every bounded Borel set $A\subset\mathbb G_\eta$,
\[
\mu(\Sigma_\X(A))=\mu(A).
\]
If also $\Y=(Y,u)\in E$, put
\[
g_{\X,\Y}=\bigl(2(X-Y),-4\eta(X,Y)\bigr).
\]
{For an integer $\ell>0$, let $g_{\X,\Y}^{\ell}$ be the product of $\ell$
copies of $g_{\X,\Y}$ under the group operation, and set
$g_{\X,\Y}^{0}=\0$ and
$g_{\X,\Y}^{-\ell}=(g_{\X,\Y}^{-1})^\ell$.
Since $\eta(X-Y,X-Y)=0$, this gives
\[
g_{\X,\Y}^{\ell}
=\bigl(2\ell(X-Y),-4\ell\eta(X,Y)\bigr),
\qquad \ell\in\mathbb Z.
\]}
Then, for every $\ell\in\mathbb Z$ and every bounded Borel set
$A\subset\mathbb G_\eta$,
\[
\mu(g_{\X,\Y}^{\ell}\cdot A)=\mu(A).
\]
In particular, for every $\Z\in\mathbb G_\eta$, $r>0$, and
$\ell\in\mathbb Z$,
{\begin{equation}\label{reflectionballidentities}
\mu\bigl(B_r(\Sigma_\X(\Z))\bigr)=\mu(B_r(\Z)),
\qquad
\mu\bigl(B_r(g_{\X,\Y}^{\ell}\cdot\Z)\bigr)=\mu(B_r(\Z)).
\end{equation}}
Consequently, if $\Z\in E$, then $\Sigma_\X(\Z)$ and
$g_{\X,\Y}^{\ell}\cdot\Z$ belong to $E$.
\end{lem}

\begin{proof}
Fix $\X\in E$.  If $A$ is bounded and open, put
\[
U=\{\mathbf Q^{-1}\cdot\X:\mathbf Q\in A\}.
\]
Then $U$ is bounded and open, and
\[
A=\X\cdot\{\mathbf Q^{-1}:\mathbf Q\in U\},
\qquad
\Sigma_\X(A)=\X\cdot
\{\overline{\mathbf Q}^{-1}:\mathbf Q\in U\}.
\]
Thus \eqref{vampcor} gives $\mu(\Sigma_\X(A))=\mu(A)$.  By regularity,
the same identity holds for every bounded Borel set $A$.

A direct calculation gives
\[
\Sigma_\X(\Z_1)^{-1}\cdot\Sigma_\X(\Z_2)
=\overline{\Z_1^{-1}\cdot\Z_2}.
\]
For $\X,\Y\in E$, another direct calculation gives
\[
\Sigma_\X\circ\Sigma_\Y=L_{g_{\X,\Y}}.
\]
Applying the preceding Borel-set identity first with $\Sigma_\Y$ and
then with $\Sigma_\X$ shows that
\[
\mu(g_{\X,\Y}\cdot A)=\mu(A).
\]
Iteration, with the order of $\X$ and $\Y$ reversed for negative powers,
gives the asserted Borel-set identity for every $\ell\in\mathbb Z$.
In particular, each left translation $L_{g_{\X,\Y}^{\ell}}$ preserves
$\mu$.

Since spatial reflection preserves the norm, the displayed formula for
$\Sigma_\X(\Z_1)^{-1}\cdot\Sigma_\X(\Z_2)$ shows that $\Sigma_\X$ is an
isometry.  Left translations are also isometries.  Therefore
\[
\Sigma_\X(B_r(\Z))=B_r(\Sigma_\X(\Z)),
\qquad
g_{\X,\Y}^{\ell}\cdot B_r(\Z)
=B_r(g_{\X,\Y}^{\ell}\cdot\Z).
\]
{Applying $\mu(\Sigma_\X(A))=\mu(A)$ and
$\mu(g_{\X,\Y}^{\ell}\cdot A)=\mu(A)$ with $A=B_r(\Z)$
gives \eqref{reflectionballidentities}.}  {If $\Z\in E$, then $\mu(B_r(\Z))>0$ for every $r>0$.
By \eqref{reflectionballidentities}, every ball centered at
$\Sigma_\X(\Z)$ or $g_{\X,\Y}^{\ell}\cdot\Z$ also has positive measure.
Hence both points belong to $E=\supp\mu$.}
\end{proof}

The next lemma is a generalization of analysis carried out by Orponen \cite{Orp} in the $1$-Heisenberg group.

\begin{lem}[Possible dimensions of the horizontal span]\label{higherHspanaltlem}
Let $0\leq k\leq n-1$, and let $\mu$ be a $(k+2)$-ADR vampiric
measure on $\mathbb G_\eta$ with $\0\in\supp\mu$.  Set
$E=\supp\mu$, $A=\pi_H(E)$, and
\[
S=\operatorname{span} A,
\qquad d=\dim S.
\]
Then
\[
k\leq d\leq k+2.
\]
Moreover, if $d>k$, then $S$ is $\eta$-isotropic.
\end{lem}

\begin{proof}
Suppose $d<k$.  We use that $E\subset S\times\R$.  Fix
$0<R<\diam E$.  For $0<\rho<R$, the set
$(S\times\R)\cap B_R(\0)$ can be covered by
$\lesssim(R/\rho)^{d+2}$ balls of radius $\rho$.  By
\eqref{globalADRupper},
\[
\mu(B_R(\0))
\lesssim (R/\rho)^{d+2}\rho^{k+2}
=R^{d+2}\rho^{k-d}.
\]
Letting $\rho\downarrow0$ contradicts the lower $(k+2)$-ADR bound.  Thus $d\geq k$.

If $d=0$, then $k=0$ and the remaining assertions are immediate.  We
may therefore assume that $d>0$.

Choose points $\X_i=(X_i,t_i)\in E$, $i=1,\ldots,d$, so that
$X_1,\ldots,X_d$ form a basis of $S$.  Lemma~\ref{higherHreflectionlem},
applied to $\X_i$ and $\0$, shows that left translation by
\[
g_i=(2X_i,0)
\]
preserves $\mu$.  For $a=(a_1,\ldots,a_d)\in\mathbb Z^d$, put
\[
h_a=g_1^{a_1}\cdots g_d^{a_d}.
\]
Then $L_{h_a}$ preserves $\mu$, and $h_a\in E$ because $\0\in E$.
In particular, $g_1^a\in E$ for every $a\in\mathbb Z$, so
$\diam E=\infty$.

 There are constants $c,C,\delta_0>0$ such that,
for all sufficiently large $R$, the points
\[
\{h_a:|a_i|\leq cR,\ i=1,\ldots,d\}
\]
lie in $B_{CR}(\0)$, are pairwise $\delta_0$-separated\footnote{This separation is a straightforward consequence of linear independence of $X_i$.}, and number
$\gtrsim R^d$.  The balls of radius $\delta_0/3$ centered at these
points are therefore pairwise disjoint and, for a fixed $C'>C$,
contained in $B_{C'R}(\0)$.
If $d>k+2$, the ADR bounds give
\[
\mu(B_{C'R}(\0))\gtrsim R^d,
\]
contradicting the upper $(k+2)$-ADR bound for large $R$.  Hence
$d\leq k+2$.

It remains to prove the isotropic assertion.  Suppose $d>k$ and $S$ is
not $\eta$-isotropic.  Then $\eta(X_i,X_j)\neq0$ for some $i,j$.  The
commutator of $g_i$ and $g_j$ is the nontrivial vertical translation
\[
[g_i,g_j]=(0,8\eta(X_i,X_j))=:(0,\tau),
\qquad \tau\neq0,
\]
and left translation by $(0,\tau)$ also preserves $\mu$.  After
decreasing $c$ and enlarging $C$ if necessary, for $|a_i|\leq cR$ and
$|b|\leq cR^2$, the points
\[
h_a\cdot(0,\tau)^b
\]
lie in $B_{CR}(\0)$ and belong to $E$.  They are pairwise separated by
\[
\delta_1=\min\{\delta_0,\|(0,\tau)\|\}>0.
\]
Thus the balls of radius $\delta_1/3$ centered at them are pairwise
disjoint and contained in $B_{C'R}(\0)$.  Since there are
$\gtrsim R^{d+2}$ such centers, the lower ADR bound gives
\[
\mu(B_{C'R}(\0))\gtrsim R^{d+2},
\]
contradicting the upper $(k+2)$-ADR bound when $d>k$.  Therefore $S$ is
$\eta$-isotropic whenever $d>k$.
\end{proof}

\begin{lem}[Horizontal projection]\label{etaprojectionlem}
Let $W\subset\mathbb R^n$ be a $d$-dimensional linear subspace, put
$P=W\times\mathbb R$, and let $\mu$ be a $(d+2)$-ADR measure with
$E=\supp\mu\subset P$.  Write $B_W(x,\rho)$ for the Euclidean ball in
$W$ centered at $x$ with radius $\rho$.  If $A=\pi_H(E)$, then, for every
$\X=(X,t)\in E$ and every $0<r<\diam E$,
\[
\mathcal H^d(A\cap B_W(X,r/2))\gtrsim r^d.
\]
Consequently, $A-A$ contains a neighborhood of $0$ in $W$.
\end{lem}

\begin{proof}
Replacing $\mu$ by $(L_{\X^{-1}})_\#\mu$ replaces $A$ by $A-X$ and
maps $P$ onto itself.  Thus we may assume that $\X=\0$.  The case
$d=0$ is immediate, so suppose that $d\geq1$.  Fix
$0<\delta<r/12$.  Let
\[
A\cap B_W(0,r/2)\subset\bigcup_iB_W(y_i,s_i)
\]
be an arbitrary cover with $s_i<\delta$.  Discard the balls which do not
meet the set on the left.  For every remaining ball choose
\[
x_i\in A\cap B_W(0,r/2)\cap B_W(y_i,s_i)
\]
and put $\rho_i=2s_i$.  The balls $B_W(x_i,\rho_i)$ still cover the set,
and $\rho_i<r/6$.  Set
\[
F_i
:=
E\cap B_{r/2}(\0)\cap
\pi_H^{-1}\bigl(B_W(x_i,\rho_i)\bigr).
\]
These sets cover $E\cap B_{r/2}(\0)$: if $(z,u)$ belongs to this set,
{then $z\in A\cap B_W(0,r/2)$. Choose $i$ such that
$z\in B_W(y_i,s_i)$. Since $x_i\in B_W(y_i,s_i)$ as well,
\[
|z-x_i|\leq |z-y_i|+|y_i-x_i|<2s_i=\rho_i.
\]
Thus $(z,u)\in F_i$.}  If $(z,u)\in F_i$, then
\[
|z-x_i|<\rho_i,
\qquad
|u-\eta(x_i,z)|
=|u-\eta(x_i,z-x_i)|
\leq \frac{r^2}{16}+\frac14r\rho_i
<r^2.
\]
Thus $u-\eta(x_i,z)\in[-r^2,r^2]$.  Cover this interval by
$N_i\lesssim(r/\rho_i)^2$
intervals of radius $\rho_i^2$, with centers $\tau_{ij}$.  Then
\eqref{explicitetadistance} gives
\[
d\bigl((z,u),(x_i,\tau_{ij})\bigr)^4
<\rho_i^4+16\rho_i^4<(3\rho_i)^4
\]
for some $j$.  Thus $F_i$ is covered by the $N_i$ metric balls of radius
$3\rho_i$ centered at $(x_i,\tau_{ij})$.

Discard any of these balls that do not meet $F_i$.  Recenter each
remaining ball at a point of $F_i$ and double its radius.  The resulting
balls are centered on $E$, have radii $6\rho_i<r<\diam E$, and cover
$E\cap B_{r/2}(\0)$.  The ADR bounds now give
\[
r^{d+2}
\lesssim
\mu(B_{r/2}(\0))
\lesssim
\sum_i\left(\frac r{\rho_i}\right)^2\rho_i^{d+2}
=2^dr^2\sum_i s_i^d.
\]
Hence $\sum_i s_i^d\gtrsim r^d$.  Taking the infimum over all such
covers and then letting $\delta\downarrow0$ proves the asserted
Hausdorff-measure estimate.

Finally,
\[
A=\bigcup_{R=1}^\infty
\pi_H\bigl(E\cap\overline{B_R(\0)}\bigr),
\]
and every set on the right is compact.  {Since $\mathcal H^d(A)>0$, there exists $R\geq1$ such that
$K:=\pi_H(E\cap\overline{B_R(\0)})$ has positive
$d$-dimensional Lebesgue measure in $W$.
By Steinhaus' theorem, $K-K$ contains a neighborhood of $0$ in $W$.
Since $K\subset A$, so does $A-A$.}
\end{proof}

\begin{lem}[Vertical-plane classification]\label{etaverticalplanelem}
Let $W\subset\mathbb R^n$ be a $k$-dimensional subspace, put
$P=W\times\mathbb R$, and write
$m_P=\mathcal H^k\vert_W\times\mathcal L^1$.  Suppose that $\mu$ is a
$(k+2)$-ADR vampiric measure on $\mathbb G_\eta$ with
$\supp\mu\subset gP$ for some $g\in\mathbb G_\eta$.
\begin{enumerate}
\item If $\eta(W,W)\neq\{0\}$, then
$\mu=c(L_g)_\#m_P$ for some $c>0$.
\item If $\eta(W,W)=\{0\}$, then
\[
(L_{g^{-1}})_\#\mu=\mathcal H^k\vert_W\times\nu
\]
for some Radon measure $\nu$ on $(\mathbb R,d_t)$.
\end{enumerate}
\end{lem}

\begin{proof}
Replacing $\mu$ by $(L_{g^{-1}})_\#\mu$, assume that $g=\0$ and
$\supp\mu\subset P$.  Let $E=\supp\mu$, $A=\pi_H(E)$, and
\[
K=\{h\in P:\mu(h\cdot B)=\mu(B)
\text{ for every bounded Borel set }B\subset P\}.
\]
Observe that $K$ is a closed subgroup of $P$.  

By Lemma~\ref{etaprojectionlem}, $A-A$ contains a neighborhood of $0$
in $W$.  Given $x,y\in A$, choose $(x,t),(y,u)\in E$.
Lemma~\ref{higherHreflectionlem} gives
\[
(2(x-y),-4\eta(x,y))\in K.
\]
Thus $\pi_H(K)$ contains the neighborhood $2(A-A)$ of $0$ in $W$.
Since $\pi_H(K)$ is an additive subgroup of $W$, it follows that
$\pi_H(K)=W$.

Assume first that $\eta(W,W)\neq\{0\}$.  Choose $w_0,w_0'\in W$ with
$\eta(w_0,w_0')\neq0$.  As $a$ ranges over $\mathbb R$, the values
$\eta(aw_0,w_0')$ fill $\mathbb R$.  For arbitrary $w,w'\in W$, choose
points $(w,\tau),(w',\tau')\in K$.  Their commutator is
\[
[(w,\tau),(w',\tau')]=(0,2\eta(w,w'))\in K.
\]
Thus $\{0\}\times\mathbb R\subset K$.  Together with $\pi_H(K)=W$,
this gives $K=P$.  Hence $\mu$ is a nonzero left-invariant Radon measure
on $P$, and therefore $\mu=c m_P$ for some $c>0$.

Assume next that $\eta(W,W)=\{0\}$.  Then the restricted group law on
$P$ is the product law
\[
(x,t)(y,s)=(x+y,t+s),\qquad (x,t),(y,s)\in P.
\]
For $x,y\in A$, the preceding reflection argument gives
$(2(x-y),0)\in K$.  Since $A-A$ contains a neighborhood of $0$ in $W$,
the subgroup $K\cap(W\times\{0\})$ contains a neighborhood of $0$ in
$W\times\{0\}$.  Therefore
\[
W\times\{0\}\subset K.
\]
{As in the proof of Lemma~\ref{planeallslices}, invariance under
translations by $(w,0)$, $w\in W$, gives
\[
\mu=\mathcal H^k\vert_W\times\nu
\]
for some Radon measure $\nu$ on $\R$.}
\end{proof}

\begin{lem}[Classification when $\dim S=k$]\label{higherHdimSlem}
Let $0\leq k\leq n-1$, and let $\mu$ be a $(k+2)$-ADR vampiric measure on $\mathbb G_\eta$ with $\0\in\supp\mu$.  Let $S=\operatorname{span}\pi_H(\supp\mu)$.  If $\dim S=k$, then $\mu\in\mathcal M_k$.
\end{lem}

\begin{proof}
The support of $\mu$ is contained in $P=S\times\R$.
Lemma~\ref{etaverticalplanelem} gives either a positive multiple of Haar
measure on $P$, when $\eta(S,S)\neq\{0\}$, or a product measure
$\mathcal H^k\vert_S\times\nu$, when $S$ is $\eta$-isotropic.  In either
case, \eqref{commonflatclass2}--\eqref{flatshiftinvariance} gives
$\mu\in\mathcal M_k$.
\end{proof}

\begin{lem}\label{etaisotropicreductionlem}
Let \(S\subset\R^n\) be $\eta$-isotropic, and suppose that \(\mu\) is a
vampiric measure on \(\mathbb G_\eta\) satisfying
\[
    \supp\mu\subset S\times\R.
\]
Then, for every \(p\in\supp\mu\) and every compactly supported Lipschitz
function \(\psi\) on \(S\times\R\) which is odd in the \(S\)-variable,
\[
    \int_{S\times\R} \psi(q^{-1}p)\,d\mu(q)=0,
\]
where \(q^{-1}p\) is computed using the product law on \(S\times\R\).
\end{lem}

\begin{proof}
The group law on $S\times\R$ is the product law because
$\eta|_{S\times S}=0$, and the induced metric is exactly the product
metric.

Let $\psi$ be Lipschitz on $S \times \R$, compactly supported, and odd in
the $S$ variable.  By the McShane extension theorem, it has a Lipschitz
extension $F$ to $\mathbb G_\eta$.  Put
\[
F_{\rm odd}(\X)=\frac12\bigl(F(\X)-F(\overline\X)\bigr).
\]
Since $\X\mapsto\overline\X$ is an isometry, $F_{\rm odd}$ is Lipschitz and
horizontally odd, and $F_{\rm odd}=\psi$ on $S\times\R$.  Choose a compactly
supported horizontally even Lipschitz cutoff $\chi$ which equals one on a
ball containing $\supp\psi$, and set $\Psi=\chi F_{\rm odd}$.  Then $\Psi$ is
compactly supported, Lipschitz, horizontally odd, and
$\Psi\vert_{S\times\R}=\psi$.  If $p,q\in S\times\R$, then $q^{-1}\cdot p \in S\times\R$.  Therefore, for
every $p\in\supp\mu$,
\[
0=\int \Psi(q^{-1}\cdot p)\,d\mu(q)
=
\int \psi(q^{-1}\cdot p)\,d\mu(q),
\]
where the expression on the right is the product convolution inside
$S\times\R$.  The left-hand side is zero by the $\mathbb G_\eta$
vampiric condition.
\end{proof}

\begin{prop}[General vampiric alternative]\label{vampHSDCalternative}
Let $0\leq k\leq n-1$ and let $\lambda>0$.  There exist
$M_1,M_2\geq2$, depending only on $n$, $k$, $\lambda$, and $C_0$, such
that the following holds.  Suppose that $\mu$ is a $C_0$-ADR vampiric
measure of dimension $k+2$ on $\mathbb G_\eta$, $\0\in\supp\mu$,
$r>0$, and
\begin{equation}\label{etanotverticalmodelunit}
\alpha_{k,\mu}(\0,r)>0.
\end{equation}
Then
\begin{equation}\label{etaHSDCalt}
\min\left\{
\alpha_{k+1,\mu}(\0,M_1r),
\alpha_{k+2,\mu}(\0,M_2r)
\right\}<\lambda,
\end{equation}
with the convention that $\alpha_{m,\mu}\equiv+\infty$ when
$\mathcal M_m=\varnothing$.
\end{prop}

\begin{proof}
{By dilation invariance, we may assume without loss of generality that $r=1$.} Let $0\leq k\leq n-1$ and let $\lambda>0$.  Select $M_1$ and $M_2$ as in  Proposition~\ref{vampprop}.  

Let $E=\supp\mu$, $A=\pi_H(E)$, and
$S=\operatorname{span}A$.  Put $d=\dim S$.
Lemma~\ref{higherHspanaltlem} gives $d\in\{k,k+1,k+2\}$.  If $d=k$, then
Lemma~\ref{higherHdimSlem} gives $\mu\in\mathcal M_k$.  Since $\mu$ is
$C_0$-ADR, it belongs to $\mathcal M_k(\0,1)$, so
$\alpha_{k,\mu}(\0,1)=0$, contradicting
\eqref{etanotverticalmodelunit}.  Hence $d>k$, and
Lemma~\ref{higherHspanaltlem} gives that $S$ is $\eta$-isotropic.

By Lemma~\ref{etaisotropicreductionlem}, $\mu$ is vampiric for the
product structure on $S\times\R$.  Write $\alpha^S_{j,\mu}$ for the
corresponding $\alpha$-coefficients defined on $S$.  For every $R>0$ and
$j\in\{k,k+1,k+2\}$,
\begin{equation}\label{Salpha}
\alpha_{j,\mu}(\0,R)\leq\alpha^S_{j,\mu}(\0,R).
\end{equation}
Indeed, 
\begin{itemize}
    \item the parabolic metric on $S\times \R$ agrees with the restriction of $d_{\eta}$ to $S\times\R$,
    \item every
 $1$-Lipschitz test function in $\R^{n+1}$ restricts to a 
$1$-Lipschitz test function on $S\times \R$, and 
\item every $C_0$-ADR measure on {$S\times\R$} extends to a $C_0$-ADR measure on {$\R^{n+1}$} (by just extending by $0$).
\end{itemize}

Taking $j=k$ and $R=1$, the hypothesis \eqref{etanotverticalmodelunit} gives
$\alpha^S_{k,\mu}(\0,1)>0$.  Proposition~\ref{vampprop} 
applies on $S\times\R$, and yields that
\[
    \alpha^S_{k+1,\mu}(\0,M_1)<\lambda
    \qquad\text{or}\qquad
    \alpha^S_{k+2,\mu}(\0,M_2)<\lambda,
\]
and so (\ref{Salpha})  gives \eqref{etaHSDCalt}.
\end{proof}

\begin{cor}\label{lowcodimHvampcor}
Let $I_\eta$ be the maximal $\eta$-isotropic dimension from
\eqref{maximalisotropicdimension}.  If $I_\eta\leq k\leq n-1$ and
$\mu$ is a $(k+2)$-ADR vampiric measure on $\mathbb G_\eta$, then
$\mu\in\mathcal M_k$.  In particular, if $k>I_\eta$, then $\mu$ is a
positive multiple of a left translate of Haar measure on a 
vertical plane.  When
$k=I_\eta$, $\eta$-isotropic product measures in $\mathcal M_k$ may
also occur.
\end{cor}

\begin{proof}
After a left translation, assume $\0\in\supp\mu$.  Let
$S=\operatorname{span}\pi_H(\supp\mu)$ and $d=\dim S$.
Lemma~\ref{higherHspanaltlem} gives $k\leq d\leq k+2$, and if $d>k$,
then $S$ is $\eta$-isotropic.  Since $d>k\geq I_\eta$ would contradict
the maximality of $I_\eta$, we must have $d=k$.
Lemma~\ref{higherHdimSlem} gives $\mu\in\mathcal M_k$.
{If $S$ is not $\eta$-isotropic, Lemma~\ref{etaverticalplanelem}
gives $\mu=c(\mathcal H^k\vert_S\times\mathcal L^1)$ for some $c>0$.
If $S$ is $\eta$-isotropic, then $k=I_\eta$, and the same lemma
gives $\mu=\mathcal H^k\vert_S\times\nu$ for some Radon measure $\nu$.}
\end{proof}

For $\eta=0$ one has $I_\eta=n$, so the corollary gives no additional range for $k\leq n-1$.  For $\eta=\frac12\omega$ on $\R^{2N}$ one has $I_\eta=N$, recovering the Heisenberg threshold.

\section{Packing bad cubes}

Throughout this section fix $0\leq k\leq n-1$.  We retain the convention that
\[
\alpha_{m,\mu}(\X,r)=+\infty
\qquad\text{whenever }\mathcal M_m(\X,r)=\varnothing.
\]
In particular, $\alpha_{m,\mu}\equiv+\infty$ when
$\mathcal M_m=\varnothing$.

Recall the scales $M_1(\lambda),M_2(\lambda)\geq2$ furnished by
Proposition~\ref{vampHSDCalternative}.

\begin{prop}\label{lipcoefnonflat}
For every $\kap,\lambda>0$, there are $\Delta>0$, $A>1$, and a finite
collection
\[
\mathcal F\subset \Lipodd(\R^{n+1})\cap C^\infty(\R^{n+1})
\]
such that the following holds.  If $\mu$ is $(k+2)$-ADR and
$E\in\dy(\mu)$ satisfies
\begin{equation}\label{badcond}
\begin{aligned}
\alpha_{k,\mu}(\Z_E,\ell(E))&\geq\kap,\\
\alpha_{k+1,\mu}(\Z_E,M_1(\lambda)\ell(E))&\geq\lambda,\\
\alpha_{k+2,\mu}(\Z_E,M_2(\lambda)\ell(E))&\geq\lambda,
\end{aligned}
\end{equation}
then
\[
\max_{\varphi\in\mathcal F}\Theta_{\mu,\varphi,A}(E)
\geq \Delta\mu(E).
\]
\end{prop}

We shall use the scale-change lemma in the following form.

\begin{lem}\label{hdsc1}
Suppose that $\mu$ is $(k+2)$-ADR and that the $\eps$-$\HSDC$ cubes are
Carleson for some $\eps>0$.  Set $\lambda=c_\alpha\eps$, where $c_\alpha$
is the constant in Lemma~\ref{alphaenlarge}.  Then, for every fixed
$M_1,M_2\geq2$, the cubes satisfying
\[
\alpha_{k+1,\mu}(\Z_E,M_1\ell(E))<\lambda
\quad\text{or}\quad
\alpha_{k+2,\mu}(\Z_E,M_2\ell(E))<\lambda
\]
form a Carleson family.  Its Carleson norm may depend on $M_1$ and $M_2$.
\end{lem}

\begin{proof}
For $i\in\{1,2\}$, the cubes satisfying
\[
\alpha_{k+i,\mu}(\Z_E,\ell(E))<\eps
\]
form a subfamily of the $\eps$-$\HSDC$ cubes.  Apply the first assertion of
Lemma~\ref{alphaenlarge} to these two families, with $M=M_i$, and take their
union.
\end{proof}

\begin{proof}[{Proof of Theorem~\ref{generaletathm}}]
{We assume Proposition~\ref{lipcoefnonflat}.}
{If the continuous $\eps$-$\HSDC$ hypothesis is assumed, Lemma~\ref{continuousdyadicHSDC} yields the dyadic condition with constant $\eps'=2^{-(k+4)}\eps$. We therefore use the dyadic formulation and relabel its constant to be $\eps$.}

Fix $\kap>0$, set $\lambda=c_\alpha\eps$, and then choose
$M_i=M_i(\lambda)$ from Proposition~\ref{vampHSDCalternative}.  Let
$\mathcal Q$ be the cubes satisfying \eqref{badcond}.  Proposition
\ref{lipcoefnonflat} and Lemma~\ref{findCarleson} show that $\mathcal Q$ is
Carleson.  By Lemma~\ref{hdsc1}, the family $\mathcal R$ of cubes satisfying
\[
\alpha_{k+1,\mu}(\Z_E,M_1\ell(E))<\lambda
\quad\text{or}\quad
\alpha_{k+2,\mu}(\Z_E,M_2\ell(E))<\lambda
\]
is also Carleson.  Moreover,
\[
\{E\in\dy(\mu):\alpha_{k,\mu}(\Z_E,\ell(E))>\kap\}
\subset \mathcal Q\cup\mathcal R.
\]
The family on the left is therefore Carleson for every $\kap>0$.
Theorem~\ref{modelstepthm} now gives the desired vertical $\beta$-Carleson estimate.
\end{proof}

\begin{proof}[Proof of Proposition~\ref{lipcoefnonflat}]
Fix $\kap,\lambda>0$.  Choose a countable dense collection
$\{\varphi_i\}_{i=1}^\infty$ in $\Lipodd(\R^{n+1})$, with each
$\varphi_i$ smooth and compactly supported.  Suppose that the proposition
fails.  For each $j\geq2$, apply its negation with
\[
\Delta=j^{-1},
\qquad A=j,
\qquad \mathcal F_j=\{\varphi_1,\ldots,\varphi_j\}.
\]
There are a $(k+2)$-ADR measure $\widetilde\mu_j$ and a cube
$\widetilde E_j\in\dy(\widetilde\mu_j)$ satisfying \eqref{badcond}, but
\begin{equation}\label{smallnormalizedsquare}
\max_{1\leq i\leq j}
\Theta_{\widetilde\mu_j,\varphi_i,j}(\widetilde E_j)
<\frac1j\widetilde\mu_j(\widetilde E_j)
\lesssim \frac1j\,\ell(\widetilde E_j)^{k+2}.
\end{equation}

For a dyadic cube $E$, recall
\[
T_E(p)=\delta_{1/\ell(E)}(\Z_E^{-1}\cdot p).
\]
Put $\ell_j=\ell(\widetilde E_j)$, $T_j=T_{\widetilde E_j}$, and
\[
\mu_j
=\ell_j^{-(k+2)}(T_j)_\#\widetilde\mu_j.
\]
We equip $\mu_j$ with the image of the dyadic grid of
$\widetilde\mu_j$, and put $E_j=T_j(\widetilde E_j)$.  Then
\[
\Z_{E_j}=\0,\qquad \ell(E_j)=1,\qquad
E_j\subset B_{A_{\dy}}(\0),\qquad \mu_j(E_j)\approx1.
\]
In particular,
\[
\mu_j(B_{A_{\dy}}(\0))\geq\mu_j(E_j)\gtrsim1.
\]
Also observe that
\begin{equation}\label{normalizedbadcond}
\begin{aligned}
\alpha_{k,\mu_j}(\0,1)&\geq\kap,\\
\alpha_{k+1,\mu_j}(\0,M_1(\lambda))&\geq\lambda,\\
\alpha_{k+2,\mu_j}(\0,M_2(\lambda))&\geq\lambda.
\end{aligned}
\end{equation}
For every $\varphi$ and $\Y\in\mathbb G_\eta$,
\[
(\varphi*\mu_j)(T_j\Y)
=\ell_j^{-(k+2)}
(\varphi_{\ell_j}*\widetilde\mu_j)(\Y).
\]
Consequently,
\[
\Theta_{\mu_j,\varphi,j}(E_j)
=\ell_j^{-(k+2)}
\Theta_{\widetilde\mu_j,\varphi,j}(\widetilde E_j).
\]
Thus \eqref{smallnormalizedsquare} gives
\[
\max_{1\leq i\leq j}\Theta_{\mu_j,\varphi_i,j}(E_j)
\lesssim j^{-1}\longrightarrow0.
\]

After passing to a subsequence, Lemma~\ref{adrwl} gives {weak}
convergence to a $C_0$-ADR measure $\mu$ with $\0\in\supp\mu$.
 For each fixed
$i$ and $R>0$, once $j\geq\max\{i,R\}$ we have
$B_R(\0)\subset B_j(\0)$, so the preceding square coefficient tends to
zero on $B_R(\0)$.  Lemma
\ref{squaretovampcompact} therefore shows that $\mu$ is vampiric.

Now, Lemma~\ref{alphaweak} and
\eqref{normalizedbadcond} give
\[
\alpha_{k,\mu}(\0,1)\geq\kap>0.
\]
Therefore, Proposition~\ref{vampHSDCalternative} now yields
\[
\alpha_{k+1,\mu}(\0,M_1(\lambda))<\lambda
\quad\text{or}\quad
\alpha_{k+2,\mu}(\0,M_2(\lambda))<\lambda.
\]
Either way,  Lemma~\ref{alphaweak} is therefore applicable to that coefficient
and gives the same strict inequality for $\mu_j$ when $j$ is large,
contradicting \eqref{normalizedbadcond}.
\end{proof}

This completes the proof of Theorem~\ref{generaletathm}.

\section{Examples}\label{examples}

The purpose of this section is to provide three pathological examples of
$(k+2)$-ADR measures which are not uniformly rectifiable but for which
all admissible CZOs are bounded in $L^2$.  We consider the case when $\eta=0$ (the parabolic metric), which can be {embedded} into the general case by considering $\eta$-isotropic subspaces.

\begin{thm} The following statements hold:\begin{itemize} 
\item[(A)] Suppose that $\nu$ satisfies
\[
\nu\bigl(B_{d_t}(t,\rho)\bigr)\approx\rho
\qquad(t\in\supp\nu,\ \rho>0),
\]
and $L\subset\R^n$ is a $(k+1)$-plane.  Equivalently, $\nu$ is
$1$-ADR on $(\R,d_t)$ and has infinite diameter.  Then
$$\mu = \mathcal{H}^{k+1}\vert_L\times \nu$$
is a $(k+2)$-ADR measure, is not uniformly rectifiable, and all CZOs are bounded in $L^2(\mu)$.
\item[(B)] Suppose $\nu$ is a measure on $(\R,d_t)$ satisfying
\[
\nu\left(\left(t-\frac{\rho^2}{4},t+\frac{\rho^2}{4}\right)\right)
\approx
\min(\rho,\rho^2)
\quad\text{for every }t\in\supp(\nu)\text{ and }\rho>0,
\]
where $\rho$ denotes radius for the metric $d_t$ on the time line.
Suppose $L$ is a $k$-plane in $\R^n$ and $e$ is a non-zero normal vector to $L$, then
$$\mu=\Bigl[\sum_{j\in\mathbb{Z}}\mathcal{H}^{k}\vert_{L+je}\Bigl]\times\nu$$
is a $(k+2)$-ADR measure, is not uniformly rectifiable, and all CZOs are bounded in $L^2(\mu)$.
\item[(C)] Suppose $k+2\leq n$, $L$ is a $(k+2)$-plane in $\R^n$, and $\mu = \mathcal{H}^{k+2}\vert_L\times \delta_0$ (where $\delta_0$ is the Dirac measure on $\R$).  Then $\mu$ is a $(k+2)$-ADR measure, is not uniformly rectifiable, and all CZOs are bounded in $L^2(\mu)$.
\end{itemize}
\end{thm}

We leave it to the reader to verify the measures constructed are not uniformly rectifiable (there exists $\eps>0$ such that for all $\X\in \supp(\mu)$ and sufficiently large $r>0$, $\beta_{{\mu}, k}(\X,{r})\geq \eps$), and also that $\mu$ is vampiric.  We will verify operator boundedness for (A) and (B) below.  Part (C) is the classical
Euclidean result in $\mathbb R^{k+2}$.

To prove this result we shall verify the T(1)-theorem in spaces of homogeneous type by David--Journ\'{e}--Semmes, as presented in the book by Deng--Han \cite{DH}.

In cases (A) and (B), we can write $\mu = \sigma\times \nu$.  Fix a CZO kernel $K$.  Since $K$ is odd in the
horizontal variables, we symmetrize only in those variables; the time
variables are not interchanged.  For $\varphi, \psi\in \Lip_0(\R^{n+1})$,
define
$$\langle T\varphi, \psi\rangle_{\mu} = \frac{1}{2}\iint K(X-Y, t-s)H(X,Y,t,s)d\sigma(Y)d\sigma(X)d\nu(s)d\nu(t),
$$
where 
$$H(X,Y, t,s) = 
\varphi(Y,s)\psi(X,t)- \varphi(X,s)\psi(Y,t).$$

This is the usual antisymmetrization in the $X,Y$ variables, with $s$ and $t$ kept fixed.  Since $\varphi, \psi \in \Lip_0(\R^{n+1})$,
$${|H(X,Y,t,s)|}\lesssim_{\varphi,\psi} |X-Y|\leq \|(X,t)-(Y,s)\|,$$
and therefore the $(k+2)$-ADR property of $\mu$ ensures that the integral defining $\langle T\varphi, \psi\rangle_{\mu}$ is absolutely convergent. 

If $R>0$, $\X_0\in\supp\mu$, and
$\varphi,\psi\in\Lip_0(B_R(\X_0))$ satisfy
$\|\varphi\|_{\Lip}\leq1/R$ and $\|\psi\|_{\Lip}\leq1/R$, then
\[
{|H(X,Y,t,s)|}\lesssim \frac1R\|(X,t)-(Y,s)\|.
\]

The first term in $H$ pairs points already in $B_R(\X_0)$.  For the
second term, the cross-paired points $(X,s)$ and $(Y,t)$ lie in that ball,
so the actual integration points $(X,t)$ and $(Y,s)$ lie in
$B_{CR}(\X_0)$ for a fixed geometric constant $C$.  Consequently,
{
\[
\begin{aligned}
|\langle T(\varphi),\psi\rangle|
&\lesssim\frac1R\iint_{B_{CR}(\X_0)\times B_{CR}(\X_0)}
\frac{d\mu(X,t)\,d\mu(Y,s)}{\|(X,t)-(Y,s)\|^{k+1}}\\
&\lesssim\mu(B_R(\X_0)),
\end{aligned}
\]
}
where the final comparison uses \eqref{adrscalecomparison}.
 Consequently, the weak boundedness property holds (cf. Definition 1.15 of \cite{DH}).

Using the bilinear form, standard calculations (see p.19-20 of Deng--Han) allow one to consider $\langle T(1), \psi\rangle_{\mu}$ where $\psi\in \Lip_0(\R^{n+1})$ with $\mu$-mean zero (and similarly with $T^*(1)$).  By the vampiric property, passing to the limit through smooth annular cutoffs in the vampiric cancellation identity gives
{
\[
\int_{\{\Y:\,\eps<\|\Y^{-1}\X\|<M\}}
K(\Y^{-1}\X)\,d\mu(\Y)=0
\]
for all $\X\in\supp\mu$ and all $\eps,M>0$.
}
The adjoint kernel is again an admissible horizontally antisymmetric kernel,
so the same cancellation applies to $T^*(1)$.  Hence the standard
cancellation argument gives
$$\langle T(1),\psi\rangle =0\text{  and }\langle T^*(1),\psi\rangle =0\text{ for all }\psi\in \Lip_0(\R^{n+1}) \text{ with }\mu\text{-mean }0.$$
The hypotheses of the homogeneous $T(1)$-theorem (Theorem 1.18 of \cite{DH}) have been verified.
Finally, the standard Cotlar inequality for Calder\'on--Zygmund operators
on spaces of homogeneous type (see, for example,
\cite[Chapter~1]{DH}) gives
\[
 \|T_*f\|_{L^2(\mu)}\lesssim
 \|Tf\|_{L^2(\mu)}+\|f\|_{L^2(\mu)},
 \qquad T_*f=\sup_{\epsilon>0}|T_\epsilon f|.
\]
Thus the maximal truncations are bounded in precisely the sense used in
Subsection~\ref{CZOdefs}.

\vspace{.25 in}

\section*{Declarations} \noindent\textbf{Competing interests.} The authors declare that they have no competing financial or non-financial interests relevant to this work.

\medskip\noindent\textbf{Author contributions.} Both authors contributed to the research and writing of the manuscript.\par\medskip\noindent\textbf{AI use.} Following feedback on the initial draft of this paper, we have treated the Heisenberg and parabolic geometries at the same level of detail (the first draft considered only codimension-one kernels in the first Heisenberg group). All significant mathematical components of these arguments are due to the authors and were in place when the first draft was written. During the redrafting phase in May--June 2026, however, ChatGPT 5.5 and 5.6 were used to experiment with ideas for implementing the results in a unified manner, and were subsequently used to proofread and streamline the paper. Some final edits were made with assistance from ChatGPT 5.6 in October 2026. AI was only used for copy editing in the earlier drafts of this paper on arXiv.\par\printbibliography

 \end{document}